\documentclass[graybox,envcountchap,envcountsame,vecarrow]{svmono_nosp}
\DeclareMathAlphabet{\mathcal}{OMS}{cmsy}{m}{n} %reset default math font
\usepackage{amsmath}
\usepackage{amssymb}
\usepackage{dsfont}
\usepackage{mathptmx}
\usepackage{helvet}
\usepackage{courier}
\usepackage{color}
\usepackage[svgnames]{xcolor}
\usepackage{type1cm}         
\usepackage{makeidx}
\usepackage{index}         % 
 \newindex{sym}{sdx}{snd}{Symbol index}
\usepackage{graphicx}        
\usepackage{multicol}        
\usepackage[bottom]{footmisc}

\usepackage[T1]{fontenc}
\usepackage{hyperref}
\usepackage{mathtools}
 \mathtoolsset{mathic=true}
\usepackage{microtype}

\newcommand{\ii}{\mathsf{i}}
\DeclareMathOperator{\im}{im}
\DeclareMathOperator{\Lin}{Lin}
\newcommand{\mom}{\mathrm{mom}}
\newcommand{\sos}{\mathrm{sos}}
\DeclareMathOperator{\Pos}{Pos}
\DeclareMathOperator{\rIm}{Im}
\DeclareMathOperator{\rRe}{Re}
\DeclareMathOperator{\rank}{rank}
\newcommand{\sA}{\mathsf{A}}
\newcommand{\sB}{\mathsf{B}}
\newcommand{\sN}{\mathsf{N}}
\DeclareMathOperator{\supp}{supp}
\DeclareMathOperator{\Sym}{Sym}
\DeclareMathOperator{\TR}{Tr}
\DeclareMathOperator{\Var}{Var}

\makeindex

\begin{document}
\newtheorem{thm}{Theorem}
\newtheorem{dfn}[thm]{Definition}
\newtheorem{lem}{Lemma}
\newtheorem{exa}[thm]{Example}
\newtheorem{prop}[thm]{Proposition}
\newtheorem{cor}[thm]{Corollary}
\newtheorem{rem}[thm]{Remark}
\newcommand{\ux}{\underline{x}}
\newcommand{\uy}{\underline{y}}
\newcommand{\un}{\underline}
\newcommand{\ov}{\overline}
\newcommand{\uz}{\underline{z}}
\newcommand{\ouz}{\overline{\underline{z}}}

\newcommand{\gp}{\mathfrak{p}}
\newcommand{\gP}{\mathfrak{P}}
\newcommand{\gN}{\mathfrak{N}}
\newcommand{\gL}{\mathfrak{L}}
\newcommand{\gq}{\mathfrak{q}}
\newcommand{\gj}{\mathfrak{j}}
\newcommand{\gm}{\mathfrak{m}}
\newcommand{\gk}{\mathfrak{k}}
\newcommand{\gn}{\mathfrak{n}}
\newcommand{\gl}{\mathfrak{l}}
\newcommand{\gx}{\mathfrak{x}}
\newcommand{\gw}{\mathfrak{w}}
\newcommand{\gy}{\mathfrak{y}}
\newcommand{\dP}{\mathds{P}}
\newcommand{\dD}{\mathds{D}}
\newcommand{\dH}{\mathds{H}}
\newcommand{\T}{\mathds{T}}
\newcommand{\C}{\mathds{C}}
\newcommand{\R}{\mathds{R}}
\newcommand{\N}{\mathds{N}}
\newcommand{\K}{\mathds{K}}
\newcommand{\Q}{\mathds{Q}}
\newcommand{\Z}{\mathds{Z}}
\newcommand{\cJ}{\mathcal{J}}
\newcommand{\cA}{\mathcal{A}}
\newcommand{\cB}{\mathcal{B}}
\newcommand{\cC}{\mathcal{C}}
\newcommand{\cE}{\mathcal{E}}
\newcommand{\cI}{\mathcal{I}}
\newcommand{\cK}{\mathcal{K}}
\newcommand{\cL}{\mathcal{L}}
\newcommand{\cF}{\mathcal{F}}
\newcommand{\cP}{\mathcal{P}}
\newcommand{\cG}{\mathcal{G}}
\newcommand{\cH}{\mathcal{H}}
\newcommand{\cN}{\mathcal{N}}
\newcommand{\cM}{\mathcal{M}}
\newcommand{\cQ}{\mathcal{Q}}
\newcommand{\cD}{\mathcal{D}}
\newcommand{\cR}{\mathcal{R}}
\newcommand{\cS}{\mathcal{S}}
\newcommand{\cV}{\mathcal{V}}
\newcommand{\cW}{\mathcal{W}}
\newcommand{\cT}{\mathcal{T}}
\newcommand{\cX}{\mathcal{X}}
\newcommand{\cY}{\mathcal{Y}}
\newcommand{\cU}{\mathcal{U}}
\newcommand{\cZ}{\mathcal{Z}}
\newcommand{\sL}{\sf{L}}
\newcommand{\sP}{\sf{P}}
\newcommand{\sM}{\sf{M}}
\newcommand{\sC}{\sf{C}}
\newcommand{\sQ}{\sf{Q}}
\newcommand{\sS}{\sf{S}}
\newcommand{\sat}{\rm{sat}}
\newcommand{\ext}{\rm{ext}}
\newcommand{\ord}{\rm{ord}}
\newcommand{\rk}{\mathtt{rk}}
\newcommand{\ind}{\mathtt{ind}}
\newcommand{\Int}{\rm{Int}}
\newcommand{\tr}{\rm{tr}}

\newcommand{\secret}[1]{}
\newcommand{\Tr}{\mathtt{Tr}}
\newcommand{\gB}{\mathfrak{B}}
\newcommand{\gs}{\mathfrak{s}}
\newcommand{\hide}[1]{}
\newcommand{\smallwedge}{\mathbin{\mathchoice{\scriptstyle\wedge}{\scriptstyle\wedge}{\scriptscriptstyle\wedge}{\scriptscriptstyle\wedge}}}
\newcommand{\dint}{\displaystyle\int}

\author{Konrad Schm\"udgen}
\title{Ten Lectures on the  Moment Problem}
%\subtitle{-- Monograph --}
\maketitle
%\frontmatter%%%%%%%%%%%%%%%%%%%%%%%%%%%%%%%%%%%%%%%%%%%%%%%%%%%%%%

%

%\include{acronym}

%\mainmatter%%%%%%%%%%%%%%%%%%%%%%%%%%%%%%%%%%%%%%%%%%%%%%%%%%%%%%%%

%\input{Hanoi-Lecture.tex}

\setcounter{tocdepth}{3}

% \bigskip
% 
% 
% \quad\quad\quad \quad\quad\quad {\bf Ten Lectures on the Moment Problem}
% \bigskip
% 
% \quad\quad\quad \quad {\sf Konrad  Schm\"udgen (University of Leipzig)}
% 
% \bigskip

\chapter*{Preface}

If $\mu$ is a positive Borel measure on the line and $k$ is a nonnegative integer,  the number
\begin{align}\label{kmo}
 s_k\equiv s_k(\mu)
 =\int_\R x^k \, d\mu(x)
\end{align}
is called the \textit{$k$-th moment} of $\mu$.
% provided  that the integral (\ref{kmo}) is finite.
For instance, if the measure $\mu$ 
comes from the distribution function $F$ of a random variable $X$, then the expectation value of $X^k$ is just the $k$-the moment,
\begin{align*}
E [X^k]=s_k=\int_{-\infty}^{+\infty} x^kdF(x),
\end{align*}
 and the variance of $X$ is\,
 $\Var(X)=E[(X-E[X])^2]=E[X^2]-E[X]^2= s_2-s_1^2$ (of course, provided  the corresponding  numbers are finite).
 
 The moment problem is the inverse problem of ``finding"   the measure  when the  moments are given. 
 
 In order to apply functional-analytic methods it is convenient to rewrite the moment problem in terms of linear functionals.
For a real sequence $s=(s_n)_{n\in \N_0}$ let $L_s$ denote the linear functional on the polynomial algebra $\R[x]$  defined by $L_s(x^n)=s_n$, $n\in \N_0$. Then, by the linearity of the integral,  (\ref{kmo}) holds for all $k$ if and only if
\begin{align}\label{integraqlLs1}
L_s(p)=\int_{-\infty}^{+\infty} p(x) \,d\mu(x)~ \quad \text{ for }~~ p\in \R[x].
\end{align}
That is, the moment problem asks whether a linear functional  on  $\R[x]$ admits an integral representation (\ref{integraqlLs1}) with  some positive measure $\mu$ on $\R$. This is the simplest version of the moment problem, the so-called \textit{Hamburger moment problem}.

Moments and the moment problem have  natural $d$-dimensional versions.
For a positive measure $\mu$ on $\R^d$ and a multi-index $\alpha=(\alpha_1,\dots,\alpha_d)\in \N_0^d$, 
the $\alpha$-the moments is defined by
\begin{align*}
 s_\alpha(\mu)
 =\int_{\R^d} x_1^{\alpha_1}\cdots x_d^{\alpha_d} \, d\mu(x).
\end{align*}
%Center of mass and moments 
%is finite, the number $s_\alpha\equiv s_\alpha(\mu)$\index[sym]{SAaGalpha@$s_\alpha(\mu)$} is called the \textit{$\alpha$-th moment\index{Moments!  of a %measure} of the measure $\mu$.} If all $\alpha$-th moments exist,  the sequence $(s_\alpha)_{\alpha\in \N_0^d}$ is called
% the \textit{moment sequence} of $\mu$. 
 
 %Moments appear at many places in physics and mathematics. For instance, if $\cK$ is a solid body in $\R^3$ with mass density $m(x_1,x_2,x_3)$, 
%then the number 
%\begin{align*}s_{(0,2,0)}+s_{(0,0,2)} =\int_\cK (x_2^2+x_3^2)\, m(x_1,x_2,x_3)dx_1dx_2dx_3
%\end{align*}
%is the moment of interia of the body  with respect to the $x_1$-axis. 

Let $\cK$ be a  closed subset of $\R^d.$
An important variant of the moment problem, called the \textit{$\cK$-moment problem}, requires that the representing $\mu$ is supported on $\cK$.  
If only a \textit{finite} number of moments are given,  we have the \textit{truncated moment problem}.

Moments were applied  to the study of functions by the Russian mathematicians  P.L. Chebychev (1874) and  
A.A. Markov (1884). However, 
the moment problem   first  appeared in a famous memoir (1894)  of the Dutch mathematician T.J. Stieltjes. He formulated and solved this problem for the half-line and gave the first explicit example of an indeterminate problem. 
%treated this problem for measures supported on the half-line  and developed a far reaching theory. 
The cases of the real line and of bounded intervals were studied only later by H. Hamburger (1920) and F. Hausdorff (1920). 

Because of the simplicity of its formulation it is surprising that the moment problem has  deep interactions   with many branches of mathematics (functional analysis, function theory, real algebraic geometry, spectral theory, optimization, numerical analysis,  convex analysis, harmonic analysis and others)  and a broad range of applications. The AMS volume ``Moments in Mathematics", edited by H.J. Landau (1987),  and the book ``Moments, Positive Polynomials and Their Applications" by J.P. Lasserre (2015) illustrate this in a convincing manner.

\smallskip

The following notes grew out from a series of lectures given at the
\[
\textit{``School on Sums of Squares, Moment Problems, and Polynomial Optimization"},
\]
which was organized from 1.1.2019 till 31.3.2019  at the \[ \textit{Vietnam Institute for Advanced Studies in Mathematics}~~ (VIASM)\]
in Hanoi. I would like to thank my Vietnamese colleague, Prof. Trinh Le Cong, for the kind invitation and his warm hospitality during this visit. The audience consisted of researchers, young scientists, and graduate students. 

The first  three lectures deal with one-dimensional full and truncated moment problems and contain well-known classical material. The topics of the other lectures were chosen according to the themes of the VIASM school. In Lectures 4--7, we investigate the multi-dimensional \textit{full} moment problem with particular emphasis on the interactions with real algebraic geometry. Lecture 7 gives a digression into applications of moment methods and real algebraic geometry to polynomial optimization. The final Lectures 8--10 are devoted to the multidimensional \textit{truncated} moment problem. In Lectures 4--10, I have  covered  selected advanced results as
 well as recent developments. 

 The following  are smoothed and slightly extended versions of the handouts I had given to the audience.
 I have tried to maintain the casual style of the lectures and to illustrate the theory by  many examples.
 The notes are essentially based on  my book\\

\textit{``The Moment Problem''}, Graduate Text in Mathematics, Springer,  2017,\\
 
 \noindent
 which is quoted as [MP] in what follows.
For some results complete proofs are given, while for most of them I referred to [MP].  These notes have no bibliography. The book [MP] contains additional material  and also detailed references for the main results.
Finally, it should be emphasized that these notes are \textit{not}  a  book on the moment problem. It is more appropriate  to think of them as an entrance guide into some  developments of the beautiful  subject.

%\begin{thebibliography}
%\bibitem[AIT02]{antoineit}
%\end{thebibliography}

%\end{document}

%\input{Lecture1a}
\tableofcontents

\makeatletter
\renewcommand{\@chapapp}{Lecture}
\makeatother
%\chapter{ }%

\chapter{Integral representations of  linear functionals}

Abstract:\\
\textit{An integral representation theorem of positive functionals on Choquet's adapted spaces is obtained.
As  applications, Haviland's theorem is derived and existence results for moment problems on intervals are developed.}
\bigskip

All variants of the moment problem deal  with integral representations of certain linear functionals.
We begin  with a rather general setup.

Suppose  $\cX$ is a \textbf{locally compact} topological Hausdorff space.
%and $E$ a vector space of continuous real-valued functions on $X$.
The  Borel algebra $\gB(\cX)$ is the $\sigma$-algebra  generated by the open subsets of $\cX$.
 \begin{definition}
A  \emph{Radon measure} on $\cX$
 is a  measure $\mu:\gB(\cX)\to [0,+\infty]$ such that\, $\mu(K)<\infty$\,  for each compact subset $K$ of $\cX$ and 
\begin{align}\label{innerreg}
\mu(M)=\sup\{\mu(K): K\subseteq M,~ K~\text{ compact}\}\quad \text{for~all}~~M\in \gB(\cX).
\end{align}
\end{definition} 
Thus, 
in our terminology  
Radon measures are always nonnegative! 
%(\ref{innerreg}) is  the \textit{regularity coondition} of   $\mu$.

Closed subsets of  $\R^d$ are  locally compact in the induced topology of $\R^d$.

Further, suppose $E$ is a \textbf{vector space of continuous real-valued functions} on $\cX$.
In a very general and modern form,
the moment problem is the following problem:
\smallskip

 \textit{Given  a linear  functional $L:E\to \R$   and a closed subset $\cK$ of $\cX$, when does there exist a   Radon measure $\mu$ \textbf{supported on $\cK$} such that}
\begin{align}
\label{intreplf}
L(f)=\int_\cX f(x) \, d\mu(x)\quad \text{\textit{for}}~~ f\in E?
\end{align}
In this case,  $L$ is called a \textbf{$\cK$-moment functional} or   a \textbf{moment functional} if $\cK=\cX$.

When we write equations such as (\ref{intreplf}) we always 
mean that  the function $f(x)$ is $\mu$-integrable and its $\mu$-integral is $L(f)$. 

 Let $\{e_j:j\in J\}$ be a basis of the vector space $E$.
 Then, for each real sequence $s=(s_j)_{j\in J}$, there is unique linear functional  $L_s:E\to \R$ given by $L_s(e_j)=s_j$ for $j\in J$; this functional is called the \textit{Riesz functional} associated with $s$.
 Clearly, (\ref{intreplf}) holds for the functional $L=L_s$ if and only if $s_j=\int e_j(x) d\mu(x)$ for all $j\in J$.
 In this case, the numbers $s_j$ are called the \textit{generalized moment} of $\mu$.
\medskip

The most important cases are: $\cX=\R^d$\smallskip

Case  1.  $E=\R_d[\ux]:=\R[x_1,\dots,x_d]$: classical full moment problem.\smallskip

Case  2. $E=\R_d[\ux]_m:=\{p\in \R_d[\ux]:\deg (p)\leq m\}$: 
 truncated moment problem.\medskip

These two special cases 
lead already to an interesting  theory with  nice results and difficult problems!
In both cases we can take a  basis  $x^\alpha=x_1^{\alpha_1}\cdots x_d^{\alpha_d}$, $\alpha=(\alpha_1,\dots,\alpha_d)\in\N_0^d,$ of monomials; then the numbers $\int x^\alpha d\mu(x)$ are the \textit{moments} of  the measure $\mu$.

\smallskip

Why \textit{positive} measures and not signed measures?
It can be shown that in both Cases 1 and 2  \textit{each} linear functional can be represented by some signed measure.
\smallskip

%Let $E_+=\{f\in :f(x)\geq 0, x\in \cX\}$.
There is an  obvious obstruction for a functional to be a moment functional. Set 
\begin{align*}E_+:=\{f\in E :f(x)\geq 0~~\text{ for }~~ x\in \cX\}.
\end{align*} Since integration of nonnegative functions by  positive measures gives  nonnegative numbers, each moment functional $L$ is $E_+$-positive, that is, $L(f)\geq 0$ for  all $f\in E_+$. In Case 1, the $E_+$-positivity is also sufficient (by Haviland's theorem below), in Case 2 it is not in general. Even if the $E_+$-positivity is sufficient, one would need a description of  $E_+$, which can be very difficult.  We will return to this problem later.
% Most existence results describe situations where positivity conditions combined with additional assumptions are also sufficient for the  existence of a solution.  

\section{Positive linear functionals on adapted spaces} 

The  representation theorem proved below is based on the following notion, invented by G.~Choquet, of an adapted space.

\begin{definition}\label{defadapted} A  linear subspace $E$ of  $C(\cX;\R)$ is called \emph{adapted}\index{Adapted space} if the following conditions are satisfied:
\begin{itemize}
\item[(i)] $ E=E_+-E_+$.
\item[(ii)]~ For each $x\in \cX$ there exists an $f_x\in E_+$ such that $f_x(x)>0$.
\item[(iii)]~ For each $f\in E_+$ there exists a $g\in E_+$ such that for any $\varepsilon >0$ there exists a compact subset $K_\varepsilon$ of $\cX$ such that $|f(x)|\leq \varepsilon |g(x)|$ for  $x\in \cX\backslash K_\varepsilon$. 
\end{itemize}
\end{definition}

If condition (iii) is satisfied, we shall say that \textit{$g$ dominates $f$}. Roughly speaking, this  means that $|f(x)/g(x)|\to 0$ as $x\to \infty$.
\begin{example}\label{exaadaptedsubspace}
Let $\cX$ be closed subset of $\R^d$.\\ 
Then \textit{$E=\R_d[\ux]\lceil \cX$ is an adapted subspace of $C(\cX;\R)$}. 
Indeed, condition (i) in Definition \ref{defadapted}   follows from the relation $4p=(p+1)^2-(p-1)^2$.  (ii) is trivial. If $p\in E_+$, then $g=(x_1^2+\cdots+x_d^2)f$ dominates $f$, so condition (iii) is also fulfilled. 

In contrast, $E=\R_d[\ux]_m\lceil \cX$, where $\R_d[\ux]_m:=\{p\in \R_d[\ux]:\deg (p)\leq m\}$ is not an adapted subspace of $C(\cX,\R)$, because (iii) fails.

From the technical side this is an important difference between the full moment problem and the truncated moment problem. $\hfill \circ$
\end{example}

\begin{lemma}\label{compactinad}
If\, $E$ is an adapted subspace of  $C(\cX;\R)$,  then for any  $f\in C_c(\cX;\R)_+$ there exists a $g\in E_+$ such that $g(x)\geq f(x)$ for all $x\in \cX$.
\end{lemma}
\begin{proof} 
Let $x\in \cX$.
 By Definition \ref{defadapted}(ii) there exists a function $g_x\in E_+$ such that $g_x(x)>0$.
 Multiplying $g_x$ by some positive constant we get $g_x(x)>f(x)$.
 This inequality remains valid in some neighborhood of $x$.
 By the compactness of\, $\supp f$\, there are finitely many  $x_1,\dots,x_n\in \cX$ such that $g(x):=g_{x_1}(x)+\cdots+g_{x_n}(x)>f(x)$ for  $x\in\supp f$ and $g(x)\geq f(x)$ for all $x\in\cX$.
\end{proof} 

%\subsection{Existence of integral representations}\label{extsinceintrepL}

We will use the following Hahn-Banach type extension result.
\begin{lemma}\label{exthbversion}
Let $E$ be a linear subspace of a real vector space $F$ and let $C$ be a convex cone of $F$ such that $F=E+C$. Then each  $(C\cap E)$-positive linear functional $L$ on $E$  can be extended to a $C$-positive linear functional $\tilde{L}$ on $F$.
\end{lemma}
\begin{proof}
Let $f\in F$. We define
\begin{equation}\label{defqx}
q(f)=\inf\, \{ L(g): g\in E,g-f\in C\}.
\end{equation}
Since $F=E+C$, there exist $g\in E$ and $c\in C$ such that  $-f=-g+c$, so $c=g-f\in C$ and the corresponding set in (\ref{defqx}) is not empty. It is easily seen that $q(f)$ is finite and that $q$ is a
 sublinear functional  such that $L(g)=q(g)$ for $g\in E$. Therefore, by the Hahn--Banach theorem, there is an extension $\tilde{L}$ of $L$ to $F$ such that $\tilde{L}(f)\leq q(f)$ for all $f\in F$. 

Let $h\in C$. Setting $g=0, f=-h$\, we have $g-f\in C$, so that $q(-h)\leq L(0)=0$\, by (\ref{defqx}). Hence  $\tilde{L}(-h)\leq q(-h)\leq 0$, so  that  $\tilde{L}(h)\geq 0$. Thus, $\tilde{L}$ is $C$-positive.
$\hfill \qed$ \end{proof}

% Many existence results on the moment problem have their origin in  the following theorem.
Our first main result is the following theorem. 
\begin{theorem}\label{choquet}
Suppose  $E$ is an adapted subspace of $C(\cX;\R)$. For any linear functional $L:E\to \R$ the following are equivalent:
\begin{itemize}
\item[\em (i)] ~ The functional $L$ is  $E_+$-positive, that is, $L(f)\geq 0$ for all $f\in E_+$.
%\item[\em (ii)]~ For each  $f\in E_+$ there exists an $h\in E_+$ such that $L(f+\varepsilon h)\geq 0$ for all $\varepsilon >0$.
\item[\em (ii)]~ $L$ is a moment functional, that is, there exists a Radon measure $\mu$ on $\cX$  such that $L(f)=\int f d\mu$ for $f\in E$.
\end{itemize}
\end{theorem}

\begin{proof} The implication (ii)$\to$(i) is obvious. We prove
 (i)$\to$(ii) and  begin by setting
\begin{align*}
 \tilde{E}
 := \{f\in C(\cX,\R): |f(x)|\leq g(x),\, x\in \cX, \quad\text{ for some }\quad g\in E\}
\end{align*}
 and claim that  $\tilde{E}=E+(\tilde{E})_+ $. Obviously, $E+(\tilde{E})_+ \subseteq \tilde{E}$. Conversely, let $f\in \tilde{E}$. We choose a  $g\in E_+$ such that $|f|\leq g$. Then we have $f+g \in ( \tilde{E})_+$, $-g\in E$ and  $-f=-g+ (g+f)\in E+(\tilde{E})_+$. That is, $\tilde{E}=E+(\tilde{E})_+$.

By Lemma \ref{exthbversion}, $L$ can be extended to an $(\tilde{E})_+$-positive linear functional $\tilde{L}$ on $\tilde{E}$. We have $C_c(\cX;\R)\subseteq \tilde{E}$ by Lemma \ref{compactinad}. From the Riesz representation theorem  it follows that there is a Radon measure $\mu$ on $\cX$  such that $\tilde{L}(f)=\int f\, d\mu$ for  $f\in C_c(\cX;\R)$. By Definition \ref{defadapted}(i),  $E=E_+-E_+$. To complete the proof it therefore suffices to show that each $f\in E_+$  is $\mu$-integrable and satisfies $L(f)\equiv \tilde{L}(f)=\int f\, d\mu$. 

 Fix $f\in E_+$.
 Let $\cU:=\{\eta\in C_c(\cX;\R): 0 \leq \eta(x)\leq 1\text{ for }x\in \cX\}$.
 Clearly, for $\eta\in \cU$,  $f\eta \in C_c(\cX;\R)$ and hence $\tilde{L}(f\eta)=\int f\eta\, d\mu$.
 Using this fact and the $(\tilde{E})_+$-positivity of $\tilde{L}$, we derive
\begin{equation}\label{esuonedirec}
\int f d\mu =\sup_{\eta \in \cU}\,  \int f\eta\, d \mu= \sup_{\eta\in \cU}\, \tilde{L}(f\eta)\leq \tilde{L}(f)=L(f) <\infty,
\end{equation}
so  $f$ is $\mu$-integrable. 

By (\ref{esuonedirec}) it suffices to  prove that $L(f)\leq \int f d\mu$.  From Definition \ref{defadapted}(iii), there exists a $g\in E_+$ that dominates $f$. For any $\varepsilon >0$  we choose  a function $\eta_\varepsilon\in \cU$ such that $\eta_\varepsilon =1$ on the set $K_\varepsilon$ from condition (iii). Then  $f\leq \varepsilon g+f\eta_\varepsilon$. Since $ f\eta_\varepsilon \leq f$, 
\begin{align*}L(f)=\tilde{L}(f)\leq\varepsilon\tilde{ L}(g)+ \tilde{L}(f\eta_\varepsilon)=\varepsilon L(g)+\int f\eta_\varepsilon d\mu \leq\varepsilon L(g)+\int f d\mu.\end{align*}
Note  that $g$ does not depend on $\varepsilon$!  Passing to the limit $\varepsilon \to +0$,
we get $L(f)\leq\int f d\mu$. Thus, $L(f)=\int f d\mu$ which  completes  the proof of (ii).
$\hfill \qed$ \end{proof}

 If  $\cX$ is compact, then $C(\cX;\R)=C_c(\cX;R)$, so condition (iii) in Definition \ref{defadapted}  is trivially fulfilled. But in this case we can obtain the desired integral representation of $L$ more directly, as the following proposition shows.
\begin{proposition}\label{integralrepcompactcase}
 Suppose   $\cX$ is a compact Hausdorff space and  $E$ is a linear subspace of $C(\cX;\R)$ which contains a function $e$ such that $e(x)>0$ for $x\in \cX$. 
 
 Then each $E_+$-positive linear functional $L$ on $E$ is a moment functional.
 %, that is, there exists a measure $\mu\in M_+(\cX)$ such that\,  $L(f)=\int f\, d\mu$\, for $f\in E$.
\end{proposition}
\begin{proof}
Set $F=C(\cX;\R)$ and $C=C(\cX;\R)_+$. Let $f\in F$. Since $\cX$ is compact, $f$ is bounded and $e$ has a positive infimum. Hence there exists a $\lambda >0$ such that $f(x)\leq \lambda e(x)$  on $ \cX$. Since $\lambda e-f\in C$ and $-\lambda e\in E$,  $f=-\lambda e +(\lambda e-f)\in E+C$. Thus, $F=E+C$. By Lemma \ref{exthbversion}, $L$ extends to a $C$-positive linear functional $\tilde{L}$ on $F$.
 By the Riesz representation theorem, $\tilde{L}$, hence $L$, can be given by a Radon measure $\mu$ on $\cX$.
$\hfill\qed$ \end{proof}

In the proof of Theorem \ref{choquet} condition (iii) of Definition \ref{defadapted} was  crucial. We  give a simple example where this condition fails and $L$ has no representing measure. 
\begin{example}
Set $E:=C_c(\R;\R)+\R\cdot 1$ and define a linear functional on $E$ by 
\begin{align*}
 L(f+\lambda \cdot 1)
 :=\lambda \quad\text{ for }\quad  f\in C_c(\R;\R),~\lambda \in \R,\end{align*} 
where $1$ is the constant function equal to $1$. Then $L$ is  $E_+$-positive, but it is not a moment functional. (Indeed, since $L(f)=0$ for  $f\in C_c(\R;\R)$, the measure $\mu$ would be zero. But this is impossible, because $L(1)=1$.)

We can consider $E$ as a subspace of $C(\ov{\R};\R)$, where $\ov{\R}=\R\cup \{\infty\}$ is the point compactification of $\R$, by setting $(f+\lambda \cdot 1)(\infty)=\lambda$. Then,  $L$ is given by the integral of $\delta_\infty$, so $L$ is a moment functional.  $\hfill \circ$
\end{example}

 Now we give an important application.
 For a closed subset $\cK$ of $\R^d$ we set 
\begin{align*}
 {\Pos} (\cK)
 =\{p\in \R_d[\ux]: p(x)\geq 0 \quad\text{ for all }~~ x\in \cK\, \}.
\end{align*}
Since  
$E=\R_d[\ux]\lceil \cK $ 
is an adapted subspace of $C(\cK,\R)$, as noted in Example \ref{exaadaptedsubspace}, Theorem \ref{choquet} gives the following result, which is called  \textbf{Haviland's theorem}.
\begin{theorem}\label{haviland}
Let $\cK$ be a closed subset of $\R^d$ and $L$ a linear functional on $\R_d[\ux]$. The following statements are equivalent:
\begin{itemize}
\item[\em (i)] $L(f)\geq 0$\, for  all $f\in {\Pos} (\cK)$.
\item[\em (ii)]~ $L$ is a $\cK$-moment functional, that is, there exists a  Radon measure $\mu$ on $\R^d$ supported on $\cK$ such that\,  $L(f)=\int_\cK f\, d\mu$\, for all $f\in \R_d[\ux]$.
\end{itemize}
\end{theorem}
%\begin{proof}
%Set $\cX=K$. Then $E=\R_d[\ux]$ is an adapted subspace of $C(K,\R)$. Indeed, condition (i) in Definition \ref{defadapted}   follows from the relation $4p=(p+1)^2-(p-1)^2$. Condition (ii) is trivial. If $p\in E_+$, then $g=(x_1^2+\cdots+x_d^2)f$ dominates %$f$, so condition (iii) is also fulfilled. Thus Theorem \ref{choquet} gives the assertion.
%$\hfill \qed$ \end{proof}

\section{Positive polynomials on intervals}

To settle the existence problem for moment problems on intervals, 
by Haviland's theorem it is natural to look for descriptions of positive polynomials on intervals.

 Let $p(x)\in \R[x]$ be a   nonconstant polynomial.
 If $\lambda$ is a non-real zero of $p$ with multiplicity $l$, so is $\overline{\lambda}$.
 Clearly,  $(x-\lambda)^l(x-\overline{\lambda})^l=((x-u)^2+v^2)^l$, where $u=\rRe\lambda$ and $v=\rIm\lambda$.
 Therefore, by the fundamental theorem of algebra,   $p$ factors as
\begin{align}
%p(x)&=a(x-\alpha_1)^{n_1}\cdots(x-\alpha_r)^{n_r}(x-\lambda_1)^{j_1}(x-\overline{\lambda}_1)^{j_1}\dots  
% p(x)&=(x-\lambda_k)^{j_k}(x-\overline{\lambda}_k)^{j_k}\\ 
p(x)&= a(x-\alpha_1)^{n_1}\cdots(x-\alpha_r)^{n_r}((x-u_1)^2+v_1^2)^{j_1}\cdots  ((x-u_k)^2+v_k^2)^{j_k},\label{pfactors2}
\end{align}
where  $x-\alpha_j$ are pairwise different linear factors and $(x-u_l)^2+v_l^2$ are pairwise different quadratic factors. Note that linear factors or quadratic factors may be absent. 
%in  (\ref{pfactors2}).  
%\begin{align*}
%&n_1,\dots,n_r,j_1,\dots, j_k\in \N,~ a, \alpha_1,\dots,\alpha_r\in \R,\\
%\lambda_1=u_1+{\ii} v_1,&\dots,\lambda_k=u_k+ {\ii} v_k,~u_1,\dots,u_k\in \R,~  v_1>0,\dots,v_k> 0,\\
%\alpha_i &\neq \alpha_j~~{\rm if}~~ i\neq j, ~{\rm and}~ \lambda_i\neq \lambda_j, ~ \lambda_i\neq %%%%%%  
%\overline{\lambda}_j~~ {\rm if}~~ i\neq j. 
%\end{align*}
%Note that linear factors or quadratic factors may be absent in  (\ref{pfactors2}).  

Let  $\sum \R[x]^2$ denote the set of finite sums of squares $p^2$, where $p\in \R[x]$.

\begin{proposition}\label{positivepolyr}
\begin{itemize}
\item[\em (i)]~ ${\Pos}(\R)=\sum \R[x]^2$.
\item[\em (ii)]~ ${\Pos}([0,+\infty))=\big\{\, f+x g:\,  f,g\in \sum\, \R[x]^2\, \}$.
\item[\em (iii)]~ 
${\Pos}([a,b]) =\big\{\, f+(b-x)(x-a) g:\, f, g\in \sum~ \R[x]^2 \, \big\}$, where $a,b\in \R, a<b$.
\end{itemize}
\end{proposition}
\begin{proof} 
It suffices to show that the sets on the left are subsets of the sets on the right.

(i): Let $p\in {\Pos}(\R)$, $p\neq0$. Since $p(x)\geq 0$ on $\R$,   $a>0$ and the numbers $n_1,\dots,n_r$ in (\ref{pfactors2}) are even. Hence $p$ is a product of squares and of sums of two squares. Therefore,  $p\in \sum \R[x]^2$.

(ii): Let $p\in {\Pos}([0,+\infty)), p\neq 0,$ and consider  (\ref{pfactors2}).
Set $Q:=\sum \R[x]^2+x\sum \R[x]^2$.  For\, $f_1,f_2,g_1,g_2\in \sum \R[x]^2$, we have 
\begin{align*}
 (f_1+ xg_1)(f_2+x g_2)=(f_1f_2 +x^2g_1g_2)+ x(f_1g_2+g_1f_2) \in Q.
\end{align*}
 Hence $Q$ is closed under multiplication. Therefore, it suffices to show that all factors in (\ref{pfactors2}) are in $Q$.
For products of  quadratic factors and   even powers of linear factors this is  obvious. It remains  to treat the constant  $a$ and  linear factors $x-\alpha_i$ with  real zeros $\alpha_i$  of odd multiplicities. Since $p(x)\geq 0$ on $[0,+\infty)$, $a>0$ by letting $x\to +\infty$  and $\alpha_i\leq 0$,  because $p(x)$ changes its sign in a neighborhood of a zero with odd multiplicity.  Hence $a\in Q$ and\, $x-\alpha_i=(-\alpha_i+ x)\in\sum \R[x]^2 +x\sum \R[x]^2=Q.$

(iii) follows by a similar, but slightly longer reasoning 
 [MP, Proposition 3.3].
$\hfill \qed$ \end{proof}

The next result, called \textit{Markov--Lukacs theorem}, sharpens  Proposition \ref{positivepolyr}(iii) by replacing sums of squares by single squares and adding degree requirements. Its proof is much more involved than that of Proposition \ref{positivepolyr}(iii). However for the solution of the truncated moment problem on $[a,b]$ in Lecture 3 it suffices to have the weaker statement with sum of squares instead of single squares which is easy to prove [MP, Proposition 3.2].
\begin{proposition}\label{markov} For $a,b\in \R$, $a<b$, and $n\in \N_0$, 
\begin{align*}
&{\Pos}([a,b])_{2n}=\{ p_n(x)^2+ (b{-}x)(a{-}x)q_{n-1}(x)^2: p_n\in \R[x]_n, q_{n-1}{\in} \R[x]_{n-1} \},\\
&{\Pos}([a,b])_{2n+1}=\{(b-x) p_n(x)^2+ (a-x)q_n(x)^2: p_n, q_n\in \R[x]_n\, \}.
\end{align*} 
\end{proposition}
\begin{proof}
[MP, Corollary 3.24]. $\hfill \qed$\end{proof}
\section{Moment problems on intervals}

In this section we combine Haviland's theorem with the description of positive polynomials on intervals and derive the solutions for the classical one-dimensional moment problems.

Let $s=(s_n)_{n\in \N_0}$ be a real sequence. The \textit{Riesz  functional} $L_s$ is the linear functional  on $\R[x]$ defined  by  $L_s(x^n)=s_n$, $n\in \N_0$.
The sequence $s$  is called \textit{positive semidefinite} if  for all  $\xi_0,\xi_1,\dots,\xi_n\in \R$ and $n\in \N$ we have
\begin{align}\label{quadraticform}
\sum_{k,l=0}^n s_{k+l}\xi_k\,\xi_l \geq 0.
\end{align}
The set of positive semidefinite  sequences  is denoted by $\cP(\N_0)$. 

A linear functional $L$ on $\R[x]$ is called \textit{positive} if $L(p^2)\geq 0$ for all $p\in R[x]$.

%Let   $Es$ denote the shifted sequence given by \begin{align*}(Es)_n=s_{n+1},\quad n\in \N_0.\end{align*} Clearly, $L_{Es}(p(x))=L_s(x p(x))$ for %$p\in \R[x]$.

Further,  we define the \textit{Hankel matrix}\index{Hankel matrix} $H_n(s)$ and the \textit{Hankel determinant}\index{Hankel determinant} $D_n(s)$ by
\begin{gather}\label{hankelmatrixt1}
 \quad H_n(s)=
\left(
\begin{array}{lllllll}
s_{0} & s_{1} &  s_{2}  & \dots &s_{n} \\
s_{1} & s_{2} & s_3  & \dots  & s_{n+1} \\
s_2 & s_3 & s_4 &   \dots & s_{n+2} \\
%s_3 & s_4 & s_5 &  \dots &s_{n+3} \\
\dots & \dots & \dots   & \dots &\dots \\
s_n & s_{n+1}& s_{n+2}& \dots & s_{2n}
\end{array}\right),\quad D_n(s)=\det H_n(s).
\end{gather}

 The following result is \textbf{Hamburger's theorem}\index{Theorem! Hamburger}.
\begin{theorem}\label{solutionhmp} (\textit{Solution of the Hamburger moment problem})\index{Hamburger moment problem}\index{Hamburger moment problem! existence of a solution}\index{Moment problem! for the real line|see{Hamburger moment problem}}\\
For any real sequence  $s=(s_n)_{n\in \N_0}$ the following statements are equivalent:
\begin{itemize}
\item[\em (i)]~ $s$ is a Hamburger moment sequence, that is, there is a Radon measure $\mu$ on $\R$ such that
\begin{align}\label{intmoments}
s_n=\int_\R x^n d\mu(x)\quad \text{for}\quad 
n\in \N_0.
\end{align}
\item[\em (ii)]~ $s\in \cP(\N_0),$ that is, the sequence $s$ is positive semidefinite.
\item[\em (iii)]~ All Hankel matrices $H_n(s)$, $n\in \N_0$,  are positive semidefinite.
\item[\em (iv)]~ $L_s$ is a positive linear functional on $\R[x]$, that is, $L_s(p^2)\geq 0$ for $p\in \R[x].$
\end{itemize}
\end{theorem}
\begin{proof}
The main implication (iv)$\to$(i) follows from Haviland's Theorem \ref{haviland} combined with Proposition \ref{positivepolyr}(i).
A straightforward computation shows that  (i) implies (ii). Since the Hankel matrix $H_n(s)$ is just the  matrix associated with the quadratic form   (\ref{quadraticform}), (ii) and (iii) are equivalent.  For $p(x)=\sum_{j=0}^n a_n x^k\in \R[x]$ we compute
\[
L_s(p^2)=\sum_{k,l=0}^n a_ka_l L_s(x^{k+l})=\sum_{k,l=0}^n a_ka_l s_{k+l}.
\]
Hence  (ii) and (iv) are equivalent.
$\hfill \qed$ \end{proof}

The next proposition answers the question of when $s$ has a representing measure of \textit{finite} support. %we do not give the proof and refer to \cite[Propositon 3.9]{mp}.
\begin{proposition}\label{mufinitesupp}
For  a positive semidefinite sequence $s$ the following are equivalent:
\begin{itemize}
\item[\em (i)]~ There is a number $n\in \N_0$ such that 
\begin{align}\label{hankelsemidef}
D_0(s)>0, \dots,D_{n-1}(s)>0 \quad \text{and} ~~~ D_k(s)=0\quad \text{for}~~ k\geq n.
\end{align}
\item[\em (ii)]~ $s$ is a moment sequence with a representing measure $\mu$ with  support of $n$ points.
\end{itemize}
\end{proposition}

\begin{remark} It was  recently proved by Berg and Szwarc (2015)
%in \cite{Bergszwarc3} 
that the assumption ``$s$ is positive semidefinite" in Proposition \ref{mufinitesupp} can be omitted.  $\hfill \circ$ 
\end{remark}

To solve moment problems on intervals it is convenient to have  the shifted sequence $Es$ defined by \begin{align*}(Es)_n=s_{n+1},\quad n\in \N_0.\end{align*} Clearly, $L_{Es}(p(x))=L_s(x p(x))$ for $p\in \R[x]$.

 The next main  result is \textbf{Stieltjes' theorem}\index{Theorem! Stieltjes}.\index{Moment problem! for the half-line|see{Stieltjes moment problem}} 
\begin{theorem}\label{solstietjesmp} (\textit{Solution of the Stieltjes moment problem})\index{Stieltjes moment problem}\index{Stieltjes moment problem! existence of a solution}\\
For any real sequence  $s$ the following statements are equivalent:
\begin{itemize}
\item[\em (i)]~ $s$ is a Stieltjes moment sequence, that is, there is a Radon measure $\mu$ on $[0,+\infty)$ such that 
\begin{align}\label{sintmoments}
s_n=\int_0^\infty x^n d\mu(x)\quad \text{for}\quad 
n\in \N_0.
\end{align}
\item[\em (ii)]~ $s\in \cP(\N_0)$ and $Es\in \cP(\N_0)$.
\item[\em (iii)]~ All Hankel matrices $H_n(s)$ and $H_n(Es)$, $n\in \N_0$,  are positive semidefinite.
\item[\em (iv)]~  $L_s(p^2)\geq 0$ and $L_s(xq^2)\geq 0$ for all $p,q\in \R[x].$
\end{itemize}
\end{theorem}
\begin{proof}
The proof is almost the same as the proof of Theorem \ref{solutionhmp}; instead of Proposition \ref{positivepolyr}(i) we use Proposition \ref{positivepolyr}(ii).
$\hfill \qed$ \end{proof} 

Combining  Haviland's theorem with  Proposition \ref{positivepolyr}(iii)  yields the following.
\begin{theorem}\label{hausdirffsolmp} (\textit{Solution of the  moment problem for a compact interval})\index{Moment problem! for compact intervals}
\\
Let $a,b \in \R$, $a<b$. For  a real sequence $s$ the following  are equivalent: 
\begin{itemize}
\item[\em (i)]~ $s$ is an $[a,b]$-moment sequence.
\item[\em (ii)]~ $s\in \cP(\N_0)$\, and\, $((a+b)Es -E(Es)-ab\, s)\in \cP(\N_0)$.
\item[\em (iii)]~ $L_s(p^2)\geq 0$ and $L_s((b-x)(x-a)q^2)\geq 0$\, for all $p,q\in \R[x]$.
\end{itemize}
\end{theorem}

 Finally, we   state  two  solvability criteria of  moment problems that are not based on  squares.
 The first result is easily derived from  \textit{Bernstein's theorem}.
 It says that each polynomial $p\in \R[x]$ such that $p(x)>0$ on $[-1,1]$ can be written as
% \cite[Proposition 3.4]{mp} 
\begin{align*}
 p(x)
 =\sum_{k,l=0}^n \alpha_{kl}(1-x)^k(1+x)^l, \quad \text{ where }~~ \alpha_{kl}\geq 0.
\end{align*}

\begin{theorem}\label{applbernstein}
 Let $s=(s_n)_{n\in \N_0}$ be a real sequence and let $L_s$ be its Riesz functional on $\R[x]$. Then $s$ is a\, $[-1,1]$-moment sequence if and only if 
\begin{align}\label{bernsteincondition}
L_s((1-x)^k(1+x)^l)\geq 0\quad \text{for~ all}\quad k,l\in \N_0.
\end{align}
\end{theorem}

 The next theorem is \textbf{Hausdorff's theorem}.
 It can be obtained by a simple computation from Theorem \ref{applbernstein} using the bijection $x\mapsto \frac{1}{2}(1+x)$ of $[-1,1]$ onto $[0,1]$. 
\begin{theorem}\label{completleymonetone}
A real sequence $s$  
is a $[0,1]$-moment sequence if and only if
\begin{align}\label{complmonotone}
((I-E)^n s)_k \equiv \sum_{j=0}^n (-1)^j\binom{n}{j} s_{k+j}\geq 0\quad \text{for}\quad k,n\in \N_0.
\end{align}
\end{theorem}

%\end{document}

\setcounter{tocdepth}{3}

\makeatletter
\renewcommand{\@chapapp}{Lecture}
\makeatother
%\chapter{ }

\chapter{One-dimensional moment problem:  determinacy}

Abstract:\\
\textit{This lecture is concerned with  the determinacy question  for one-dimensional Hamburger and Stieltjes moment problems. The  Carleman condition (a sufficient condition for determinacy) and the Krein condition (a sufficient condition for indeterminacy) are developed.}
\bigskip

Let $\mu$ be a Radon measure on $\R$ with finite moments $s_n=\int_\R x^n d\mu(x), n\in \N_0$. If $\mu$ has compact support, it is easily seen (using the Weierstrass approximation theorem on uniform approximation of continuous functions by polynomials) that $\mu$ is uniquely determined by its moment sequence $s=(s_n)_{n\in \N_0}$. Further, if $c>0$, it can be shown that $\mu$ is supported on the interval $[-c,c]$ if and only if \[\liminf_{n\to \infty} \sqrt[2n]{s_{2n}}\leq c.\]
In particular, $\mu$ is supported on $[-1,1]$ if and only if the sequence $s$ is bounded.

Measures with unbounded supports are not necessarily determined  by their moment sequences. In this very short Lecture we will study when this happens.
\begin{definition}\label{hamburgerdet} A moment sequence $s$ is called \emph{determinate} if it has only one representing measure; otherwise $s$ is called \emph{indeterminate}.
\end{definition}
Likewise, a Radon measure $\mu$ with finite moments is called determinate (resp. indeterminate) if and only if its moment sequence has this property. 
% A striking explicit example was already  discovered by Stieltjes (we reproduce it as Example xxx below).  

\section{An indeterminate measure: lognormal distribution}
 The first  example of an indeterminate moment sequence was discovered by T. Stieltjes  (1894).
 He showed that the \textit{log-normal distribution}\, $d\mu= f(x)dx$\, with density
\begin{align*}
 f(x)
 =\frac{1}{\sqrt{2\pi}}~ \chi_{(0,+\infty)}(x) x^{-1}\exp( -( \ln x)^2/2)
\end{align*}
has finite moments and the corresponding moment sequence is indeterminate. We reproduce  this famous classical example here.

 Let $n\in \Z$.
 Substituting $y=\ln x$ and $t=y-n$, we compute 
\begin{align*}
 s_n
 &=\int _\R x^n~ d\mu(x)=\frac{1}{\sqrt{2\pi}}\int_0^\infty x^{n-1} e^{-(\ln x)^2/2}~dx
 = \frac{1}{\sqrt{2\pi}} \int_\R e^{ny} e^{-y^2/2}  ~dy\\
 &=\frac{1}{\sqrt{2\pi}}\int_\R e^{-(y-n)^2/2} e^{n^2/2} ~dy
 =e^{n^2/2}\frac{1}{\sqrt{2\pi}} \int_\R~e^{t^2/2}\, dt
 = e^{n^2/2} .
\end{align*}
 This proves   $\mu$ that finite moments and its moment sequence is $s=(e^{n^2/2})_{n\in \N_0}$.

 For  $c\in [-1,1]$ we define a positive (!) measure $\mu_{c}$ by 
\begin{align*}
 d\mu_{c}(x)
 =[1+c\sin(2\pi\ln x) ]~d\mu(x).
\end{align*}
 Since $\mu$ has finite moments, so has $\mu_c$.
 For $n\in \Z$, we compute
\begin{align*}
 &~\int_\R x^n\sin(2\pi~ {\ln}~ x)~d \mu(x)
 = \frac{1}{\sqrt{2\pi}}~\int_\R  e^{ny} \,(\sin2\pi  y)~ e^{-y^2/2} dy~ \\
 & =\frac{1}{\sqrt{2\pi}}~\int_\R e^{-(y-n)^2/2} e^{n^2/2} \sin2\pi  y ~dy
 = \frac{1}{\sqrt{2\pi}} e^{n^2/2}\int_\R e^{-t^2/2}\sin2\pi (t+n)~ dt=0,
\end{align*}
 where we used the fact that the function $\sin2\pi(t+n)$  is odd.
 By the definition of $\mu_c$, it follows that any $c\in [-1,1]$ the measure $\mu_c$ has the same moments as $\mu$.
 (This is even true for all $s_n$ with $n\in \Z$.)
 Thus  $\mu$ is \textit{not} determinate!

\section{Carleman's  condition  for the Hamburger  moment problem}%\label{carlemanscondition}
%A Hamburger  moment sequence is called {\it determinate} if it has only one representing measure. A measure $\mu\in \cM(\R)$ is said to be {\it %determinate} if its moment sequence is determinate.

 The following important result is the \textbf{Carleman theorem}. 
\begin{theorem}\label{Carleman}
 Let  $s=(s_n)_{n\in \N_0}$ be a Hamburger moment sequence.
 If $s$ satisfies the \emph{Carleman condition}\index{Carleman condition}
\begin{align}\label{carlemanhamburger}
\sum_{n=1}^\infty ~   s_{2n}^{-\frac{1}{2n}} =+\infty,
\end{align}
 then $s$ is a \emph{determinate}.
\end{theorem}
\begin{proof}
 {[MP,Theorem~ 4.3]}.
%\cite[Theorem 4.3]{mp}.
 $\hfill \qed$
\end{proof}
The 
%most elegant 
standard proof of Theorem \ref{Carleman}  derives the assertion from the Denjoy--Carleman theorem  
on quasi-analytic functions. 
An operator-theoretic proof based  on Jacobi operators is developed in Section 6.4 of [MP].

 Since $s_{2n}=\int_\R x^{2n} d\mu\geq 0$ for all $n\in \N_0$, we have either $s_{2n}^{-\frac{1}{2n}}=0$ or  $s_{2n}^{-\frac{1}{2n}} >0.$ 

The Carleman condition (\ref{carlemanhamburger})  requires that the moment sequence is not growing to fast.

\begin{corollary}\label{analyticcase} 
Let $s=(s_{n})_{n\in \N_0}$ be a Hamburger moment sequence. 
If there is a constant $M>0$ such that 
\begin{align}\label{analyticcase1}
s_{2n}\leq M^n (2n)!\quad \text{for}~~  n\in \N,
\end{align}
then Carleman's condition (\ref{carlemanhamburger}) holds and $s$ is  determinate.
\end{corollary}
\begin{proof} 
It is obvious that\, $(2n)!\leq (2n)^{2n}$ for $n\in \N$. Therefore,\, $ [(2n)!]^{1/2n}\leq 2n$, so that\, $\frac{1}{2n}\leq[(2n)!]^{-1/2n} $ and hence
\begin{align*} M^{-1/2}\frac{1}{2n}\leq M^{-1/2}[ (2n)!]^{-1/2n}\leq s_{2n}^{-1/2n}, \quad n\in \N.\end{align*}
Thus, Carleman's condition (\ref{carlemanhamburger}) is satisfied, so   Theorem \ref{Carleman}(i) applies.
\qed \end{proof}

\begin{corollary}\label{evarepsilonx^2} 
Let\, $\mu$ be a Radon measure on $\R$. If 
there exists an $\varepsilon>0$ such that 
\begin{align}\label{eepsilo|x|ebd}
\int_\R e^{\varepsilon |x|}\, d\mu(x)<\infty,
\end{align} then $\mu$ has all moments, condition (\ref{analyticcase1}) 
holds, and   $\mu$ is determinate.
\end{corollary} 
\begin{proof}
Clearly,\, $e^{\varepsilon |x|}\geq (\varepsilon x)^{2n}\frac{1}{(2n)!}$\, and  hence\, $x^{2n}e^{-\varepsilon |x|}\leq \varepsilon^{-2n}(2n)!$\, for $n\in \N_0$ and $x\in \R.$ 
Therefore,
\begin{align}\label{eepsilonx}
\int_\R x^{2n}d\mu(x)= \int_\R x^{2n}e^{-\varepsilon |x|}e^{\varepsilon |x|}  d\mu(x)\leq 
\varepsilon^{-2n} (2n)!~\int_\R e^{\varepsilon |x|}  d\mu(x) <\infty.
\end{align}
This implies\, $\int |x^k| d\mu(x)<\infty$ for  $k\in \N_0$. From 
(\ref{eepsilonx}) it follows that  (\ref{analyticcase1}) holds, so  Corollary  \ref{analyticcase} gives the assertion.\qed \end{proof}

%In probability theory,  condition (\ref{eepsilo|x|ebd})  is called {\it Cramer's condition}.

As shown by C. Berg and J.P.R. Christensen,    Carleman's condition (\ref{carlemanhamburger}) implies also that the polynomials $\C[x]$ are dense in $L^p(\R,\mu)$ for  $p\in [1,+\infty),$ where $\mu$ is the unique representing measure of $s$. 

\section{Krein's condition for the Hamburger moment problem}

 The following \textbf{Krein theorem} \index{Theorem! Krein} shows that, for measures given by a density, the so-called \textit{Krein condition}\index{Krein condition} (\ref{kreincond}) is a sufficient condition for \textit{indeterminacy}.
\begin{theorem}\label{Kreincondition}
 Let  $f$ be a nonnegative Borel function on $\R$ such  that the measure $\mu$ defined by\,  $d\mu=f(x)dx$\,  has finite moments   $s_n:=\int x^n d\mu(x)$\, for all $n\in \N_0$.
 If 
\begin{align}\label{kreincond}
 \int_\R  \frac{\ln f(x)}{1+x^2}~ dx >-\infty,
\end{align}
then the  moment sequence $s=(s_n)_{n\in \N_0}$ is indeterminate and the polynomials $\C[x]$ are  not dense in $L^2(\R,\mu)$. 
\end{theorem}
\begin{proof}
 {[MP, Theorem~ 4.14]}.
%\cite[Theorem 4.14]{mp}.
$\hfill\qed$ \end{proof}

The  proof in [MP]  uses  results on boundary values of analytic functions in the upper half plane.
% Another  proof is given in \cite{berg}.
\smallskip

Note  that  (\ref{kreincond}) implies that $f(x)>0$ a.e. 

 To give a slight reformulation  of condition (\ref{kreincond}), set\, $\ln^+x:=\max\, (\ln x,0)$\, and\, $\ln^-x:=-\min\, (\ln x,0)$ for $x\geq 0$.
 Since $f(x)\geq 0$ and hence $\ln^+ f(x)\leq f(x)$, 
\begin{align*}
 0
 \leq \int_\R  \frac{\ln^+f(x)}{1+x^2}~ dx
 \leq \int_\R  \frac{f(x)}{1+x^2}~ dx
 \leq \int_\R  f(x)~ dx
 =s_0
 <+\infty.
\end{align*}
 Therefore, since\, 
$\ln x=\ln^+x- \ln^-x$, Krein's condition (\ref{kreincond}) is equivalent to
\begin{align*} 
 \int_\R  \frac{\ln^-f(x)}{1+x^2}~ dx
 <+\infty.
\end{align*}

%The integral
%\begin{align*}
%\frac{1}{\pi}\int_\R  \frac{{\rm ln}\, f(x)}{1+x^2}\, dx
%\end{align*}
%is of interest in itself. It is    called the {\it entropy integral}\index{Entropy integral}\, or\, {\it logarithmic integral}.

\begin{example}\label{carlemannotnec}~\textit{The Hamburger moment problem for\, $d\mu=e^{-|x|^\alpha}dx$,  $\alpha>0.$}\\
 Clearly,  $\mu\in \cM_+(\R)$ for any $\alpha>0$.  
If\, $0<\alpha<1$,  then 
%Krein's condition (\ref{kreincond}) is satisfied, since
\begin{align*}
 \int_\R \frac{\ln e^{-|x|^\alpha}}{1+x^2}\, dx
 =\int_\R \frac{-|x|^\alpha}{1+x^2}\,dx
 >-\infty.
\end{align*}
 so Krein's condition (\ref{kreincond}) is satisfied and hence the Hamburger moment sequence of $\mu$ is indeterminate by Theorem \ref{Kreincondition}.

Now we suppose that $\alpha\geq 1$, then
\begin{align}\label{anesti}
s_n=\bigg(\int_{-1}^1 +\int_{|x|\geq 1}\bigg)x^n e^{-|x|^\alpha}dx\leq 2+ \int_\R x^n e^{-|x|}dx=2+2n!\leq 2^n n!.
\end{align}
Therefore, by Corollary \ref{analyticcase},  the  moment sequence  of $\mu$ is determinate.
$\hfill \circ$ 
\end{example}

\section{Carleman   and Krein conditions for the Stieltjes  moment problem}%\label{carlemanscondition}
%\end{document}
 We shall say that a Stieltjes moment sequence is \textit{determinate} if it has only one representing measure supported on $[0,+\infty)$. 

The following is Carleman's theorem   for the Stieltjes moment problem.
\begin{theorem}\label{Carlemans}
If $s=(s_n)_{n\in \N_0}$ is a Stieltjes moment sequence such that
\begin{align}\label{carlemanstieltjes}
\sum_{n=1}^\infty ~   s_{n}^{-\frac{1}{2n}} =+\infty,
\end{align}
 then $s$ is a \emph{determinate} Stieltjes moment sequence.
\end{theorem}
\begin{proof}
%\cite[Theorem 4.3(ii)]{mp}.
 {[MP, Theorem~4.3(ii)]}.
$\hfill \qed$
\end{proof}

 Note that $s_n\geq0$ for all $n\in\N_0$, because $s$ is a Stieltjes moment sequence.

The following two corollaries can be easily derived from Theorem \ref{Carlemans}.

\begin{corollary}\label{analyticcases} 
Let $s=(s_{n})_{n\in \N_0}$ be a Stieltjes moment sequence. If  there is a constant $M>0$ such that 
\begin{align}\label{analyticcase2}
s_{n}\leq M^n (2n)!\quad \text{for}~~  n\in \N,
\end{align}
then $s$ is a determinate Stieltjes moment sequence.
\end{corollary}

\begin{corollary}
Suppose that\, $\mu$ is a Radon measure supported on $[0,+\infty)$. If  
there exists an $\varepsilon>0$ such that 
\begin{align}\label{eepsilo|x|ebd2}
\int_0^\infty e^{\varepsilon \sqrt{x}}\, d\mu(x)<\infty,
\end{align} then  $\mu$ has finite moments,  condition (\ref{analyticcase2}) is satisfied, and the corresponding Stieltjes moment sequence is  determinate.
\end{corollary} 

 In probability theory the two sufficient determinacy conditions (\ref{eepsilo|x|ebd}) and (\ref{eepsilo|x|ebd2}) are called \emph{Cramer's condition}\index{Cramer condition} and \emph{Hardy's condition}\index{Hardy condition}, respectively.

 The next theorem is about  Krein's condition for the Stieltjes moment problem.
 
\begin{theorem}\label{kreinstieltjesmp}
Let $f$, $\mu$, and $s$ be as in Theorem \ref{Kreincondition}. If the measure $\mu\in \cM_+(\R)$ is supported on $[0,+\infty)$ and
\begin{align}\label{kreincondstielt}
 \int_\R \frac{\ln f(x^2)}{1+x^2}~dx
 \equiv\int_0^\infty  \frac{\ln f(x)}{(1+x)}~ \frac{dx}{\sqrt{x}}
 >-\infty,
\end{align}
 then $s$ is an \emph{indeterminate} Stieltjes moment sequence.
\end{theorem}
\begin{proof} 
%\cite[Theorem 4.17]{mp}.
 {[MP, Theorem~4.17]}.
$\hfill\qed$
\end{proof}

\begin{example}\label{tsieltakrca} ~\textit{The Stieltjes moment problem for\, $d\mu=\chi_{[0,\infty)}(x)e^{-|x|^\alpha}dx$,  $\alpha>0.$}\\
 If  $0<\alpha<1/2$,  then  (\ref{kreincondstielt}) holds, so the Stieltjes moment sequence  is indeterminate.
 If\, $\alpha \geq 1/2$, then $2\alpha\geq 1$ and hence by (\ref{anesti}),
\begin{align*}
 s_n
 = \int_0^\infty x^n d\mu(x)
 =\int_0^\infty (x^2)^n d\mu(x^2)
 =\int_0^\infty x^{2n}e^{-|x|^{2\alpha}} dx\leq 4^n (2n)!\, .
\end{align*}
 By Corollary \ref{analyticcases},  the Stieltjes moment sequence is determinate for $\alpha \geq 1/2$.
$\hfill \circ$
\end{example}

In both Examples \ref{carlemannotnec} and \ref{tsieltakrca}  the determinacy was decided for all parameter values $\alpha$ by using  only the Carleman and Krein theorems. This indicates that both criteria are efficient and strong to  cover even borderline cases.

\begin{remark}
 Suppose $s$ is a determinate Stieltjes moment sequence.
 This means that $s$ has a unique    representing measure $\mu$ supported on $[0,+\infty)$.
 Then it may happen (see, e.g., [MP, Example~ 8.11]) that $s$ is not determinate as a Hamburger moment sequence, that is, $s$ may have another  representing measure which is not supported on $[0,+\infty)$.
 However, if  $\mu(\{0\})=0$, then  $s$ is also a determinate  Hamburger moment sequence according to Definition \ref{hamburgerdet} by [MP, Corollary~ 8.9].
\end{remark}

\secret{
\section{Measures supported on bounded intervals}
The following proposition  characterizes  measures with bounded support in terms of their moments.
\begin{proposition}\label{comapctsupp}
Suppose $s=(s_n)_{n\in \N_0}$ is a Hamburger moment sequence. Let $\mu$ be  representing measure of $s$ and let $c\in \R, c>0$. The following are equivalent:
\begin{itemize}
\item[\em (i)]~ $\mu$ is supported on $[-c,c]$. 
\item[\em (ii)]~ There exists a number\, $d>0 $ such that\, $|s_n|\leq dc^n$ for $n\in \N_0$.
%\item[\em (iii)]~ There exists a number\, $d>0 $ such that\, $s_{2n}\leq dc^{2n}$ for $\in\N_0$.
\item[\em (iii)]~  $S:=\liminf_{n\to \infty}\, s_{2n}^{\frac{1}{2n}}\leq c$.
\end{itemize}
\end{proposition}
\begin{proof} The implications (i)$\to$(ii)$\to$(iii) are easily verified  with $d>s_0$, so it suffices to prove  (iii$\to$(i). 
For $\alpha>0$, let $\chi_\alpha$ denote the characteristic function of  $\R\backslash (-\alpha,\alpha)$ and $M_\alpha:=\int \chi_\alpha d\mu$. Then 
\begin{align*}
M_\alpha \alpha^{2n}= \int_\R \alpha^{2n} \chi_\alpha d\mu\leq  \int_\R x^{2n}\chi_\alpha d\mu\leq\int_\R x^{2n}d\mu=s_{2n}
\end{align*}
and hence\, $(M_\alpha)^{\frac{1}{2n}} \alpha \leq  s_{2n}^{\frac{1}{2n}}$\, for $n\in \N$. Therefore, if $M_\alpha >0$, by passing to the limits we obtain $\alpha \leq S$. Thus, $M_\alpha=\mu(\R\backslash (-\alpha,\alpha)) =0$ when $\alpha >S$. Since $S\leq c$ by (iii),\, ${\rm supp}\, \mu \subseteq [-S,S]\subseteq [-c,c]$, which proves (i).
\qed\end{proof}
\begin{corollary}\label{detcompct}
If a Hamburger moment sequence $s$  has a representing measure with compact support, then $s$ is determinate.
\end{corollary}
\begin{proof}
If $\mu_1,\mu_2$ are representing measures of $s$, both measures  are supported on $[-S,S]$  by Proposition \ref{comapctsupp}. Then, for all $f\in \R[x]$, 
\begin{align}\label{intf(x)}\int_{-S}^S f(x)\, d\mu_1=\int_{-S}^S f(x)\, d\mu_2.
\end{align}
  Since the polynomials $\R[x]$ are dense in $C([-S,S];\R)$  by the Weierstrass theorem,  (\ref{intf(x)}) holds for all continuous functions $f$. Hence $\mu_1=\mu_2.$
\qed \end{proof}

Another immediate consequence of Proposition \ref{comapctsupp} is the following.
\begin{corollary}
A  moment sequence $s$ has a representing measure supported on $[-1,1]$ if and only if the  sequence $s$ is bounded.
\end{corollary}
}

\setcounter{tocdepth}{3}

\makeatletter
\renewcommand{\@chapapp}{Lecture}
\makeatother

\chapter{The one-dimensional truncated moment problem on a bounded interval}

Abstract:\\
\textit{The truncated moment problem on a compact interval $[a,b]$ is treated. Basic results on the existence and the determinacy are obtained. For interior points of the moments cone principal measures and canonical measures are developed.}
\bigskip

In contrast to the preceding lecture  we study in this Lecture moment problems where only finitely many moments are given and the measures are supported on a bounded interval. More precisely, suppose
$a$ and $b$ are  real numbers such that $a<b$ and $m\in \N$. 
We consider the \emph{truncated moment problem} on the interval $[a,b]$:

\textit{Given a real sequence $s=(s_j)_{j=0}^m$, when is there a  Radon measure $\mu$ on $[a,b]$ such that 
 $s_j=\int_a^b x^j\, d\mu(x)$ for $j=0,\dots,m$?} 
 
 In this case we say that $s$ is a \emph{truncated $[a,b]$-moment sequence} and $\mu$ is a representing measure for $s$. We shall see that  odd and  even cases are different.

\section{Existence of a solution}\label{existence}

First we fix some notation. 
Suppose  $s=(s_j)_{j=0}^m$ is a real sequence. Let    $L_s$  be the Riesz functional on $\R[x]_m:=\{p \in \R[x]:\deg p \leq m\}$ defined by $L_s(x^j)=s_j$,  $j=0,\dots,m,$ and  $H_k(s)$, $2k\leq m,$ the Hankel matrix $H_k(s):=(s_{i+j})_{i,j=0}^k$. The shifted sequence $Es$ is  $Es:=(s_1,\dots,s_m)=(s_{j+1})_{j=0}^{m-1}.$Further, we define
\begin{align*}
 {\Pos}([a,b])_m:=\{ p\in \R[x]_m: p(x)\geq 0~~\text{ on }~~ [a,b]\}.
\end{align*}

%Recall that   $A\succeq 0$ means that the  matrix $A=A^T\in M_m(\R)$ is positive semidefinite. 

 The following notation differs between the two even and the odd cases  $m=2n$:
\begin{align*}
&\un{H}_{2n}(s)\index[sym]{HAnA@$\un{H}_{2n}(s),\ov{H}_{2n}(s),\un{H}_{2n+1}(s),\ov{H}_{2n+1}(s)$}:=H_n(s)\equiv (s_{i+j})_{i,j=0}^n, \\ &\ov{H}_{2n}(s):=H_{n-1}((b-E)(E-a))s)\equiv((a+b)s_{i+j+1}-s_{i+j+2}-abs_{i+j})_{i,j=0}^{n-1},\\
&\un{H}_{2n+1}(s):=H_n(Es-as)\equiv(s_{i+j+1}-as_{i+j})_{i,j=0}^n,\\
&\ov{H}_{2n+1}(s):=H_n(bs-Es)\equiv (bs_{i+j}-s_{i+j+1})_{i,j=0}^n,\\
&\un{D}_m(s):=\det \un{H}_m(s),\quad  \ov{D}_m(s):=\det \ov{H}_m(s).
\end{align*}
The upper and lower bar notation  allow us  to treat the even and odd
 cases at once. 

For  $f=\sum_{j=0}^k a_jx^j\in \R[x]_k$\, let $\vec{f}\index[sym]{fAvec@$\vec{f}$}:= (a_0,\dots,a_k)^T\in \R^{k+1}$  denote the coefficient vector of $f$. Then for $p,q\in \R[x]_n$ and $f,g\in \R[x]_{n-1}$  simple computations yield
\begin{align}\label{Lspghn2n}
L_s(pq)&=\vec{p}^T\, \un{H}_{2n}(s)\vec{q}, ~~~L_s((b-x)(a-x)fg)=\vec{f}^{\, T}\, \ov{H}_{2n}(s)\vec{g},\\
L_s((x-a&)pq)=\vec{p}^{\, T}\, \un{H}_{2n+1}(s)\vec{q}, ~~~~ L_s((b-x)pq)=\vec{p}^{\, T}\, \ov{H}_{2n+1}(s)\vec{q}.\label{Lspghn2n+1}
\end{align}

%From Proposition \ref{positivepolyr1} we restate the formulas (\ref{int3}) and (\ref{int4})  describing the  positive polynomials\, ${\Pos}([a,b])_m$ %on $[a,b]$ of degree at most $m$:Proposition \label{markov}
%\begin{align}
%&   {\Pos}([a,b])_{2n} =\big\{\, f+(b-x)(x-a)g:\, f\in \Sigma^2_n, \, g\in \Sigma_{n-1}^2 \big\}\label{intab3},\\
%&   {\Pos}([a,b])_{2n+1}=\big\{\, (b-x) f+(x-a)g:\, f,g\in \Sigma_n^2  \big\}\label{intab4}.
%\end{align}
Then   following two  theorems  settle the existence problem. The notation $A\succeq 0$ for a hermitian  matrix $A$ means that $A$ is positive semidefinite. 
\begin{theorem}\label{exittruncevenhaus} (Truncated\, $[a,b]$-moment problem; even case $m=2n$)\\  For a real sequence $s=(s_j)_{j=0}^{2n}$ the following statements are equivalent:
\begin{itemize}
\item[\em (i)]~ $s$ is a truncated $[a,b]$-moment sequence.
\item[\em (ii)]~ $L_s(p^2)\geq 0$\, and\, $L_s((b-x)(x-a)q^2)\geq 0$\, for\, $p\in \R[x]_n$ and\, $q\in \R[x]_{n-1}$.
\item[\em (iii)]~  $\un{H}_{2n}(s)\succeq 0$\, and\, $\ov{H}_{2n}(s)\succeq 0.$
\end{itemize}
 \end{theorem}
 
 \begin{theorem}\label{existoddhaus}\index{Truncated moment problem on a bounded interval! existence of a solution} (Truncated\, $[a,b]$-moment problem; odd case $m=2n+1$)\\  For a real sequence $s=(s_j)_{j=0}^{2n+1}$ the following are equivalent:
\begin{itemize}
\item[\em (i)]~  $s$ is a truncated $[a,b]$-moment sequence.
\item[\em (ii)]~  $L_s((x-a)p^2)\geq 0$\, and\, $L_s((b-x)p^2)\geq 0$\, for all\, $p\in \R[x]_n$.
\item[\em (iii)]~ $\un{H}_{2n+1}(s)\succeq 0$\, and\, $\ov{H}_{2n+1}(s)\succeq 0.$
\end{itemize}
\end{theorem}
\textit{Proofs of Theorems \ref{exittruncevenhaus} and \ref{existoddhaus}:}

(i)$\leftrightarrow$(ii): We  apply Proposition \ref{integralrepcompactcase} to the  subspace  $E=\R[x]_m$  of $C([a,b];\R)$. 
Then $L_s$ is a truncated $[a,b]$-moment functional if and only if $L_s(p)\geq 0$ for all  $p\in E_+= {\Pos}([a,b])_m.$ By  Proposition \ref{markov}  this is equivalent to condition (ii). 

(ii)$\leftrightarrow$(iii)\, follows at once from  the identities (\ref{Lspghn2n}) and (\ref{Lspghn2n+1}).$\hfill\qed$

\medskip

These  theorems give characterize truncated moment sequences  in terms of the positivity of two Hankel matrices. These are very useful criteria 
that can be verified by means of well-known criteria from linear algebra.

Note that  since\, $p^2 =(b-a)^{-1}[(b-x)p^2 +(x-a)p^2]$,\, condition (ii) in Theorem \ref{existoddhaus}  implies in particular that $L_s(p^2)\geq 0$ for  $p\in \R[x]_n.$

\section{The moment cone  and its boundary points}\label{momentcone}

Let $\cM_+$ denote the Radon measures on $[a,b]$.

\begin{definition}\label{momentconedim1}
 The \emph{moment cone}\index{Truncated moment problem on a bounded interval!  moment cone}\index{Moment cone} ${\cS}_{m+1}$\index[sym]{SCm@${\cS}_{m+1}$}\index[sym]{cBm@$\mathsf{c}_{m+1}$}   and the \emph{moment curve}\index{Truncated moment problem on a bounded interval! moment curve}\index{Moment curve} $\mathsf{c}_{m+1}$ are 
\begin{align*}
{\cS}_{m+1}&:=\{s=(s_0,s_1,\dots,s_m): s_j=\int_a^b t^j\, d\mu(t), j=0,\dots,m,\, \mu\in \cM_+\, \},\\ \mathsf{c}_{m+1}&=\{{\gs}(t):=(1,t,t^2,\dots,t^m):~ t\in [a,b]\, \}
\end{align*}
\end{definition}
If we identify the row vectors $s$ and ${\gs}(t)$ with the corresponding column vector $s^T$ and ${\gs}(t)^T$, then  become ${\cS}_{m+1}$ and  $\mathsf{c}_{m+1}$ as subsets of $\R^{m+1}$. The curve $\mathsf{c}_{m+1}$ is contained in ${\cS}_{m+1}$, since ${\gs}(t)$ is the moment sequence of the delta measure $\delta_t$. 
 It is not difficult to show that \emph{the moment cone ${\cS}_{m+1}$ is a closed convex cone in $\R^{m+1}$ with nonempty interior and that it is  the conic  hull\, $\cC_{m+1}$ of the moment curve $\mathsf{c}_{m+1}$.}

%Let $\partial {\cS}_{m+1}$ denote the set of boundary points and $\Int\, {\cS}_{m+1}$ the set ofinterior points of ${\cS}_{m+1}$. 

Each moment sequence $s\in{\cS}_{m+1}, s\neq 0$, has a  $k$-atomic representing measure 
\begin{align}\label{muatomicrepr}
\mu=\sum_{j=1}^k m_j\delta_{t_j},
\end{align}
where $k\leq m+1$ and $t_j\in [a,b]$ for all $j$. (This follows at once from the Richter-Tchakaloff theorem \ref{richter} proved in Lecture 8.)

 That  $\mu$ is $k$-atomic means that  the points $t_j$ are pairwise distinct and   $m_j>0$.
 The numbers $t_j$ are called \emph{roots}  or \emph{atoms} of $\mu$.
 Without loss of generality we assume
\begin{align}\label{orderroots}
a\leq t_1<t_2<\dots<t_k\leq b.
\end{align}

 To formulate our next theorem the following notion is convenient.

\begin{definition}\label{definitionindex}
 Let $s\in {\cS}_{m+1}, s\neq 0$. 
 The \emph{index} $\ind (\mu)$ of the $k$-atomic representing measure (\ref{muatomicrepr}) for $s$ is the sum  
\begin{align*}
 \ind (\mu)\index[sym]{mGu@$\ind (\mu)$}\index[sym]{eGpsilont@$\epsilon (t)$}:=\sum_{j=1}^k \epsilon(t_j),\quad \text{ where }~~ \epsilon (t):=2\quad \text{ for }~~ t\in (a,b)~~\text{ and }~~ \epsilon (a)=\epsilon (b):=1.
\end{align*}
The \emph{index}\index{Truncated moment problem on a bounded interval! index of representing measure}  $\ind (s)$ of $s$ is the minimal index of all representing measures (\ref{muatomicrepr}) for $s$.
\end{definition}
The reason why  boundary points and interior points are counted differently is  the following fact which is used in the  proofs:  If  $t_0\in [a,b]$ is a zero of $p\in {\Pos}([a,b])_m$ with multiplicity $k$, then  $k\geq 2$\, if\, $t_0\in (a,b)$, 
while $k=1$ is possible if $t_0=a,b$.

% For $s\in {\cS}_{m+1}$ let $\cM_s$  denote the set of all representing measures for $s\in {\cS}_{m+1}$.
%If the set $\cM_s$ is a singleton, then $s$ is called {\it $[a,b]$-determinate}.

The following theorem characterize the boundary points of the moment cone.
\begin{theorem}\label{boundarypointsmn+1}
For $s\in {\cS}_{m+1}$, $s\neq 0$, the following statements are equivalent:
\begin{itemize}
\item[\em (i)]~  $s$ is a boundary point\index{Truncated moment problem on a bounded interval! moment cone! boundary point} of the convex cone  ${\cS}_{m+1}$.
\item[\em (ii)]~  $\ind (s)\leq m$.
\item[\em (iii)]~ There exists a $p\in {\Pos}([a,b])_m, p\neq  0$, such that $L_s(p)=0$.
\item[\em (iv)]~ $\un{D}_m(s)=0$\, or\, ${\ov{D}}_m(s)=0.$
\item[\em (v)]~  $s$ has a unique representing measure $\mu\in \cM_+$.
%[a,b]$-determinate,\index{Truncated moment problem on a bounded interval! determinacy}  that is, $\cM_s$  is a singleton.
\end{itemize}
If $p$ is as in (iii) and $\mu$ is as in (v), then\, $\supp\, \mu \subseteq \{t\in [a,b]:p(t)=0\}$.
\end{theorem}
\begin{proof}
 {[MP, Theorem~ 10.7]}.
%\cite[Theorem 10.7]{mp}.
$\hfill \qed$
\end{proof}

This theorem shows that several important properties of $s$ are equivalent.
In particular, the last assertion is crucial: The atoms of $\mu$ are contained in the zero set of the polynomial $p\in {\Pos}([a,b])_m$ satisfying $L_s(p)=0$. 

By (iii),   the unique representing measure  $\mu$ has    $\ind\, (\mu)=\ind (s)\leq m$. Hence, if all $t_j$ are in the open interval $(a,b)$, then $k\leq \frac{m}{2}$. If precisely one $t_j$ is an end point, then $k\leq \frac{m+1}{2}$ and if both end points are among the $t_j$, then $k\leq \frac{m}{2}+1$. The case $k=\frac{m}{2}+1$ can only happen if $m$ is even and both $a$ and $b$ are among the $t_j$. Thus, all boundary points of ${\cS}_{m+1}$ can be represented by $k$-atomic measures with $k\leq\frac{m}{2}+1.$

\section{Interior points and principal measures}\label{interiorpoints}

 Now we consider the interior points of the moment cone.
 Since $s\in {\cS}_{m+1}$ is an interior point if and only if $s$ is \emph{not} a boundary point, Theorem \ref{boundarypointsmn+1} yields the following:
 $s\in {\cS}_{m+1}$ is an interior point of  ${\cS}_{m+1}$ if and only if the Hankel matrices $\un{H}_m(s)$\, and\, $\ov{H}_m(s)$ are positive definite, or equivalently,
$\un{D}_m(s)>0$\, and\, $\ov{D}_m(s)>0$

 Throughout this section, we suppose that \emph{$s$ is an interior point of}  ${\cS}_{m+1}$.

Then, by Theorem \ref{boundarypointsmn+1},  we have $\ind\, ( s)\geq m+1$.
Thus, it is natural to ask whether there are representing measure which  have  the minimal possible index $m+1$ and  to describe these measures provided they exists.
\begin{definition} 
A  representing measure $\mu$ of the form (\ref{muatomicrepr}) for $s$  is called\\ $\bullet$~  \emph{principal}\index{Truncated moment problem on a  bounded interval! principal measure}  if\, $\ind\, (\mu)=m+1$,\\
$\bullet$~ \emph{upper principal} if it is principal and $b$ is an atom of $\mu$,\\$\bullet$~ \emph{lower principal} if it is principal and $b$ is not an atom of $\mu$.
\end{definition}
Thus, for principal measures the index is equal to the number of prescribed moments.
Then we have  following existence theorem for such measures.
\begin{theorem}\label{principallowerupper}
Each  interior point $s$ of ${\cS}_{m+1}$ has  a unique upper principal representing measure $\mu^+$ and a  unique lower  principal representing measure $\mu^-$.
% The roots isof $\mu^+$ and $\mu^-$ are stricly interlacing.
\end{theorem}
\begin{proof}
%\cite[Theorem 10.17]{mp}
 {[MP, Theorem~ 10.17]}.
 $\hfill\qed$
\end{proof}

Let $t^\pm_j$ denote the roots of  the principal measures $\mu^\pm$. The location of these roots 
$t^\pm_j$   in the even  and  odd cases are illustrated 
by the following scheme:
\begin{align*}
 m=2n,\quad\quad\quad~~\mu^+: &~~ a<t_1^+<t_2^+<\dots<t_{n+1}^+=b,\\ m=2n,\quad\quad\quad ~~\mu^-:&~~ a=t_1^-<t_2^-<\dots<t_{n+1}^-<b,\\
 m=2n+1,\quad\quad \mu^+: &~~ a=t_1^+<t_2^+<\dots<t_{n+1}^+<t^+_{n+2}=b,\\ m=2n+1,\quad\quad
  \mu^-: &~~a<t_1^-<t_2^-<\dots<t_{n+1}^-<b.
\end{align*}
Further, we have
\begin{align*}
 m=2n:\quad\quad~~ & a=t_1^-<t_1^+<t_2^-<t_2^+< \dots <t_n^+<t_{n+1}^-<t_{n+1}^+=b,\\
 m=2n+1:\quad  & a=t_1^+<t_1^-<t_2^+<t_2^-<\dots<t_{n+1}^+<t_{n+1}^-<t^+_{n+2}=b.
\end{align*}
These formulas show that the roots of $\mu^+$ and $\mu^-$ are strictly interlacing.
\smallskip 
\begin{definition}
A representing measure $\mu$ of  the form  (\ref{muatomicrepr}) for $s$ is called \emph{canonical} if\, $\ind (\mu) \leq m+2$.
\end{definition}
While there exist precisely two principal measures by Theorem \ref{principallowerupper}, there is a one-parameter family of canonical measures, as the following theorem shows.
\begin{theorem}
For each point $\xi\in (a,b)$ there exists a unique canonical representing measure $\mu_\xi$ of $s$ which has $\xi$ as an atom.
\end{theorem}
\begin{proof}
%\cite[Corollary 10.13]{mp}
 {[MP, Corollary~10.13]}.
$\hfill \qed$
\end{proof}

A deeper study of the set  of representing measures can be found in the books by Karlin and Studden (1966) and by Krein and Nudelman (1977).

%

%\input{Hanoi-Lecture.tex}
%\input{Hanoi-Summary.tex}
%\input{Lecture1}
%\input{Lecture2}
%\input{Lecture3}
%\input{Lecture4}
%\input{Lecture5}
%\input{Lecture6}
%\input{Lecture7}
%\input{Lecture8}
%\input{Lecture9}
%\input{Lecture10}
%\input{ch1integralrepn.tex}

%\include{Part0}
%\include{PartI}
%\include{PartII}

%

%\include{PartIII}
%

%\include{PartIV}

%\include{App+Bibl}

%\backmatter          

%\addcontentsline{toc}{chapter}{Symbol index}
%\printindex[sym]
%\addcontentsline{toc}{chapter}{Index}
%%%%%%%%%%%%%%%%%%%%%%%%%%%%%%%%%%%%%%%%%%%%%%%%%%%%%%%%%%%%%%%%%%%%%%

%\printindex

\makeatletter
\renewcommand{\@chapapp}{Lecture}
\makeatother
%\chapter{ }
%\newpage
%\setcounter{page}{1}

{\chapter{The moment problem on compact semi-algebraic sets}}
\noindent
Abstract:\\
\textit{The moment problem for compact semi-algebraic sets is investigated and solved. The interplay between 
 moment problem and Positivstellens\"atze of real algebraic geometry is discussed.}
\bigskip

In this Lecture we enter the  multidimensional moment problem by treating the moment problem for compact subsets of $\R^d$ that are defined by means of finitely many polynomial inequalities. In Lecture 1 we have seen how the solvability criteria  for intervals have been derived from descriptions of positive polynomials. This suggests that real algebraic geometry and Positivstellens\"atze might be  useful tools for the  multidimensional moment problem. For compact semi-algebraic sets this is indeed true and will be elaborated in this Lecture.

We  begin by reviewing some  concepts and results from real algebraic geometry. References are the excellent books by Marshall (2008) and Prestel and Delzell (2001).
\bigskip

\section{Basic notions from real algebraic geometry}

%\section{Semi-algebraic sets and the Krivine-Stengle Positivstellensatz}\label{basicssemialgebraicsets}

 In this section, $\sA$ denotes a \textbf{unital commutative real algebra}.
 The reader might always think of the polynomial algebra $\R_d[\ux]\equiv \R[x_1,\dots,x_d]$. 

Positivity in $\sA$ is described by means of the following notions.

\begin{definition}
 A \emph{ quadratic module}\index{Quadratic module} of\, $\sA$ is a subset $Q$ of $\sA$ such that 
\begin{align}\label{axiomquadmodule}
Q + Q \subseteq Q,~~ 1 \in Q,~~
a^2 Q  \in Q~~\text{ for all }~~a\in  \sA.
\end{align} 
 A quadratic module $T$ is called a {\em preordering}\index{Preordering} if\, $T\cdot T\subseteq T$. 
%\\  A {\em semiring}\index{Semiring}\index{Preprime} is a subset $S$ of ${\sA}$ satisfying
%\begin{align}
%S + S \subseteq S,~~
%S\cdot S  \subseteq S,~~ \lambda \in S ~~{\rm for~all}~~\lambda\in \R,\lambda \geq 0.
%\end{align}
%A {\em cone}\index{Cone} is a subset $C$ of ${\sA}$ such that\, $C+C\subseteq C$\, and\, $\lambda \cdot C\subseteq C$\, for $\lambda\geq 0$.
\end{definition}

\begin{example}
 Obviously, $Q=\sum \R_d[\ux]^2+x_1\sum \R_d[\ux]^2+\dotsb+ x_d\sum \R_d[\ux]^2$ is a quadratic module of $\R_d[\ux]$.
 If  $d\geq 2$, then product of elements of $Q$ are not in $Q$ in general, so $Q$ is not a preordering. \hfill $\circ$ 
\end{example}

%Each quadratic module $Q$ yields an ordering $\preceq$ on  $\R_d[\ux]$\,  by defining 
%\begin{align*}a \preceq b \quad {\rm 
%if~ and~ only~ if}\quad b-a \in Q.
%\end{align*}
Obviously, the set $\sum  \sA^2$ of all finite sums of squares of elements $a\in \sA$ is the smallest quadratic module and also the smallest  preordering of $\sA$.

 Let $\mathsf{f}=\{f_1,\dots,f_k\}$ be a finite subset of $\R_d[\ux]$. The set
\begin{align}\label{basicclosed}
K(\mathsf{f})\equiv K(f_1,\dots,f_k)\index[sym]{KCfB@$K(\mathsf{f}), K(f_1,\dots,f_k)$}=\{x\in \R^d: f_1(x)\geq 0,\dots,f_k(x)\geq 0\}
\end{align}
 is called the \emph{basic closed semi-algebraic set associated with $\mathsf{f}$}.\index{Basic closed semi-algebraic set}\index{Semi-algebraic set} It is easily seen that  
\begin{align}\label{quadraticqf}
Q(\mathsf{f})\equiv Q(f_1,\dots,f_k)\index[sym]{QAfB@$Q(\mathsf{f}), Q(f_1,\dots,f_k)$}= \big\{\, \sigma_0+ f_1 \sigma_1+\dots+f_k \sigma_k : \,\sigma_0,\dots,\sigma_k\in \sum\R_d[\ux]^2\big\}
\end{align}
 is the \emph{quadratic module generated by the set} $\mathsf{f}$ and that
\begin{align}\label{preorderingtf}
T(\mathsf{f})\equiv T(f_1,\dots,f_k)\index[sym]{TAfB@$T(\mathsf{f}), T(f_1,\dots,f_k)$}=\bigg\{  \sum_{e=(e_1,\dots,e_k)\in \{0,1\}^k} f_1^{e_1}\cdots f_k^{e_k} \sigma_e:\, \sigma_e\in \sum\R_d[\ux]^2 \, \bigg\}
\end{align}
is the {\em preordering generated by the set $\mathsf{f}$}. Obviously, $Q(\mathsf{f})\subseteq T(\mathsf{f})$
and all polynomials from $T(\mathsf{f})$ are nonnegative on $K(\mathsf{f})$.
%, but in general $T(\mathsf{f})$ does not exhaust the nonnegative polynomials on $K(\mathsf{f})$. 

 The following \textbf{Positivstellensatz of Krivine--Stengle} is a central results of real algebraic geometry.
 It describes nonnegative resp. positive polynomials on $K(\mathsf{f})$ in terms of \emph{quotients} of elements of  $T(\mathsf{f})$. 
\begin{theorem}\label{krivinestengle}\index{Theorem! Krivine--Stengle}   Let $K(\mathsf{f})$ and $T(\mathsf{f})$ be as above and let $g\in \R_d[\ux]$. Then we have:
\begin{itemize}
\item[\textnormal{(i)}]\emph{(Positivstellensatz)}~ $g(x)>0$  for all\, $x\in K(\mathsf{f})$\, if and only if there exist polynomials $p,q\in T(\mathsf{f})$  such that $pg=1+q$.
\item[\textnormal{(ii)}]\emph{(Nichtnegativstellensatz)} $g(x)\geq 0$ for all $x\in K(\mathsf{f})$ if and only if there exist $p,q\in T(\mathsf{f})$ and $m\in \N$ such that $pg=g^{2m}+q$.
%\item[\em (iii)]~ {\rm (Nullstellensatz)} $g(x)= 0$ for  $x\in K(\mathsf{f})$ if and only if $-g^{2n}\in %T(\mathsf{f})$ for some $n\in \N$.
%\item[\em (iv)]~ $K(\mathsf{f})$ is empty if and only if~ $-1$ belongs to $T(\mathsf{f}).$
\end{itemize}
\end{theorem}
%\begin{proof} 
%See \cite{presteld} or \cite{marshall}. The orginal papers are \cite{krivine1} and \cite{stengle74}.
%\qed \end{proof}  

 The \emph{``if''} assertions are easily checked.
 Proofs  of the \emph{``only\, if''} directions in Theorem \ref{krivinestengle} are given in the books by Marshall or by Prestel and Delzell; they are based on the Tarski--Seidenberg transfer  principle.

Another important concept is introduced in the following definition.
\begin{definition}
 A  quadratic module $Q$ of  $\sA$ is called \emph{Archimedean} if $\sA$ coincides with the set $ \sA_\mathrm{b}$ of bounded elements with respect to $Q$, where
\begin{align*}
 \sA_\mathrm{b}
 :=\{ a\in \sA:\text{there exists a $\lambda >0$ such that $\lambda - a \in Q$ and $\lambda +a\in Q$}\}.
\end{align*}
\end{definition}
%We derive some simple facts on this notion.
\begin{lemma}\label{boundedele1} Let $Q$ be a quadratic module of\, $\sA$\, and\, $a\in\sA$. Then we have 
 $a\in \sA_\mathrm{b}$\,  if and only if\,   $\lambda^2 -a^2\in Q$ for some $\lambda >0$.
\end{lemma}
\begin{proof}
 If $\lambda \pm a\in Q$ for $\lambda>0$, then  
\begin{align*}
\lambda^2 -a^2= \frac{1}{2\lambda}\big[(\lambda+a)^2(\lambda -a) +(\lambda -a)^2(\lambda + a)\big] \in Q.\end{align*}
Conversely, if $\lambda^2-a^2 \in Q$ and $\lambda >0$, then
\begin{align*}\hspace{2,7cm}
\lambda\pm a=\frac{1}{2\lambda}\big[(\lambda^2-a^2) +(\lambda\pm a)^2 \big] \in Q.
\hspace{3cm} \Box
\end{align*}

\begin{lemma}\label{boundedele2}
 $\sA_\mathrm{b}$ is a unital  subalgebra of $\sA$ for any quadratic module $Q$.
\end{lemma}
\begin{proof} 
Clearly,  sums  and scalar multiples of elements of $\sA_\mathrm{b}$ are again in $\sA_\mathrm{b}$. It suffices to verify that this holds for the product of  $a,b \in \sA_\mathrm{b}$.
By Lemma \ref{boundedele1},   there are $\lambda_1 >0$ and $\lambda_2>0$
such that $\lambda_1^2-a^2$ and $\lambda_2^2-b^2$ are in $Q$. Then
\begin{align*}
(\lambda_1\lambda_2)^2 - (ab)^2= \lambda_2^2(\lambda_1^2-a^2)+ a^2(\lambda_2^2-b^2) \in Q,\end{align*}
so that $ab\in \sA_\mathrm{b}$ again by  Lemma \ref{boundedele1}.
\qed \end{proof}

\begin{corollary}\label{archirux}
For a quadratic module $Q$  of\, $\R_d[\ux]$ the following are equivalent:
\begin{itemize}
\item[\em (i)]~ $Q$ is Archimedean.
\item[\em (ii)]~\,  There exists a number $\lambda >0$ such that $\lambda -\sum_{k=1}^d x_k^2\in Q$.
\item[\em (iii)]~\,  For any $k=1,\dots,d$ there exists a  $\lambda_k>0$ such that $\lambda_k-x_k^2\in Q$.
\end{itemize}
\end{corollary}
\begin{proof}
(i)$\to$(ii) is clear by definition. If $\lambda -\sum_{j=1}^d x_j^2\in Q$, then \begin{align*}\lambda- x_k^2=\lambda -\sum\nolimits_j x_j^2~ +~ \sum\nolimits_{j\neq k}x_j^2 \in Q.\end{align*} This  proves (ii)$\to$(iii). Finally, if (iii) holds,  then $x_k\in \R_d[\ux]_\mathrm{b}$ by Lemma \ref{boundedele1} and hence $\R_d[\ux]_\mathrm{b}=\R_d[\ux]$ by Lemma \ref{boundedele2}. Thus, (iii)$\to$(i).
\qed \end{proof}  

\begin{corollary}\label{archicompact}
If the quadratic module $Q(\mathsf{f})$ of $\R_d[\ux]$ is Archimedean, then the set $K(\mathsf{f})$ is compact.
\end{corollary}
\begin{proof}
By Corollary \ref{archirux},   $\lambda -\sum_{k=1}^d x_k^2\in Q(\mathsf{f})$ for some $\lambda >0$. Therefore, since polynomials of  $Q(\mathsf{f})$ are nonnegative on $K(\mathsf{f})$, the set $K(\mathsf{f})$ is compact. 
\qed \end{proof}  

The converse of Corollary \ref{archicompact}  does not hold.
However, it does hold for the preordering $T(\mathsf{f})$, as shown by Proposition \ref{prearchcom} below.

The following  separation result is used  in the next section.
 \begin{proposition}\label{eidelheit}
Let $Q$ be an Archimedean quadratic module of ${\sA}$. If $a_0\in {\sA}$ and $a_0\notin Q$, there exists a $Q$-positive linear functional $\varphi$ on ${\sA}$ such that $\varphi(1)=1$ and $\varphi(a_0)\leq 0$.
\end{proposition}
\begin{proof}
Let $a\in {\sA}$ and choose $\lambda >0$ such that\, $\lambda \pm a\in Q$. If\,  $0<\delta \leq\lambda^{-1}$,\, then $\delta^{-1} \pm a\in Q$ and hence $1\pm \delta a \in Q$. This shows that   $1$ is an internal point of $Q$. Therefore,  Eidelheit's separation theorem  applies and there exists a $Q$-positive linear functional $\psi\neq 0$ on ${\sA}$ such that $\psi(a_0)\leq 0$. Since $\psi\neq 0$, we have  $\psi(1)> 0$. (Indeed, if $\psi(1)=0$, since $\psi$ is $Q$-positive,  $\lambda\pm a \in Q$ implies $\psi(a)=0$ for all $a\in {\sA}$ and so $\psi=0$.) Then $\varphi:=\psi(1)^{-1}\psi$ has the desired properties.
\qed \end{proof}  
%
%\begin{example}
%Let ${\sA}=\R_d[\ux]$ and let $K$ be a closed subset of $\R^d$. If $Q$ is the preordering ${\Pos} (K)$ of %nonnegative polynomials on $K$, then ${\sA}_b(Q)$ is just the set of bounded polynomials on $K$. Hence $Q%$ is Archimedean if and only if $K$ is compact. $\hfill \circ$
%\end{example}
\section{Strict Positivstellensatz and solution of the moment problem}

 Throughout this section,  $\mathsf{f}=\{f_1,\dots,f_k\}$ denotes a \emph{finite} subset of $\R_d[\ux]$ such that the \textbf{semi-algebraic set $K(\mathsf{f})$ is compact.}

 The following  is the \textbf{strict Positivstellensatz} for compact semi-algebraic sets.
 
\begin{theorem}\label{schmps}
 Let\, $h\in \R_d[\ux]$.
 If  $ h(x) > 0$ for all $x \in K(\mathsf{f})$, then $h \in T(\mathsf{f})$.
\end{theorem}

 The next theorem contains  the \textbf{solution of the moment problem} for compact semi-algebraic sets.
 
\begin{theorem}\label{mpschm}
 Let  $L$ be a linear functional on $\R_d[\ux]$. If $L$ is $T(\mathsf{f})$-positive, then $L$ is a $K(\mathsf{f})$-moment functional.
\end{theorem}

 We next show that \emph{both theorems are equivalent and can be derived from each other}.
 This emphasizes that there is a close interplay between real algebraic geometry and moment theory.
 
\secret{
 The crucial technical ingredient of the proof of Theorem \ref{schmps} is Proposition \ref{prearchcom} below. It states that the preorder $T(\mathsf{f})$ is Archimedean if the semi-algebraic sets $K(\mathsf{f})$ is compact. In  the proof of this result  we apply  assertion (i) of the Krivine-Stengle's Positivstellensatz. Using  Proposition \ref{prearchcom} we next show that {\it both theorems are equivalent and can be derived from each other}. All that emphasizes the close interplay between real algebraic geometry and the  moment problem.
 
 The proof of Theorem \ref{mpschm} is given by Proposition \ref{continuityLprop}(ii) in Section \ref{proof}.
}

\medskip  

\textit{Proof of Theorem \ref{mpschm} (assuming Theorem \ref{schmps}):}

 Let $h\in \R_d[\ux]$ and suppose $h(x)\geq 0$ on $K(\mathsf{f})$.
 Then, for any $\varepsilon>0$, $h(x)+\varepsilon >0$ on $K(\mathsf{f})$ and therefore $h\in T(\mathsf{f})$ by  Theorem \ref{schmps}.
 Hence  $L(h+\varepsilon)=L(h)+\varepsilon L(1)\geq 0$ by the assumption.
 Then $L(h)\geq 0$ by letting $\varepsilon\to 0$.
 Therefore, $L$  is a $K(\mathsf{f})$-moment functional by  Haviland's Theorem.$\hfill\qed$

\medskip
%\ref{haviland}.

\textit{Proof of Theorem \ref{schmps} (assuming Theorem \ref{mpschm} and Proposition \ref{prearchcom}):}

Suppose   $h\in \R_d[\ux]$ and $h(x)>0$ on $K(\mathsf{f})$. Assume to the contrary that $h\notin T(\mathsf{f})$. Since the preordering  $T(\mathsf{f})$ is Archimedean by Proposition \ref{prearchcom},  Proposition \ref{eidelheit} applies, so  there  is a $T(\mathsf{f})$-positive linear functional $L$ on $\R_d[\ux]$ such that $L(1)=1$ and $L(h)\leq 0$. By Theorem \ref{mpschm}, $L$ is a $K(\mathsf{f})$-moment functional, that is, there is a Radon measure $\mu$ supported on $K(\mathsf{f})$ such that $L(p)=\int_{K(\mathsf{f})} p\, d\mu$  for $p\in \R_d[\ux]$. But $L(1)=\mu(K(\mathsf{f}))=1$ and $h>0$ on $K(\mathsf{f})$ imply that $L(h)>0$. This is a contradiction, since $L(h)\leq 0$.$\hfill\qed$

\medskip

\secret{

\begin{remark}
The preordering $T(\mathsf{f})$\, is defined as the sum of  sets $f_1^{e_1}\cdots f_k^{e_k}\cdot\sum \R_d[\ux]^2$. It is natural to ask whether or not  all such sets with mixed products $f_1^{e_1}\cdots f_k^{e_k}$ are really needed. To formulate the corresponding result we put $l_k:=2^{k-1}$ and let
 $g_1,\dots,g_{l_k}$ denote the first $l_k$ polynomials of the following row of mixed  products:
 \begin{align*}
 f_1,\dots,f_k,f_1f_2,f_1f_3,\dots,f_1f_k,\dots,f_{k-1}f_k,f_1f_2f_3,\dots, f_{k-2}f_{k-1}f_k,\dots,f_1f_2\dots,f_k.
 \end{align*}
Let $Q(\mathsf{g})$ denote the quadratic module generated by $g_1,\dots,g_{l_k}$, that is,
 \begin{align*}Q(\mathsf{g}):=\sum\R_d[\ux]^2+ g_1\sum\R_d[\ux]^2+\dots+g_{l_k}\sum\R_d[\ux]^2.\end{align*} 
Then, as shown  by T. Jacobi and A. Prestel (2001),
% \cite{japr}, 
the assertions of  Theorems \ref{mpschm} and \ref{schmps} remain valid if the preordering $T(\mathsf{f})$ is replaced by $Q(\mathsf{g})$.

We briefly discuss the relations between $T(\mathsf{f})$ and $Q(\mathsf{g})$. If $k=1$, then  $Q(\mathsf{f})=T(\mathsf{f})$. However, for $k=2$,
\begin{align*}
Q(\mathsf{f})&=\sum \R_d[\ux]^2+f_1\sum \R_d[\ux]^2+f_2\sum \R_d[\ux]^2,
\end{align*}
so $Q(\mathsf{g})$ differs $T(\mathsf{f})$ by the summand $f_1f_2\sum \R_d[\ux]^2$. If $k=3$, then
\begin{align*}
Q(\mathsf{g})
%=\cS_\mathsf{f}
=\sum \R_d[\ux]^2+&f_1\sum \R_d[\ux]^2+ f_2\sum \R_d[\ux]^2+f_3\sum \R_d[\ux]^2+f_1f_2\sum \R_d[\ux]^2 \,,
\end{align*} 
that is, the sets 
$h\sum \R_d[\ux]^2$ with $h=f_1f_3,f_2f_3,f_1f_2f_3$  do not enter into $Q(\mathsf{g})$. 
$\hfill \circ$ \end{remark}
}

 The proof of Theorem \ref{mpschm} is given by Proposition \ref{continuityLprop}(ii) in Section \ref{proof}.
 The crucial technical ingredient of this proof is Proposition \ref{prearchcom} which states that the preorder $T(\mathsf{f})$ is Archimedean if the semi-algebraic sets $K(\mathsf{f})$ is compact. In  the proof of this result   assertion (i) of the Krivine-Stengle's Positivstellensatz is used. 
 
 If we plug in these proofs and \emph{assume} (!) the Archimedean property, then the assertions hold for a quadratic module rather than the preordering.
 The gives the following theorem.
 Assertion (i) is usually the \textbf{Archimedean Positivstellensatz}.

%\section{The moment problem for Archimedean modules}\label{reparchimodiules}

%The assertions of both theorems from the preceding section remain valid if we assume the Archimedean property of the quadratic module $Q(\mathsf{f})$ %instead of the compactness of the semi-algebraic set
%$ K(\mathsf{f})$.

\begin{theorem}\label{archmedps}
 Suppose the quadratic module $Q(\mathsf{f})$ defined by \eqref{quadraticqf} is Archimedean. 
\begin{enumerate}
 \item[\em (i)]\, If $h\in \R_d[\ux]$ satisfies $h(x)> 0$ for all $x\in K(\mathsf{f})$, then $h\in Q(\mathsf{f}).$
 \item[\em (ii)]\, Each  $Q(\mathsf{f})$-positive linear functional $L$ on\, $\R_d[\ux]$ is a $K(\mathsf{f})$-moment functional.
\end{enumerate}
\end{theorem}

%Assertion (i) is usually the {\it Archimedean Positivstellensatz}.

 The shortest and probably the most elegant approach to Theorems \ref{archmedps} and \ref{mpschm} is based on the multi-dimensional spectral theorem combined with the GNS construction; it is developed in [MP, Section 12.5].
%\cite[Section 12.5]{mp}.

We close the section with two examples.

\begin{example}\label{n-dimintervalexist}\textit{($d$-dimensional compact interval~ $[a_1,b_1]\times\dots \times [a_d,b_d]$)}\\
Let  $a_j,b_j\in \R$, $a_j< b_j,$ and set $f_{2j-1}:=b_j-x_j$, $f_{2j}:=x_j-a_j,$ for $j=1,\dots,d$. The
semi-algebraic set $K(\mathsf{f})$ is the $d$-dimensional interval $[a_1,b_1]\times\dots \times [a_d,b_d]$ and
$Q(\mathsf{f})$ is Archimedean by Lemma \ref{boundedele2}(ii). Hence, by Theorem \ref{archmedps}(ii),

 \textit{$L$  is a  $K(\mathsf{f})$-moment functional if and only if it is $Q(f)$-positive, or equivalently, if\,  $ L_{f_1},L_{f_2}, \dots, L_{f_k}$  are positive  functionals, that is,}
\begin{align*} \hspace{0.4cm}
 L((b_j{-}x_j)p^2)\geq 0~~ \text{\textit{and}} ~~ L((x_j{-}a_j)p^2)\geq 0~~ \text{\textit{for}} ~~ j=1,\dots,d,\, p\in \R_d[\ux]. \hspace{1cm}\circ
\end{align*}
\end{example}

An immediate corollary of Theorem \ref{mpschm} is the following.

\begin{corollary}\label{compactalgvar}
Suppose that $\cI$ is an ideal of $\R_d[\ux]$ such that the real algebraic set\, $V:=\cZ(\cI)=\{x\in \R^d:f(x)=0~~\text{for}~ f\in \cI\}$ is compact.
Then  each positive linear functional  on $\R_d[\ux]$ which annihilates $\cI$ is a $V$-moment functional.
\end{corollary}

\begin{example} (\textit{Moment problem on unit spheres})\\
 Let $S^{d-1}=\{x\in \R^d:x_1^2+\dots+x_d^2=1\}$.
 Then  a linear functional $L$  on $\R_d[\ux]$  is a $S^{d-1}$-moment functional if and only if
\begin{align*}
 \hspace{1cm}L(p^2)\geq 0 \quad \text{ and }~~ L((x_1^2+\dots+x_d^2-1)p)=0 \quad \text{ for }~~ p\in \R_d[\ux].\hspace{1.3cm} \circ
\end{align*} 
\end{example}

\section{Localizing functionals and Hankel matrices}

The preceding result on the moment problem suggest the following question:
\smallskip

\textit{How to verify the $Q(\mathsf{f})$-positivity  or  $T(\mathsf{f})$-positivity of a functional $L$ on $\R_d[\ux]$?}
\smallskip

Let $L$ be a  linear functional on  $\R_d[\ux]$ and  $g\in\R_d[\ux]$. The  linear functional $L_g$\index[sym]{LAgA@$L_g$}   defined by $L_g(p)=L(gp),\,  p\in\R_d[\ux]$, is called the \emph{localization} of $L$ at $g$. 

For instance, suppose  $L$ is a $K(\mathsf{f})$ moment functional and $g(x)\geq 0$ on $K(\mathsf{f})$. If $\mu$ is a representing measure for $L$, then $L_g(p)=\int p(x)g(x)d\mu$, so\, $gd\mu$ is a representing measure for $L_g$. This justifies the name "localization".
\smallskip

  Let $s=(s_\alpha)_{\alpha \in \N_0^d}$ be the $d$-sequence given by  $s_\alpha=L(x^\alpha)$ and  write $g=\sum_\gamma g_\gamma x^\gamma$. We define a $d$-sequence $g(E)s=((g(E)s)_\alpha)_{\alpha\in \N_0^d}$ by 
\begin{align*}(g(E)s)_\alpha:=\sum\nolimits_\gamma ~ g_\gamma s_{\alpha+\gamma},~~\alpha\in \N_0^d,\end{align*}  
 and the \emph{localized Hankel matrix}\, $H(gs)=(H(gs)_{\alpha,\beta})_{\alpha,\beta\in \N_0^d} $\,  with entries 
\begin{align*}
H(gs)_{\alpha,\beta}:= \sum\nolimits_\gamma ~ g_\gamma  s_{\alpha+\beta+\gamma},~~\alpha,\beta\in \N_0^d.
\end{align*}
Then,
 for  $p(x)=\sum_\alpha
a_\alpha x^{\alpha}\in \R_d[\ux]$, we obtain
\begin{align*}
L_s(gp^2)=\sum_{\alpha,\beta} a_\alpha a_\beta (g(E)s)_{\alpha+\beta}=\sum_{\alpha,\beta}\, a_\alpha a_\beta H(gs)_{\alpha,\beta} .
\end{align*}
By the preceding, $g(E)s$ is the\, sequence for the functional $L_g$ and  $H(gs)$ is a Hankel matrix for the sequence $g(E)s$. From these formulas we easily derive the following.

\begin{proposition}\label{equvialentsolvmp}
 Let $Q(\mathsf{g})$ be  the quadratic module generated by $\mathsf{g}=\{g_1,\dotsc,g_m\}$.
 For a linear functional $L$ on $\R_d[\ux]$ the following are equivalent:
\begin{enumerate}
 \item[\em (i)]~ $L$ is a $Q(\mathsf{g})$-positive linear functional on $\R_d[\ux]$.
 \item[\em (ii)]~  $L, L_{g_1},\dotsc,L_{g_m}$ are positive linear  functionals on $\R_d[\ux]$.
 \item[\em (iii)]~ $s, g_1(E)s,\dotsc,g_m(E)s$ are positive semidefinite $d$-sequences.
 \item[\em (iv)]~ $H(s),H(g_1s),\dotsc, H(g_ms)$ are positive semidefinite matrices.
\end{enumerate}
\end{proposition}

Condition (iv) can be verified by using standard criteria from linear algebra. Thus, the $Q(\mathsf{g})$-positivity or $T(\mathsf{g})$-positivity of a functional  are reasonable and familiar conditions in moment theory.

\section{Moment problem criteria based on semirings}
\begin{definition}
 A \emph{semiring} of $\sA$ is a subset $S$ satisfying
\begin{align*}
 S+S\subseteq S,~~ S\cdot S\subset S,~~ \lambda\cdot 1\in S\quad \text{ for }~~\lambda\geq 0.
\end{align*}
\end{definition}
A quadratic module is not necessarily invariant under multiplication, but a semiring is. While  a quadratic module contains all squares, a semiring does not in general. Clearly, a quadratic module is a preordering if and only if it is a semiring.

 The \emph{semiring} $S(\mathsf{f})$ generated by $f_1,\dots,f_k\in \sA$ is the set of all finite sums of terms
\begin{align}\label{semringgene}
 \lambda f_1^{n_1}\cdots f_k^{n_k},\quad \text{ where }~~n_1,\dots,n_k\in \N_0,~ \lambda  \geq 0.
\end{align}
 The next theorem applies to semi-algebraic sets contained in \textbf{compact polyhedra}.

 Let $f_1,\dots,f_k\in\R_d[\ux]$ such that  $f_1,\dots,f_m$, $m\leq k$, are \emph{linear}.
 Setting
\begin{align*}
\mathsf{\hat{f}}=\{f_1,\dots,f_m\},\quad \mathsf{f}=\{f_1,\dots,f_k\}.
\end{align*}
 Then $K(\mathsf{\hat{f}})$ is  a \emph{polyhedron}  such that\,  $K(\mathsf{f})\subseteq K(\mathsf{\hat{f}}).$

\begin{theorem}\label{prestel}
 Suppose  the polyhedron\, $K(\, \mathsf{\hat{f}}\, )$\, is  compact and  nonempty.
 Then a linear functional $L$ on $\R_d[\ux]$ is a $K( \mathsf{f} )$-moment functional if and only if 
\begin{align}\label{solvconsemiring}
 L( f_1^{n_1}\cdots f_k^{n_k})\geq 0\,  \quad \text{for~ all}~~~ n_1,\dots,n_k\in \N_0.
\end{align}
\end{theorem} 
\begin{proof}
 {[MP, Theorem~12.45]}.
 $\hfill \qed$
\end{proof}
 
Condition (\ref{solvconsemiring}) means that $L$ is nonnegative on the semiring $S(\mathsf{f})$ defined by (\ref{semringgene}).
Note that  the criterion in Theorem 1.17
%\ref{applbernstein} 
in Lecture 1 was  also a positivity condition for a semiring.
 Finally, we state three examples based on Theorems  \ref{archmedps} and \ref{prestel}.
%First we discuss an some examples based on Theorems \ref{mpschm} or  \ref{archmedps}.
\medskip

%The following examples use Theorem \ref{prestel}

\begin{example} (\textit{Simplex  in $\R^d, d\geq 2$})\\
Let $f_1=x_1,\dots,f_d=x_d, f_{d+1}=1- \sum_{i=1}^d x_i$. The set  $K(\mathsf{f})$ is the simplex
\begin{align*}
K_d=\{ x\in \R^d: x_1\geq 0,\dots, x_d\geq 0,\, x_1+\dots+x_d\leq 1\, \}.
\end{align*}
\textit{A linear functional $L$ is a\, $K_d$--moment functional if and only if}
\[
 L(x_ip^2)\geq 0, ~ i=1,\cdots,d,~~\text{\textit{and}}~~~ L((1- ( x_1+x_2+\dots+x_d))p^2)\geq 0 ~~~\text{\textit{for}}~~p \in \R_d[\ux],  
\]
\textit{or equivalently,}
\begin{align*}
\hspace{0,9cm}L(x_1^{n_1}\dots x_d^{n_d}(1- (x_1+\dots+x_d))^{n_{d+1}}) \geq 0 \quad \text{\textit{for}}~~~n_1,\dots,n_{d+1}\in \N_0.\hspace{0,9cm}\hfill \circ
\end{align*}
\end{example}

\begin{example}\label{Delta_dsimplex} (\textit{Standard simplex $\Delta_d$ in $\R^d$})\\
 Let \, $f_1=x_1,\dots,\, f_d=x_d,\, f_{d+1}=1- \sum_{i=1}^d x_i,\,  f_{d+2}=-f_{d+1}$.
 Then  $K(\mathsf{f})$ is 
\begin{align*}
\Delta_d=\{x\in \R^d: x_1\geq 0,\dots,x_d\geq 0, x_1+\dots+x_d=1\}.
\end{align*}
\textit{A linear functional $L$ is a $\Delta_d$-moment functional if and only if}
\begin{align*}
L(x_1^{n_1}\dots x_d^{n_d})\geq 0,~  L(x_1^{n_1}\dots x_d^{n_d}(1{-}(x_1{+}\dots{+}x_d))^r)=0, ~~ n_1,\dots,n_d\in \N_0, r\in \N.\circ
\end{align*}
\end{example}

\begin{example} (\textit{Multidimensional\, Hausdorff\, moment\, problem on $[0,1]^d$})\\
 Set  $f_1=x_1, f_2=1-x_1,\dots,f_{2d-1}=x_d,f_{2d}=1-x_d, k=2d$. Then $K(\,\mathsf{f}\,)= [0,1]^d$. Let $s=(s_\gn)_{\gn\in \N_0^d}$ be a 
 multisequence. We define  the shift $E_j$ of the $j$-th index by  \begin{align*}
 (E_js)_\gm=s_{(m_1,\dots,m_{j-1},m_j+1,m_{j+1},\dots,m_d)}, ~~~\gm \in \N_0^d.
 \end{align*}
  Then \textit{$L_s$ is a $[0,1]^d$-moment functional on $\R_d[\ux]$ if and only if }
  \begin{align*}
  \hspace{1cm} L_s( x_1^{m_1}(1-x_1)^{n_1}\cdots x_d^{m_d}(1-x_d)^{n_d})\geq 0~~~{for}~~ \gn,\gm\in \N_0^d.\hspace{0.8cm}\hfill \circ
  \end{align*}
  \end{example}

\section{Two technical results and the proof of Theorem \ref{mpschm}}\label{proof}

 In this section,  we reproduce from [MP] the proofs of the following results.
 Both results of interest in themselves.
 As noted above, they complete the proof of Theorem \ref{mpschm}.
 Suppose  $K(\mathsf{f})$ is \emph{compact}.
 We  write $a \preceq b$\, if \, $b-a \in T(\mathsf{f})$.

\begin{proposition}\label{prearchcom}
The 
preordering $T(\mathsf{f})$ is Archimedean.
\end{proposition}
\begin{proof}
Let\, $g\in \R_d[\ux]$ be a fixed polynomial  and  $\lambda >0$. Suppose that\, $\lambda^2>g(x)^2$\, for all $x\in  K(\mathsf{f})$. Our first aim to to show that there exists a $p \in T(\mathsf{f})$
such that 
\begin{align}\label{p2n2}
g^{2n}  \preceq \lambda^{2n+2}p ~~~\text{for}~~~n\in \N.
\end{align}
Indeed, by the Krivine--Stengle Positivstellensatz (Theorem \ref{krivinestengle}(i)) there exist polynomials
$p,q \in T(\mathsf{f})$ such that
\begin{align}\label{stengle}
p(\lambda^2-g^2) =1 +q.
\end{align}
Since $q \in T(\mathsf{f})$ and $T(\mathsf{f})$ is a quadratic module, $g^{2n}(1+q) \in T(\mathsf{f})$ for $n \in \N_0$. Therefore, using (\ref{stengle}) we conclude  that
\begin{align*}g^{2n+2}p = g^{2n} \lambda^2 p-g^{2n}(1+q) \preceq g^{2n} \lambda^2 p.\end{align*} By induction it follows that 
\begin{align}\label{p2n1}
g^{2n} p \preceq \lambda^{2n} p.
\end{align}
Since $g^{2n}(q+pg^2) \in T(\mathsf{f})$, using first (\ref{stengle}) and then  (\ref{p2n1}) we derive
\begin{align*}
g^{2n} \preceq g^{2n} + g^{2n}(q+pg^2) = g^{2n}(1+q+pg^2)=g^{2n}\lambda^2 p \preceq \lambda^{2n+2}p\,. 
\end{align*}
\end{proof}
This completes the proof of (\ref{p2n2}).

Now we put $g(x):=(1+x_1^2)\cdots(1+x_d^2)$.
Since $g$ is bounded on the compact set $K(\mathsf{f})$, we have $\lambda^2 >g(x)^2$ on  $K(\mathsf{f})$ for some $\lambda>0$. Therefore, by the preceding  there exists a $p \in T(\mathsf{f})$ such that (\ref{p2n2}) holds.  

 Further, for any multiindex $\alpha\in \N_0^d$, $|\alpha| \leq k$, $k \in \N$, we obtain
\begin{align}\label{pk}
\pm 2x^\alpha \preceq x^{2\alpha} + 1 \preceq \sum_{|\beta| \leq k} x^{2\beta} = g^k.
\end{align}
Hence there exist numbers $c >0$ and $k \in \N$ such that $p
\preceq 2c g^k$. Combining the latter with $g^{2n}  \preceq \lambda^{2n+2}p$ by~(\ref{p2n2}), we get
$g^{2k} \preceq \lambda^{2k+2} 2c g^k$ and so
\begin{align*}(g^k{-}\lambda^{2k+2}c)^2 \preceq (\lambda^{2k+2}c)^2{\cdot}1.\end{align*} Hence, by
Lemma \ref{boundedele1}, $g^k{-}\lambda^{2k+2}c \in \R_d[\ux]_\mathrm{b}$ and so  $g^k \in\R_d[\ux]_\mathrm{b}$ with respect to the preorder  $T(\mathsf{f})$. Since $\pm x_j \preceq g^k$ by
(\ref{pk}) and $g^k \in\R_d[\ux]_\mathrm{b}$, we obtain $x_j  \in\R_d[\ux]_\mathrm{b}$ for $j=1,{\cdots},d$. Now from Lemma \ref{boundedele2}(ii) it follows that $T(\mathsf{f})$ is Archimedean.
$\hfill\qed$ \end{proof}  

\begin{proposition}\label{continuityLprop}
Suppose  $L$ is a\, $T(\mathsf{f})$-positive linear functional on $\R_d[\ux]$. 
\begin{itemize}
\item[\em (i)]\,\,\,  If\, $g\in\R_d[\ux]$ and $\|g \|_\infty$ denotes the  supremum of  $g$ on $K(\mathsf{f}),$  then
\begin{align}\label{conitnuityL}
|L(g)|\leq L(1) ~\|g\|_\infty .
\end{align}
\item[\em (ii)]~ $L$ is a $K(\mathsf{f})$-moment functional.
\end{itemize}
\end{proposition}
\begin{proof}
(i):  Fix $\varepsilon >0$ and put $\lambda:= \parallel g \parallel_\infty +\varepsilon$. 
We define a real sequence $s=(s_n)_{n\in \N_0}$ by  $s_n:=L(g^n)$. Then $L_s(q(y))=L(q(g))$ for  $q\in \R[y]$. For any $p\in \R[y]$, we have $p(g)^2\in \sum \R_d[\ux]^2\subseteq T(\mathsf{f})$ and hence $L_s(p(y)^2)=L(p(g)^2)\geq 0$, since $L$ is $T(\mathsf{f})$-positive.
Thus, by  Hamburger's theorem%\ref{solutionhmp}
, there exists a Radon measure $\nu$ on $\R$ such that $s_n= \int_\R t^n d\nu(t)$, $n \in \N_0$. 

For $\gamma >\lambda$ let $\chi_\gamma$ denote the characteristic function of the set $(-\infty,-\gamma]\cup[\gamma,+\infty)$. Since $\lambda^2-g(x)^2>0$ on $K(\mathsf{f})$,  we have $g^{2n} \preceq \lambda^{2n+2}p$\, by equation (\ref{p2n2}) in the proof of Proposition \ref{prearchcom}. Using the $T(\mathsf{f})$-positivity of $L$  we derive
\begin{align}\label{gamma2nlp}
\gamma^{2n} \int_\R \chi_\gamma(t) ~d\nu(t) \leq \int_\R t^{2n} d\nu(t) =s_{2n}= L(g^{2n}) \leq \lambda^{2n+2} L(p)
\end{align}
for all $n\in \N$.
 Since $\gamma >\lambda$, (\ref{gamma2nlp}) implies that $\int_\R \chi_\gamma(t)~d\nu(t) =0$.
 Hence $\supp\nu \subseteq [-\lambda,\lambda]$.
 Therefore, applying the Cauchy--Schwarz inequality for $L$, we derive
\begin{align*}
|L(g)|^2 &\leq L(1)L(g^2) 
= L(1) s_2=
L(1)\int_{-\lambda}^\lambda~ t^2 ~d\nu(t)\\& \leq L(1)\nu(\R)\lambda^2 = L(1)^2\lambda^2=L(1)^2(\parallel g \parallel_\infty + \varepsilon)^2.
\end{align*}
Letting $\varepsilon \to +0$, we get\, $|L(g)| \leq L(1)\parallel g \parallel_\infty$. 
\smallskip

(ii): Suppose that $g\in \R_d[\ux]$ and  $g\geq 0$ on $K(\mathsf{f})$. Then, we conclude easily that\, $\|\,1\cdot \|g\|_\infty-2\,g\|_\infty = \|g\|_\infty.$ Using this equality and (\ref{conitnuityL}) we obtain
\begin{align*}
L(1) \|g\|_\infty - 2\,L(g)=L( 1 \cdot\|g\|_\infty - 2\,g) \leq L(1)\| 1\cdot \|g\|_\infty -2\,g\|_\infty =L(1)\|g\|_\infty ,
\end{align*}
which in turn implies that\, $L(g)\geq 0$. Therefore,  by Haviland's theorem, $L$ is a $K(\mathsf{f})$-moment functional.
$\hfill\qed$ \end{proof}

%\setcounter{tocdepth}{3}

%\end{document}

%%%%%%%%%%%%%%%%%%%%%%%%%%%%%%%%%%%%%%%%%%%%%%%%%%%%%%%%%%%%%%%%%

%\chapter{Polynomial optimization and semidefinite programming}

\setcounter{tocdepth}{3}

\makeatletter
\renewcommand{\@chapapp}{Lecture}
\makeatother
%\chapter{ }
%\chapter{ }
%\newpage
%\setcounter{page}{1}

\chapter{The moment problem on closed  semi-algebraic sets: the fibre theorem}

Abstract:\\
\textit{This main result of this lecture is the fibre theorem about the existence for the moment problem on (unbounded) semi-algebraic sets.
This theorem reduces  moment properties of a preordering $T(\mathsf{f})$ to those for fibre preorderings built by fibres of  bounded polynomials on the set $\cK(\mathsf{f})$. }
%Most of the known general affirmative results on the moment problem for closed semi-algebraic sets can be derived from this theorem.}
\bigskip

%The moment problem on unbounded sets leads to  new and principal difficulties, compared to the case of bounded semi-algebraic sets. 

The moment problem  on compact semi-algebraic sets has  rather  satisfactory results. There  are useful existence theorems  (Theorems \ref{mpschm}, \ref{archmedps} and \ref{prestel})
%(Theorem  \ref{mpschm})  
and the solution is always unique (by the Weierstrass theorem on uniform approximation by polynomials).
Both facts are  longer true in the noncompact case. 
The moment problem on unbounded sets leads to  new principal difficulties concerning the existence of a solution and also concerning the uniqueness of the representing measure. 

 For noncompact semi-algebraic sets, only in rare cases the positivity of a linear functional on the preordering is sufficient for being a moment functional. 
The fibre theorem is a powerful and deep general result from which almost all known affirmative results can be derived.
%

%\chapter{Truncated multidimensional moment problem: existence via positivity}
 
\section{Positive functionals which are not moment functionals}\label{possums}

Each real polynomial in one variable which is nonnegative on $\R$ is a sum of squares. This result was the crucial step main to the simple criteria for the one-dimensional Hamburger moment problem.
As already noted  by Hilbert (1888), there are nonnegative polynomials in two variables which cannot be represented as finite sums of squares of polynomials. A prominent example is Motzkin's polynomial 
\begin{align}\label{motzkinformula}
p_c(x_1,x_2):= x_1^2x_2^2(x_1^2+x_2^2- c) +1,~~0<c\leq 3.
\end{align}
It can be shown that the cone $\sum \R_d[\ux]^2$ of sums of squares of polynomials is closed in the finest locally convex topology on the  vector space $\R_d[\ux]$. Hence, since $p\notin \sum \R_2[\ux]^2$,  the separation theorem of convex sets implies that there is a linear functional $L$ on $\R_2[\ux]$ such that $L$ is nonnegative on $\sum \R_2[\ux]^2$ and $L(p_c)<0$. Thus $L$ is a positive linear functional on the algebra $\R_2[\ux]$. But, since $p_c\geq 0$ on $\R^2$, $L$ cannot be represented as an integral by a positive measure on $\R^2$.
This indicates that the existence problem for the moment problem on noncompact sets is much more difficult than in the compact case.

A similar example for  the Stieltjes moment problem in $\R^2$ is
\begin{align}\label{qcstieltjes}
q_c(x_1,x_2):=x_1x_2(x_1+x_2- c) +1, ~~c\in (0,3].
\end{align}
 Since $q_c(x_1^2,x_2^2)=p_c(x_1,x_2)$,  $q_c\geq 0 $ on the positive quarter plane $\R_+^2$ and $q_c$ does not belong to the preordering
 \begin{align*}
T=\sum \R[x_1,x_2]^2+x_1\sum \R[x_1,x_2]^2+x_2\sum \R[x_1,x_2]^2+x_1x_2\sum \R[x_1,x_2]^2.
\end{align*}
Hence the functional $L^\prime$ defined by $L^\prime(f)=L(f(x_1^2,x_2^2))$,  $f\in\R[x_1,x_2]$, is  a $T$-positive linear functional, which cannot be given by a  Radon measure on\,  $\R_+^2$.

%For noncompact semi-algebraic sets, only in rare cases the positivity of a linear functional on the preordering is sufficient for being a %moment functional. 

\section{Properties (MP) and (SMP) and the fibre theorem}\label{propertiessmpandmpandfibretheorem}

In this Lecture,  ${\sA}$ is a \textbf{finitely generated commutative real unital algebra}. 

Then ${\sA}$ is (isomorphic to) a quotient algebra $\R_d[\ux]/\cJ$ for some ideal $\cJ$ of $\R_d[\ux]$ and the  character set  $\hat{\sA}$ of ${\sA}$ is the zero set $\cZ(\cJ)$ of $\cJ$. Hence
 $\hat{\sA}$, equipped with the weak topology,   is a locally compact Hausdorff space and the elements of $\sA$ become continuous functions on this space. 
 
 For a quadratic module $Q$ of ${\sA}$ we define
\[\cK(Q):=\{x\in\hat{{\sA}}: f(x)\geq 0 ~, f\in Q \}\] 
Let $\cM_+(\hat{\sA})$ denote the  Radon measures $\mu$ on $\hat{\sA}$ such that each $f\in {\sA}$ is $\mu$-integrable.

Our main concepts are introduced in the following definition. 
\begin{definition}\label{definionsmpmp}
 A quadratic module $Q$ of\, ${\sA}$ has the\smallskip

 $\bullet$~  \emph{moment property (MP)}\index{Moment property (MP)}\, if   each $Q$-positive linear functional $L$ on ${\sA}$ is a moment functional, that is, there exists a  Radon measure $\mu\in \cM_+(\hat{{\sA}})$ such that
\begin{align}\label{Lrepreborelmu}
L(f) =\int_{\hat{{\sA}}}\, f(x) \, d\mu(x) \quad \text{ for~all}~~f\in {\sA},
\end{align} 
 $\bullet$ ~ \emph{strong moment property (SMP)}\index{Strong moment property (SMP)} if  each $Q$-positive linear functional $L$ on ${\sA}$ is a $\cK(Q)$--moment functional, that is, there is a Radon measure $\mu\in \cM_+(\hat{{\sA}})$ such that\, 
$\supp\mu\subseteq \cK(Q)$\, and
(\ref{Lrepreborelmu}) holds.
\end{definition}

As noted in Section \ref{possums}, (MP) fails for the preordering $\sum \R_d[\ux]^2$, $ d\geq 2$. Obviously, (SMP) implies (MP). The converse is not true, as shown by the next example.
\begin{example}\label{stelnosmp} 
Let $\sA=\R[x]$ and consider the preordering $T_3=\sum \R[x]^2+ x^3\sum \R[x]^2$. Then $\cK(T_3)=[0,+\infty)$. By Hamburger's theorem each $T_3$-positive linear functional can be given  by a Radon measure on $\R$, so $T_3$   satisfies  (MP).
 It can be shown (see [MP, Example~ 13.7]) 
%\cite[Example 13.7]{mp})
 that there exists a $T_3$-positive functional on $\sA$ which has no representing measure supported on $[0,+\infty)$. This means that (SMP) fails for the preordering $T_3$. 

Set $T_1=\sum \R[x]^2+ x\sum \R[x]^2$. Then, by Stieltjes theorem, each $T_1$-positive linear functional has a representing measure supported on $[0,+\infty)$, so $T_1$ has (SMP).

Note that $\cK(T_3)=\cK(T_1)=[0,+\infty)$. Thus, whether a preordering $T$ has (SMP) is not determined by the semi-algebraic set $\cK(T)$. It  depends essentially on the ``right" generators of $T$.
$\hfill\circ$
\end{example}

Now we begin with the preparations for the fibre theorem.

 Suppose $T$ is a finitely generated preordering  of ${\sA}$ and  $\mathsf{f}=\{f_1,\dots,f_k\}$ is a set of generators of  $T$.
 We consider an $m$-tuple\, $\mathsf{h}=(h_1,\dots,h_m)$\,  of  elements $h_k\in {\sA}$. Let
\begin{align}\label{defhKT}
\lambda=(\lambda_1,\dots,\lambda_r)\in  \mathsf{h}(\cK(T)):=\{(h_1(x),\dots,h_m(x)): x \in \cK(T)\}.
 \end{align}
 We denote by 
$\cK(T)_\lambda$ the subset of $\hat{{\sA}}$ given by
\begin{align*}
\cK(T)_\lambda=\{ x\in \cK(T): h_1(x)=\lambda_1,\dots,h_m(x)=\lambda_m \},
\end{align*} 
and by\, $T_\lambda$ the  preordering of ${\sA}$ generated by the sequence
\begin{align*}\mathsf{f}(\lambda):= \{f_1,\dots,f_k,h_1-\lambda_1,\lambda_1-h_1,\dots,h_m-\lambda_m,\lambda_m
-h_m\}.
\end{align*}
Clearly,   $\cK(T)$ is  the disjoint union of fibre set $\cK(T)_\lambda$, where $\lambda \in\mathsf{h}(\cK(T))$.

Let $\cI_\lambda$ denote  the ideal of ${\sA}$ generated by $h_1-\lambda_1,\dots,h_m-\lambda_m$. Then the preordering $(T +\cI_\lambda)/\cI_\lambda$ of the quotient algebra ${\sA}/\cI_\lambda$  is generated by 
\begin{align*}
\pi_\lambda(\mathsf{f}):=\{\pi_\lambda(f_1),\dots,\pi_\lambda(f_k)\},
\end{align*}
where $\pi_\lambda:{\sA}\to {\sA}/\cI_\lambda$ denotes the canonical map.
\smallskip

Let us illustrate this in the simple  case of a strip $[a,b]\times\R$ in $\R^2$.
\begin{example}\label{abstrip}
 Let $a,b\in \R$, $a<b$, $d=2$, and $\mathsf{f}_1=\{f_1:=(x_1-a)(b-x_1)\}$.
 Then 
\begin{align*}
%\cK(\mathsf{f}_1)=[a,b]\times\R ~,~~
T:= T(\mathsf{f}_1)=\sum \R[x_1,x_2]^2+ (x_1-a)(b-x_1)\sum \R[x_1,x_2]^2,~~\cK(T)=[a,b]\times\R.
 ~\end{align*} 
 Set $\mathsf{h}=(h_1)$, where $h_1(x_1,x_2):=x_1$.
 Clearly,  $\mathsf{h}(\cK(T))=[a,b]$.
 Fix $\lambda \in [a,b]$. 

 Then the fibre $\cK(T)_\lambda $ is the line $\{\lambda\}\times \R$, the ideal $\cI_\lambda$ is generated by $x_1-\lambda$ and $\pi_\lambda (f_1)=(\lambda-a)(b-\lambda)$ is a constant.
 Since   $x_1=\lambda$ in the quotient algebra by $\cI_\lambda$, we obtain ${\sA}/\cI_\lambda=\R[x_2]$ and  $(T+\cI_\lambda)/\cI_\lambda =\sum \R[x_2]^2$.

 Now we add the polynomial $f_2(x_1,x_2)=x_2$ to  $\mathsf{f}_1$, that is, we replace $\mathsf{f}_1$ by\,  $\mathsf{f}_2=\{f_1=(x_1-a)(b-x_1),f_2:=x_2\}$.
 Then we have $\cK(T)=[a,b]\times [0,+\infty) $.
 As above, let $\mathsf{h}=(h_1)$.
 Let $\lambda \in [a,b]$.
 As in the case of $\mathsf{f}_1$, we get ${\sA}/\cI_\lambda=\R[x_2]$.
 But now  $\cK_\lambda = \{\lambda\}\times [0,+\infty)$ and the preordering is $(T+\cI_\lambda)/\cI_\lambda =\sum \R[x_2]^2+ x_2 \sum \R[x_2]^2$.

 In both cases, the fibre preorderings have (SMP) and (MP) by the solutions of the Hamburger and Stieltjes moment problems, respectively.

 Finally, let us consider  the case  $\mathsf{f}_3=\{f_1=(x_1-a)(b-x_1),f_2:=x_2^3\}$.
 Then we obtain $(T+\cI_\lambda)/\cI_\lambda =\sum \R[x_2]^2+ x_2^3 \sum \R[x_2]^2$.
 This fibre preordering has (MP), but (SMP) fails, as noted in Example \ref{stelnosmp}. $\hfill\circ$
\end{example}

 The  following \emph{fibre theorem} is the main result of this Lecture. It was proved by the author first (2003) for $\sA=\R_d[\ux]$  and later (2015) in the general case.%It  allows us to derive  (SMP) or (MP) for $T$ from the corresponding properties of fibre preorderings.
\begin{theorem}\label{fibreth1} 
 Let ${\sA}$ be a finitely generated commutative real unital algebra and let $T$ be a finitely generated preordering of  ${\sA}$.
 Suppose  $h_1,\dots,h_m$ are elements of ${\sA}$ that are \textbf{bounded} on the  set  $\cK(T)$.
 Then the following are equivalent:
\begin{itemize}
\item[\em (i)]~  $T$ satisfies property (SMP) (resp. (MP)) in ${\sA}$.
%\item[\em (ii)]~   $T_\lambda$ satisfies (SMP) (resp. (MP)) in\, ${\sA}$\, for all\, $\lambda \in \mathsf{h}(\cK(T)).$
%\item[\em (ii)$^\prime$]~\,  $\hat{T}_\lambda$ satisfies (SMP) (resp. (MP)) in\, ${\sA}$\, for all\, $\lambda \in \mathsf{h}(\cK(T)).$
\item[\em (iii)]~\, $(T+\cI_\lambda)/\cI_\lambda$ satisfies (SMP) (resp. (MP)) in\, ${\sA}/\cI_\lambda$\, for all\, $\lambda \in \mathsf{h}(\cK(T)).$
%\item[\em (iii)$^\prime$]~\,   $\hat{T}_\lambda/\hat{\cI}_\lambda$ satisfies (SMP) (resp. (MP)) in\, ${\sA}/\hat{\cI}_\lambda$\, for all\, $\lambda \in \mathsf{h}(\cK(T)).$ 
\end{itemize}
\end{theorem}
\begin{proof}
 {[MP, Theorem~13.10]}. 
%\cite[Theorem 13.10]{mp}.
$\hfill \qed$
\end{proof}

The main implication (ii)$\to$(i) of this theorem  allows us to derive the properties (SMP) or (MP) for $T$ from the corresponding properties of fibre preorderings $T_\lambda$. Thus, if the algebra $\sA$ contains elements $h_j$ that are bounded on the set $\cK(T)$, then the moment properties of $T$ can be reduced to the preorderings $T_\lambda$ which have (in general) ``lower dimensional" fibre sets $\cK(T_\lambda)$.

 The proof of (i)$\to$(ii) is not difficult, while the proof of   (ii)$\to$(i) is very long  and technically involved (see [MP, Section~13.10]).
 It uses  the Krivine--Stengle Positivstellensatz.
 There is also a version of the fibre theorem for quadratic modules, but this is much weaker and it is stated in [MP, Theorem~13.12].
%\cite[Theorem 13.12]{mp}.

\begin{remark} 
 Let us consider  the special case ${\sA}=\R_d[\ux]$, $T=T( \mathsf{f}),$ and assume that the semi-algebraic sets $\cK(T( \mathsf{f}))$ is \emph{compact}.
 Since the coordinate functions $x_j$ are bounded on  $\cK(T( \mathsf{f}))$, they can be taken as\, $h_j$.
 Then, all  fibre algebras ${\sA}/\cI_\lambda$ are equal to $\R$ and $(T+\cI_\lambda)/\cI_\lambda$ is $[0,+\infty)$, so it   has obviously (SMP). Therefore, by Theorem \ref{fibreth1}, $T=T(\mathsf{f})$ obeys (SMP) in $\R_d[\ux]$. This is the assertion of Theorem \ref{mpschm}  of Lecture 3.
% Theorem \ref{mpschm}!
$\hfill \circ$
\end{remark}

Theorem  \ref{fibreth1} is  useful if the algebra contains many bounded elements. This is often the case for algebras of rational functions as the following example shows.
\begin{example}
Let $\sA$  be the real algebra of rational functions on $\R^d$ generated by the polynomial algebra $\R_d[\ux]$ and $q_j(x)=(1+x_j^2)^{-1},$  $j=1,\dots,d$. Let $T=\sum \sA^2$. Then $\hat{\sA}$ is given the evaluations at points of $\R^d$. Hence $\cK(T)=\hat{\sA}=\R^d$.  Clearly, the elements $h_j:=q_j$ and   $h_{d+j}=x_jq_j$, $j=1,\dots,d$, of $\sA$   are bounded on $\cK(T)$. 

Let $\lambda \in \mathsf{h}(\cK(T)).$ Since $\lambda_j=q_j(\lambda)\neq 0$,  $ \lambda_{d+j}=\lambda_jq_j(\lambda)$, we have  $x_j=\lambda_{d+j}\lambda_j^{-1}$,  $j=1,\dots,d$, in the fibre algebra $\sA_\lambda.$ Thus,  $\sA_\lambda=\R$, so it is trivial that $\sum (\sA_\lambda)^2$  has  (MP). 
Therefore, $\sum \sA^2$   obeys (MP) by Theorem \ref{fibreth1}.
 This means that \emph{each positive linear functional on the algebra $\sA$ is given by some Radon measure on\, $\hat{\sA}=\R^d$}.

This conclusion and the preceding reasoning remain valid if $\sA$ is replaced by the algebra generated by $\R_d[\ux]$ and   the single function $q(x)=(1+x_1^2+\dots+x_d^2)^{-1}.$  $\hfill \circ$
\end{example}

The following simple lemmas are needed for the applications in the next sections.
\begin{lemma}\label{quotientsuma2}
If\, $\sum{\sA}^2$ has  property (MP) in $\sA$ and $\cI$ is an ideal of  ${\sA}$, then  $\sum ({\sA}/\cI)^2$ has (MP) in\, ${\sA}/\cI$.
\end{lemma}
\begin{proof}
Let $\rho$ denote  the canonical map of $\sA$ into $\sA/\cI$. Suppose  $L$ is a positive linear functional on $\sA/\cI$. Then $\tilde{L}:=L\circ \rho$  is a positive functional on $\sA$. Since $\sum \sA^2$ has (MP) by assumption,  $\tilde{L}$, hence $L$, is given a Radon measure on $\hat{\sA}$. Because $\tilde{L}$ annihilates $\cI$, the measure is supported on the zero set of $\cI$ and so on $\widehat{\sA/\cI}$. This means that $\sum ({\sA}/\cI)^2$ has (MP).
$\hfill \qed$\end{proof}
\begin{lemma}\label{singeHamburger}
If a real  algebra ${\sA}$ has a single generator, then\,
$\sum {\sA}^2$\,   obeys (MP).
\end{lemma}
\begin{proof}
Because $\sA$ is single generated,  it is a quotient algebra $\R[y]/\cI$ for some ideal $\cI$ of $\R[y]$. 
Since $\sum \R[y]^2$ has (MP) by  Hamburger's theorem, it follows from  Lemma \ref{quotientsuma2}  that\,  $\sum \sA^2=\sum (\R[y]/\cI)^2$\, has (MP).
$\hfill\qed $\end{proof} 
\section{First application:  cylinder sets with compact base}\label{cylindersets}
 
 The  moment problem for cylinders with compact base fits nicely to  the fibre theorem, as shown by the following Theorem.
\begin{theorem}\label{cylindercompact}
Let $\cC$ be a compact set in $\R^{d-1}$, $d\geq 2,$ and let $\mathsf{f}$ be a finite subset of $\R_d[\ux]$. Suppose the semi-algebraic subset $\cK(\mathsf{f})$  of $\R^{d}$  is contained in the cylinder\, $\cC\times \R$. Then the preordering  $T(\mathsf{f})$ has (MP). If $\cC$ is a semi-algebraic set in  $\R^{d-1}$ and $\cK(\mathsf{f})=\cC\times \R$, then $T(\mathsf{f})$ satisfies (SMP).
\end{theorem} 
\begin{proof}
Define $h_j(x)=x_j$ for $j=1,\dots,d-1$. Since $\cK(\mathsf{f})\subseteq\cC\times \R$ and $\cC$ is compact, the polynomials $h_j$ are bounded on $\cK$, so the assumptions of Theorem \ref{fibreth1} are fulfilled. 
Then all fibres $\cK(\mathsf{f})_\lambda$  are subsets of $ (\lambda_1,\dots,\lambda_{d-1})\times \R$,  the  preordering 
$T(\mathsf{f})_\lambda$ contains $\sum \R[x_d]^2$, and the quotient algebra $\R_d[\ux]/ \cI_\lambda$ is an algebra of polynomials in the single variable $x_d$. 
By Lemma \ref{singeHamburger}, $\sum (\R[x_d]/\cI_\lambda)^2$ obeys (MP)  and so obviously does the larger preordering $T(\mathsf{f})_\lambda/\cI_\lambda$. Therefore,    $T(\mathsf{f})$ has (MP) by  Theorem \ref{fibreth1},(ii)$\to$(i). 

If $\cK(\mathsf{f})=\cC\times \R$, the fibres for $\lambda \in \mathsf{h}(\cK(\mathsf{f}))$ are equal to  $ (\lambda_1,\dots,\lambda_{d-1})\times \R$ and $T(\mathsf{f})_\lambda=\sum \R[x_d]^2$. Hence the $T(\mathsf{f})_\lambda$ satisfy (SMP) and so does   $T(\mathsf{f})$.
\qed \end{proof}

We return to the  case of a strip in $\R^2$ and derive a classical result due to Devinatz (1955).

\begin{example} (\textit{Example \ref{abstrip} continued})\\
Let us retain the notation of Example \ref{abstrip}.
% con$a,b\in \R$, $a<b$, $d=2$, and $\mathsf{f}=\{(x_1-a)(b-x_1)\}$. Then 
%\begin{align*}\cK(\mathsf{f})=[a,b]\times\R ~,~~ T(\mathsf{f})=\sum \R[x_1,x_2]^2+ (x_1-a)(b-x_1)\sum \R[x_1,x_2]^2.\end{align*} 
As noted therein, all fibre preorderings for $\mathsf{f}_1$ and $\mathsf{f}_2$  obey (SMP), so do $T(\mathsf{f}_1)$ and $T(\mathsf{f}_2)$ by Theorem \ref{fibreth1}. 

We state this result for $T(\mathsf{f}_1)$: Given a linear functional $L$ on $\R[x_1,x_2]$, there exists a  Radon measure $\mu$  on $\R^2$ supported on\, $[a,b]\times\R$ such that $p$ is $\mu$-integrable  and
\begin{align*}
 L(p)
 =\int_a^b \int_\R  p(x_1,x_2)\, d\mu(x_1,x_2)~~~\text{ for~ all}~~p\in \R[x_1,x_2]
\end{align*}
 if and only if 
\begin{align*}
\hspace{1,6cm}L(q_1^2+(x_1-a)(b-x_1)q_2^2)
 \geq 0 ~~~\text{ for~ all}~~~q_1,q_2\in\R[x_1,x_2].\hspace{1,5cm}\hfill \circ
\end{align*}
\end{example}

\section{(SMP) for subsets of the real line}

From Example \ref{stelnosmp} we know that (MP) does not imply (SMP) even for the polynomial algebra $\R[x]$ in one variable. The  reason is that (SMP) requires the "right set of generators" for a non-compact semi-algebraic set. 

Suppose that $\cK$ is a (nonempty)  closed semi-algebraic proper subset of $\R$, that is, $\cK$ is the union of finitely many closed  intervals. (These intervals can be  points.) In order to state the result it is convenient to introduce the following notion.

\begin{definition}\label{naturalchoicegen}
 A finite subset\, $\mathsf{g}$ of\, $\R[x]$ is called a \emph{natural choice of generators} for $\cK$ if  $\mathsf{g}$  is the smallest set  satisfying the following conditions:\\ 
$\bullet$~ If\, $\cK$ contains a least element $a$ (that is, if $(-\infty,a)\cap \cK=\emptyset$), then $(x-a)\in  \mathsf{g}$.  \\
$\bullet$~ If\, $\cK$ contains a greatest element $a$ (that is, if $(a,\infty)\cap \cK=\emptyset$), then $(a-x)\in  \mathsf{g}$.\\
$\bullet$~ If\, $a,b\in \cK$, $a<b$, and $(a,b)\cap\cK=\emptyset$, then $(x-a)(x-b)\in \mathsf{g}$. 
\end{definition}   
It is not difficult to verify  that a choice of natural generators  always exists and that  it is uniquely determined. Moreover,   w  $\cK=\cK(\mathsf{g})$.
For $\cK=\R$ we set $\mathsf{g}=\{1\}$. 
\begin{example}
%Let us give some examples for the natural choice of generators:\\
$\cK=[b,+\infty)$: $\mathsf{g}=\{(x-b)\},$\\
$\cK=\{a\}\cup [b,+\infty)$, where $a<b$:~ $\mathsf{g}=\{x-a, (x-a)(x-b)\},$\\
%$\cK=[a,b]\cup\{c\} $, where $a<b<c$:~$\mathsf{g}=\{x-a, (x-b)(x-c), c-x\},$\\
$\cK=\{a\}\cup\{b\}$, where $a<b$:~ $\mathsf{g}=\{x-a,(x-a)(x-b), b-x\},$\\
$\cK=\{a\}$:~ $\mathsf{g}=\{x-a,a-x\}.$
\end{example}
It can be shown that ${\Pos}(\cK)=T(\mathsf{g})$ if  $\mathsf{g}$ is the natural choice of generators. If the semi-algebraic set if compact, the preordering has always (SMP) by  Theorem \ref{mpschm}. The difficulty of getting the "right support" of the representing measure appears only in the non-compact case and the answer is given by the following result of Kuhlmann and Marshall (2002).
%The next proposition shows that for the natural choice of generators the preordering contains all nonnegative polynomials on $\cK$.
%\begin{proposition}\label{kgposgnatural}
%Suppose that $\cK$ is a nonempty basic closed semi-algebraic subset of $\R$. If $\mathsf{g}$ is the natural choice of generators for $\cK$, then
%${\Pos}(\cK)=T(\mathsf{g}).$
%\end{proposition} 
\begin{theorem}\label{charonedimsmpnc}
Suppose  $\mathsf{f}$  is a finite subset of $\R[x]$ such that the semi-algebraic subset\,  $\cK(\mathsf{f})$ of\, $\R$ is  not compact. Then $T(\mathsf{f})$ obeys  (SMP) if and only if $\mathsf{f}$ contains  positive multiples of all polynomials of the natural choice of generators $\mathsf{g}$ of $\cK(\mathsf{f})$. 
%The following are equivalent:
%\begin{itemize}
%\item[\em (i)]~ $T(\mathsf{f})$ obeys  (SMP).
%\item[\em (ii)]~\, $T(\mathsf{f})={\Pos}(\cK(\mathsf{f}))$.
%\item[\em (iii)]~~ $\mathsf{f}$ contains  positive multiples of all polynomials of the natural choice of generators $\mathsf{g}$ of $\cK(\mathsf{f})$.
%\end{itemize}
\end{theorem}
\begin{proof}
%\cite[Theorem 13.24]{mp}. 
 {[MP, Theorem~13.24]}.
$\hfill\qed$
\end{proof}

\begin{example}\label{examplepoint} $f_1(x)=x_1, f_2(x)=1-x_1, f_3(x)=x_2^3-x_2^2-x_1,f_4(x)= 4-x_1x_2$.\\ Then $h_1(x)=x_1$ is bounded and $h_1(\cK(\mathsf{f}))=[0,1]$. The fibres for $\lambda\in (0,1]$ are compact, so the  preordering $T_\lambda$ has (SMP). 
The fibre set at $\lambda=0$ is\,  $\{0\}\cup [1,+\infty)$ and the  sequence $\pi_0(\mathsf{f})$ is $\{0,1,x_2^3-x_2^2,4\}$. Since $\pi_0(\mathsf{f})$ does not contain  multiples of all natural choice generators for\,  $\{0\}\cup [1,+\infty)$,  $T(\mathsf{f})$ does not have (SMP). $\hfill\circ$
\end{example}
\begin{example}\label{examplesx2^3}
$f_1(x)=x_1,f_2(x)=1-x_1, f_3(x)=1-x_1x_2, f_4(x)=x_2^3$.\\
Again let $h_1(x)=x_1$, so that\,  $h_1(\cK(\mathsf{f}))=[0,1]$. All fibres at $\lambda\in (0,1]$ are compact, so they obey (SMP). The fibre set at $\lambda=0$ is $[0,+\infty)$ and the  sequence\, $\pi_0(\mathsf{f})=\{ 0,1,0,x_2^3\}$\, does not contain a multiple of the natural choice generator $x_2$ for $[0,+\infty)$. Hence $T(\mathsf{f})$ does not satisfy (SMP). However, if we  replace $f_4$ by $\tilde{f}_4(x)=x_2$, then $T(\mathsf{f})$ has (SMP). $\hfill \circ$
\end{example}
\section{Second application: the two-sided complex moment problem}\label{twosidedcomplexmp}

 Let $\sB=\C[z,\ov{z},z^{-1},\ov{z}^{-1}]$\, be the $*$-algebra of  complex Laurent polynomials in $z$ and $\ov{z}$. The involution $f\mapsto f^+$ of $\sB$ is the complex conjugation. Note that $\sB$ is ($*$-isomorphic to) the semigroup $*$-algebra for the $*$-semigroup $\Z^2$ with involution $(m,n)^*=(n,m)$. 
%It is easily seen that the Hermitean  characters of $\sB$ are given by the    evaluation functionals at points of $\C^\times:=\C\backslash \{0\}.$

Suppose $s=(s_{m,n})_{(m,n)\in \Z^2}$ is a complex sequence and  $L_s$ is the $\C$-linear  functional  on  $\sB$   defined by 
$L_s(z^m \ov{z}^n)=s_{m,n}$, $(m,n)\in \Z^2$.

 The \textbf{two-sided complex moment problem} is  the following question:
\smallskip

 \textit{When does there exist  a Radon measure $\mu$ on $\C^\times:=\C\backslash \{0\}$  such that 
the function $z^m\ov{z}^n$ on $\C^\times$ is $\mu$-integrable and
\begin{align*}
s_{mn} =\int_{\C^\times}~ z^m\ov{z}^n\, d\mu(z)\quad \text{\textit{for}}\quad (m,n)\in \Z^2 ,
\end{align*}
or equivalently,}
\begin{align*}
L_s(p)=\int_{\C^\times} p(z,\ov{z})\, d\mu(z)\quad \text{\textit{for}} \quad p\in \sB\, ?
\end{align*}
%Note that this requires conditions for the measure $\mu$ at infinity {\it and} at zero. 
In this case we call $s$  a moment sequence  and $L_s$  a moment functional.
\smallskip

 The next  result is\, \textbf{Bisgaard's theorem} (1989).
 In terms of $*$-semigroups it means  that \emph{each positive semidefinite sequence  on $\Z^2$ is a moment sequence on $\Z^2$}.
 The main assertion of Theorem \ref{bisgaardthm} says  that each positive  functional on $\sB$ is a moment functional. Since $\sB$ has the character set  $\R^2\backslash \{0\}$,  this result is really surprising.
\begin{theorem}\label{bisgaardthm}
A linear functional $L$ on $\sB$ is a moment functional if and only if $L$ is a positive functional, that is,\, $L(f^+f)\geq 0$\, for all $f\in \sB$. 
\end{theorem}
\begin{proof}
It clearly suffices to prove the if part.

First we describe the  $*$-algebra $\sB$ in terms of generators.
 A vector space basis of $\C[z,\ov{z},z^{-1},\ov{z}^{-1}]$  is $\{z^k\ov{z}^l: k,l\in \Z\}$.
 Writing $z=x_1+{\ii} x_2$ with $x_1,x_2\in \R$ we  get
\begin{align*}
 z^{-1}=\frac{x_1-{\ii} x_2}{x_1^2+x_2^2}~~~\text{ and}~~~{\ov{z}}^{\,-1}=\frac{x_1+{\ii} x_2}{x_1^2+x_2^2}.
\end{align*}
Hence,\, as a complex  unital algebra, $\sB$ is generated by the four real functions 
\begin{align}\label{bisgaard}
x_1, ~~x_2,~~ y_1:= \frac{x_1}{x_1^2+x_2^2}\,,~~y_2:=\frac{x_2}{x_1^2+x_2^2}.
\end{align}
%on $\R^2\backslash \{0\}$. 
%Note that all four functions are unbounded on $\R^2\backslash \{0\}$.

The Hermitean part  ${\sA}:=\{f\in {\sB}: f^+=f \}$  of the complex $*$-algebra $\sB$  is a real algebra and  its character set $\hat{{\sA}}$ is given by  the point evaluations $\chi_x$  at  $x\in  \R^2\backslash \{0\}.$ (Obviously,  $\chi_x$ is  a character  for $x\in  \R^2\backslash \{0\}.$  Since  $(y_1+{\ii}y_2)(x_1-{\ii }x_2)=1$, there is no character $\chi$ on ${\sA}$ for which $\chi(x_1)=0$ and $\chi(x_2)=0$.)

The three functions 
\begin{align*}h_1(x)=x_1y_1=\frac{x_1^2}{x_1^2+x_2^2}\,,~  h_2(x)=x_2y_2=\frac{x_2^2}{x_1^2+x_2^2}\,,~ h_3(x)=x_1y_2=x_2y_1=\frac{x_1x_2}{x_1^2+x_2^2}\end{align*} 
are elements of ${\sA}$ and they are bounded on $\hat{{\sA}}.$ 

Consider a nonempty fibre set given by $h_j(x)=\lambda_j$, where $\lambda_j\in \R$ for $j=1,2,3.$ Since $\lambda_1+\lambda_2=1$, we can assume without loss of generality that $\lambda_1\neq 0.$  In the quotient algebra,\, ${\sA}/\cI_\lambda$\, we  have $x_1y_1=\lambda_1\neq 0$, so $y_1=\lambda_1 x_1^{-1}$, and $x_2y_1=x_1y_2=\lambda_3$, so $y_2=\lambda_3x_1^{-1}$ and $x_2=\lambda_3\lambda_1^{-1}x_1$. Hence the algebra ${\sA}/\cI_\lambda$ is generated by $x_1$,  $x_1^{-1}$, so it is a quotient of the algebra $\R[x_1,x_1^{-1}]$ of Laurent polynomials.

 The character set of $\R[x,x^{-1}]$ is the set  of nonzero reals and the corresponding moment problem is the \emph{two-sided} Hamburger moment problem.
 Since, as  easily shown, each nonnegative $f\in \R[x,x^{-1}]$  is a sum of squares,    $\sum \R[x,x^{-1}]^2$ obeys (MP) by Haviland's Theorem and so does the preordering $\sum ({\sA}/\cI_\lambda)^2$ of its quotient algebra ${\sA}/\cI_\lambda$
by Lemma \ref{quotientsuma2}. Therefore, by  Theorem \ref{fibreth1},\, $\sum {\sA}^2$ satisfies (MP). This means that each positive functional on ${\sA}$, hence on  $\sB$, is a moment functional.
$\hfill\qed$ \end{proof}

\setcounter{tocdepth}{3}

\makeatletter
\renewcommand{\@chapapp}{Lecture}
\makeatother
%\chapter{}
%\newpage
%\chapter{}

\chapter{The moment problem on closed sets: determinacy and Carleman condition}

Abstract:\\
\textit{In this lecture we discuss various form of determinacy in the multidimensional case. The multivariate Carleman condition leads to a far reaching existence and uniqueness theorem for the moment problem.}
\bigskip

%The moment problem on unbounded sets leads to  new and principal difficulties, compared to the case of bounded semi-algebraic sets. 

In  Lecture 3 we have noted that the moment problem  for noncompact semi-algebraic sets  leads to  new  difficulties concerning the existence of a solution. 
In this Lecture we will see that the same is true for the problem of determinacy.

\section{Various notions of determinacy}
Let us  begin with  the following  classical result due to M. Riesz (1923) for the determinacy of the one-dimensional Hamburger moment problem. 

\begin{theorem}\label{mriesz} Suppose  $\mu$ is a Radon measure on $\R$ such that all moments are finite and let $s$ be its moment sequence. Then the measure $\mu$, or equivalently, its moment sequence  $s$,  is determinate if and only if  the set of polynomials
 $\C[x]$ is dense in the Hilbert space $L^2(\R,(1+x^2)d\mu)$.

In particular,  if $\mu$ is determinate, then  $\C[x]$ is dense in $L^2(\R,\mu)$.
\end{theorem}
\begin{proof}
[MP, Theorem 6.10 and Corollary 6.11] $\hfill \qed$
\end{proof}

That $\C[x]$ is dense in the Hilbert space $L^2(\R,(1+x^2)d\mu)$ means that, given a  $f\in L^2(\R,(1+x^2)d\mu)$ and $\varepsilon >0$, there exists a polynomial $p\in \C[x]$ such that
\[
\int_\R |f(x)-p(x)|^2 (1+x^2)d\mu(x)< \varepsilon.
\]
 The use of \emph{complex} polynomials here is only of technical nature, because Hilbert spaces such as $L^2(\R,(1+x^2)d\mu)$ are usually over the complex field.

In higher dimensions $d\geq 2$ the determinacy problem is much more subtle and the equivalence stated in Theorem \ref{mriesz} is no longer valid.

Let $\cM_+(\R^d)$ denote the set of  Radon measures on $\R^d$ for which all moments are finite.
For $\mu \in \cM_+(\R^d)$ let $\cM_\mu$ be  the set of measures $\nu \in \cM_+(\R^d)$ which have the same moments as $\mu$, or equivalently, which satisfy 
\begin{align*}
 \int_{\R^d} p(x) \, d\nu(x)
 =\int_{\R^d} p(x)\, d\mu(x) \quad \text{ for all }\quad p\in \C_d[\ux].
\end{align*}
 We write $\nu\cong \mu$ if $\nu\in \cM_\mu$. Obviously, $``\cong"$ is an equivalence relation in $\cM_+(\R^d)$.
\begin{definition}\label{defvardetnotions}
For a measure $\mu \in \cM_+(\R^d)$ we shall say that \smallskip

$\bullet$~ $\mu$ is 
\emph{determinate}\index{Multidimensional determinacy notions! determinate} if\, $\cM_\mu$ is a singleton, that is, if\,   $\nu \in \cM_\mu$ implies $\mu=\nu$,\smallskip

$\bullet$~ $\mu$ is \emph{strictly
determinate}\index{Multidimensional determinacy notions! strictly determinate} if\, $\mu$ is determinate and\,
$\C_d[\ux]$ is dense in $L^2(\R^d,\mu)$,\smallskip

$\bullet$~ 
$\mu$ is \emph{strongly determinate}\index{Multidimensional determinacy notions! strongly determinate} if\, $\C_d[\ux]$ is dense in $L^2(\R^d,(1{+}x_j^2)d\mu)$ for  $j{=}1,\dots,d$,\smallskip

$\bullet$~ $\mu$ is \emph{ultradeterminate}\index{Multidimensional determinacy notions! ultradeterminate} if\,  $\C_d[\ux]$ is dense in $L^2(\R^d,(1+\|x\|^2)d\mu)$.
\end{definition}

 Suppose  $\mu \in \cM_+(\R^d)$.
 Let $s$  be the moment sequence of $\mu$ and $L$ the moment functional of $\mu$.
 Then we  say that $s$, or $L$, is \emph{determinate, strongly determinate, strictly determinate, ultradeterminate}, if $\mu$ has this property. 
%Let us  set $\cM_s=\cM_L=\cM_\mu$.

Note that all four determinacy notions are defined in terms of the measure $\mu$!

 If $\mu$ has a \emph{compact} support, it follows from the Weierstrass  theorem that it is ultradeterminate, so    all four determinacy notions are valid in this case. 

Further, in the case $d=1$ it follows at once from Theorem \ref{mriesz} that the four
concepts 
(determinacy, strict determinacy, strong determinacy, ultradeterminacy) 
are equivalent. However, in dimension $d\geq 2$
all of them are  different! Counterexamples can be found in [MP]. 

 More surprisingly, as shown by Berg and Thill (1991),   there exist \emph{determinate} measures $\mu\in \cM_+(\R^d)$, $d\geq 2$, such that $\C_d[\ux]$ is \emph{not} dense in $L^2(\R^d,\mu)$.
 Such measures are  not strictly determinate.

\smallskip
By definition strict determinacy implies determinacy. 
Theorem \ref{op-det} below shows that if $\mu$ is strongly  determinate it is  strictly
determinate and hence determinate.
Since the
norm of $L^2(\R^d,(1+||x||^2)d\mu)$ is obviously stronger than that 
of $L^2(\R^d,(1+x_j^2)d\mu)$,  ultradeterminacy
implies strong determinacy. 
Thus we have the following implications:
\begin{align*}
 ultradeterminate \Rightarrow strongly ~determinate\Rightarrow strictly ~determinate\Rightarrow  determinate.
\end{align*}

\section{Determinacy via marginal measures}\label{polynomialapproximation}

 Suppose $\mu$ is a Radon measure on $\R^d$. 
 Let $\pi_j(x_1,\dots,x_d)=x_j$ denote the $j$-th coordinate mapping of $\R^d$ into $\R$.
 The$ j$-th \emph{marginal measure} $\pi_j(\mu)$ is the Radon measure on $\R$ defined by $\pi_j(\mu)(M):=\mu(\varphi^{-1}(M))$ for any Borel set $M$ of $\R$.
 Then the
transformation formula
\begin{align}\label{immeasure}
\int_{\R} f(y)~ d\pi_j(\mu)(y) = \int_{\R^d} f(x_j)~
d\mu(x)
\end{align}
holds for any function $f \in \cL^1(\R,\pi_j(\mu))$.

 The following basic result is \textbf{Petersen's theorem} (1982).
 
\begin{theorem}\label{pet}
Let\, $\mu\in \cM_+(\R^d)$. If all marginal measures\, $\pi_1(\mu),\dots,\pi_d(\mu)$ are determinate, then $\mu$\,  itself is determinate.
\end{theorem}
\begin{proof}
%\cite[Theorem 14.6]{mp}.
 {[MP, Theorem~ 14.6]}.
\qed\end{proof}

 The converse of Theorem \ref{pet} does not hold.
 That is, there exists a determinate measure of $\mu\in \cM_+(\R^2)$ such that $\pi_1(\mu)$ and $\pi_2(\mu)$ are not determinate (see , e.g., [MP, Exercise~14.7]).

\section{The multivariate Carleman condition and Nussbaum's theorem}\label{carlemanmultidetcase}

 From Lecture 2 we recall the \emph{Carleman condition}  for  a positive semidefinite $1$-sequence $(t_n)_{n\in \N_0}$:
\begin{align}\label{carlemand1}
%\sum_{n=1}^\infty  t_{2n}^{-1/{(2n)}}
\sum_{n=1}^\infty \,t_{2n}^{-\frac{1}{2n}} =+\infty.
\end{align}
For a $d$-sequence $s=(s_\gn )_{\gn\in \N_0^d}$ the $1$-sequences
\begin{align}\label{det4}
s^{[1]}:=(s_{(n,0,\dots,0))})_{n \in \N_0},\, s^{[2]}:=(s_{(0,n,\dots,0)})_{n \in \N_0} , \dots, s^{[n]}:=(s_{(0,\dots,0,n)})_{n \in \N_0}
\end{align}
 are called  \emph{marginal sequences} of $s$.
 If $s$ is the moment sequence of $\mu\in \cM_+(\R^d)$, then it follows from formula (\ref{immeasure}) that 
\begin{align*}
s^{[j]}_n=\int_{\R^d} x_j^n ~d\mu(x_1,\dots,x_d)= \int_\R y^n\, d\pi_j(\mu)(y), \quad n\in \N_0, j=1,\dots,d,
\end{align*}
that is, the $j$-th marginal sequence $s^{[j]}$ is just the moment sequence of the $j$-th marginal measure $\pi_j(\mu)$ of $\mu$. 
Thus, Petersen's  Theorem \ref{pet} says that a $d$-moment sequence $s$ is determinate if all marginal sequences\, $s^{[1]}, \dots, s^{[d]}$\,  are determinate.

 A $d$-sequence $s=(s_\gn )_{\gn\in \N_0^d}$ is called \emph{positive semidefinite} if
\[
\sum_{\gn,\gm\in \N_0^d}  s_{\gn+\gm}\, \xi_\gn \xi_\gm\geq 0
\]
for all finite real multisequences $(\xi_\gn )_{\gn\in \N_0^d}$. It is easily verified that $s$ is positive semidefinite if and only if its Riesz functional $L_s$ is positive, that is, $L_s(p^2)\geq 0$ for  $p\in \R_d[\ux]$.
\begin{definition}\label{multicarle}
Let $s$ be a positive semidefinite $d$-sequence. We shall say that $s$, and equivalently the   functional $L_s$, satisfy the 
\emph{multivariate Carleman condition}\index{Carleman condition! multivariate} if all marginal sequences $s^{[1]},\dots, s^{[n]}$ satisfy  Carleman's condition (\ref{carlemand1}), that is, if
\begin{align}\label{multicarlemancon}
 \sum_{n=1}^\infty  (s_{2n}^{[j]})^{-\frac{1}{2n}}
 \equiv \sum_{n=1}^\infty\, L_s(x_j^{2n})^{-\frac{1}{2n}}
 =+\infty \quad \text{ for }\quad j=1,\dots,d.
\end{align}
\end{definition}

 The following fundamental result is \textbf{Nussbaum's theorem} (1965).
 
\begin{theorem}\label{carlemenmaindetex}
Each positive semidefinite $d$-sequence $s=(s_\gn)_{\gn \in \N_0^d}$ satisfying the multivariate Carleman condition is a strongly determinate moment sequence.
\end{theorem}
Theorem \ref{carlemenmaindetex} follows at once from the following  more general result.
\begin{theorem}\label{qavector}
Suppose  $s=(s_\gn)_{\gn \in \N_0^d}$ is a positive semidefinite $d$-sequence such that the first $d{-}1$ marginal sequences $s^{[1]}, \dots, s^{[d-1]}$
fulfill  Carleman's condition (\ref{carlemand1}). Then $s$ is a moment sequence.

 If in addition the  sequence $s^{[d]}$ satisfies
Carleman's condition (\ref{carlemand1}) as well, then the moment sequence $s$ is  strongly determinate.
\end{theorem}
\begin{proof}
%\cite[Theorem 14.6]{mp}.
 {[MP, Theorem~14.6]}.
\qed\end{proof}

 Theorem \ref{carlemenmaindetex}  is  a  strong and very useful result.
 It  shows that for any positive semi-definite $d$-sequence the multivariate Carleman condition implies the existence \emph{and} the uniqueness  of a solution of the moment problem!

We illustrate the usefulness of Theorem \ref{carlemenmaindetex} 
by the following application.
\begin{corollary}\label{evarepsilonnormx}
Let $\mu \in M_+(\R^d)$. Suppose that  there exists an $\varepsilon>0$ such that
\[
\int_{\R^d} e^{\varepsilon \|x\|}~ d\mu(x)<+\infty.
\]
Then $\mu\in \cM_+(\R^d)$ and  $\mu$ is  strongly determinate.
\end{corollary}
\begin{proof}
The proof is similar to the proof of Corollary \ref{evarepsilonx^2} in Lecture 2. Let $j\in \{1,\dots,d\}$ and $n\in \N_0$. Then\, $x_j^{2n}e^{-\varepsilon |x_j|}\leq \varepsilon^{-2n}(2n)!$\, for $x_j\in \R$ and hence
\begin{align}\label{x2ne-normx}
\int_{\R^d} x_j^{2n}\,d\mu =\int_{\R^d} x_j^{2n}e^{-\varepsilon |x_j|}e^{\varepsilon |x_j|}\, d\mu \leq \varepsilon^{-2n}(2n)! \int_{\R^d} e^{\varepsilon \|x\|}~ d\mu<+\infty.
\end{align}
Let $p\in \R_d[\ux]$. Then $p(x)\leq c(1+x_1^{2n}+\dots+x_d^{2n})$ on $\R^d$ for some $c>0$ and $n\in \N$, so (\ref{x2ne-normx}) implies that $p$ is $\mu$-integrable. Thus $\mu\in \cM_+(\R^d).$ 

Let $s$ be the moment sequence of $\mu$. By  (\ref{x2ne-normx})  there is a constant $M>0$ such that 
\begin{align*} 
 s^{[j]}_{2n}
 =L_s(x_j^{2n})
 =\int_{\R^d} x_j^{2n}\, d\mu \leq M^{2n}(2n)! \quad \text{ for }~~~n\in \N_0.
\end{align*}
 By Corollary  \ref{analyticcase},  this inequality implies that the marginal sequence $s^{[j]}$ satisfies Carleman's condition  (\ref{carlemand1}). Therefore, $s$ and $\mu$ are strongly determinate by combining Theorems \ref{carlemenmaindetex} and \ref{mriesz}.
$\hfill\qed$ \end{proof}

 Carleman's condition can be  also used to localize the  support of representing measures, as the following result of Lasserre (2013) shows.
 
\begin{theorem}\label{carlemansupport}
 Let $s$ be a real $d$-sequence and  $\mathsf{f}=\{f_1,\dots,f_k\}$  a finite subset of $\R_d[\ux]$.
 Suppose the Riesz functional $L_s$ is $Q(\mathsf{f})$-positive (that is, $L_s(p^2)\geq 0$ and $L_s(f_jp^2) \geq 0$ for  $j=1,\dots,k$, $p\in \R_d[\ux]$) and satisfies the multivariate Carleman condition. Then the unique representing measure of the  determinate moment sequence $s$  is supported on the semi-algebraic set 
$\cK(\mathsf{f})$.
\end{theorem}
\begin{proof}
%\cite[Theorem 14.25]{mp}.
 {[MP, Theorem~14.25]}.
\qed \end{proof}

\section{Some operator-theoretic reformulations}
Apart from  real algebraic geometry,  
the operator theory on Hilbert space is another powerful tool for the study of the multidimensional moment problem. This connection  is not treated in these lectures. In this final section we briefly touch a few operator-theoretic points. First we recall some basic notions.

 Suppose $\cH$ is a complex Hilbert space with scalar product $\langle\cdot, \cdot\rangle$.
 A \emph{symmetric operator}\, on $\cH$ is a linear mapping $T$ of a linear subspace $\cD(T)$, called the domain of $T$, into $\cH$ such that
 $\langle T\varphi,\psi\rangle=\langle \varphi,T\psi\rangle$\, for\, $ \varphi,\psi\in \cD(T)$

%If $T_1$ and $T_2$ are linear operators on $\cH$, we say that $T_2$ is an {\it extension} of $T_1$ and write $T_1\subseteq T_2$ if $\cD(T_1)\subseteq %\cD(T_2)$ and $T_1\varphi=T_2\varphi$ for $\varphi\in \cD(T_1)$.  

 An operator $T$ is called \emph{closed} if for each sequence $(\varphi_n)$ from $ \cD(T)$ such that $\varphi_n\to \varphi$ and $T\varphi_n\to \psi$ in $\cH$ then $\varphi\in \cD(T), \psi=T\varphi$.
 If $T$ has a closed extension,  it has a smallest closed extension, called the \emph{closure} of $T$ and denoted by $\ov{T}$.
%For a closed operator $T$ the {\it resolvent set}\index{Operator! resolvent set}  $\rho(T)$ is the set of  numbers $z\in \C$ for which the operator %$T-z I$ has a bounded inverse $(T-z I)^{-1}$ that is defined on the whole Hilbert space $\cH$. The set %$\sigma(T)\index[sym]{rGhoT@$\rho(T)$}\index[sym]{sGigmaT@$\sigma(T)$}:=\C\backslash \rho(T)$ is the {\it spectrum}\index{Operator! spectrum} of $T$.

 Suppose that $\cD(T)$ is dense in $\cH$.
 Then the \emph{adjoint operator}\index{Operator! adjoint} $T^*$\index[sym]{TAstar@$\ov{T}, T^*$} of $T$ is defined:
 Its domain $\cD(T^*)$ is  the set of vectors $\psi\in \cH$ for which there exists an $\eta\in \cH$ such that $\langle T\varphi,\psi\rangle=\langle \varphi,\eta\rangle $ for all $\varphi\in \cD(T)$; in this case $T^*\psi=\eta$. 

%The operator  $T$ is symmetric if and only $T\subseteq T^*.$ 

 A densely defined operator $T$ is called \emph{self-adjoint}\index{Operator! self-adjoint} if $T=T^*$ and 
\emph{essentially self-adjoint}\index{Operator! essentially self-adjoint} if its closure $\ov{T}$ is self-adjoint, or equivalently, if $\ov{T}=T^*$.

We recall the following well-known fact from operator theory. 
%\cite[Proposition 3.8]{schm2012}.

\begin{lemma}\label{essent} A densely defined symmetric operator $T$   is essentially self-adjoint if and only if
the  subspaces $(T-{\ii})\cD(T)$ and $(T+{\ii})\cD(T)$ are dense in $\cH$.
\end{lemma}

%\begin{lemma}\label{essent} Suppose $T$ is densely defined symmetric operator $T$ on $\cH$. Then $T$  is essentially self-adjoint if and only if
%the  subspaces $(T-{\ii}\cdot I)\cD(T)$ and $(T+{\ii}\cdot I)\cD(T)$ are dense in $\cH$.
%\end{lemma}

Now we  reformulate Theorem \ref{mriesz}. Suppose $\mu\in \cM_+(\R)$. Let $X$ denote the multiplication operators by the variable $x$ on the Hilbert space $L^2(\R,\mu)$, that is, $(Xp)(x)=xp(x)$\, for\, $ p\in \cD(X):=\C[x]$. It is easily checked that $X$ is symmetric. 

For a function $f$ of $\R$ we set $g_\pm:=(x\pm {\ii})f$. Then, for $p\in \C[x]$,
\begin{align}\label{gpmf}
\int_\R |g_\pm(x)-(x\pm {\ii})p(x)|^2 d\mu(x)=\int_\R |f(x)-p(x)|^2 (1+x^2)d\mu(x)
\end{align}
 Since $g_\pm \in L^2(\R,\mu)$ if and only if $f\in L^2(\R,(1+x^2)d\mu)$, it follows from (\ref{gpmf}) that $(X\pm {\ii} )\cD(X)$ is dense in $L^2(\R,\mu)$ if and only if $\C[x]$ is dense in  $L^2(\R,(1+x^2)d\mu)$.
 Therefore, by Lemma \ref{essent}, the \emph{measure $\mu$ is determinate if and only if the multiplication operator $X$ on $L^2(\R,\mu)$ is essentially self-adjoint}.
 This is an interesting operator-theoretic characterization of determinacy.

Strong determinacy has also a nice operator-theoretic interpretation. For a measure $\mu\in \cM_+(\R^d)$, let $X_j$ denote the multiplication operator by the variable $x_j$ with domain $\cD(X_j)=\C_d[\ux]$ in the Hilbert space $L^2(\R^d,\mu)$.
\begin{theorem}\label{op-det}
Suppose $L$ is a moment functional on $\C_d[\ux]$ and  $\mu$ is a representing
measure for $L$. Then $\mu$ is strongly determinate if and only if the symmetric operators $X_1,\dots,X_d$
are essentially self-adjoint.
If this holds,  then $\mu$ is strictly determinate, that is, $\mu$ is determinate 
 and 
$\C_d[\ux]$ is dense in $L^2(\R^d,\mu)$.
\end{theorem}
\begin{proof} 
%\cite[Theorem 4.2]{mp}.
 {[MP, Theorem~4.2]}.
\qed \end{proof}

 Finally, we relate   Carleman's condition  to the notion of a \emph{quasi-analytic vector}.

%In the operator-theoretic approach  Carleman's condition is closely related to  {\it quasi-analytic vectors}. We briefly discuss this connection. 

Let $T$ be a symmetric linear operator on a Hilbert  space $\cH$ and  $\varphi\in \cap_{n=1}^\infty \cD(T^n).$ Since the operator $T$ is symmetric, it is easily verified that the real sequence
\begin{align*}t=(t_n:=\langle T^n\varphi,\varphi\rangle)_{n\in \N_0}\end{align*} 
is  positive semidefinite and hence a moment sequence by Hamburger's theorem \ref{solutionhmp}.
 The vector $\varphi$ is called \emph{quasi-analytic}\index{Quasi-analytic vector} for $T$ if
\[
\sum_{n=1}^\infty \|T^n \varphi\|^{-1/n}=+\infty.
\]
Note that\, $t_{2n}=\langle T^{2n}\varphi,\varphi\rangle= \|T^n\varphi\|^2$\, for $n\in \N_0$. Hence 
 the vector $\varphi$ is  quasi-analytic  for $T$ if and only if the sequence $t$ satisfies Carleman's condition (\ref{carlemand1}). 

Now suppose  $\mu\in \cM_+(\R^d)$. Let  $s$ be the moment sequence of $\mu$ and  $L_s$  the corresponding Riesz functional on $\R_d[\ux]$. On the Hilbert space $L^2(\R^d,\mu)$ we have
\begin{align*}
s^{[j]}_{2n} =L_s(x_j^{2n})=\int x_j^{2n} d\mu(x)=  \langle X_j^{2n} 1,1 \rangle =\|X_j^n\, 1\|^2, \quad j=1,\dots,d,\, n\in \N_0.
\end{align*} 
 Therefore, from the preceding discussion it follows that the  sequence\,  $s$\, fulfills the \emph{multivariate Carleman condition} (according to Definition \ref{multicarle}) if and only if the  constant function $\varphi=1$ is a \emph{quasi-analytic vector} for the multiplication operators $X_1,\dots,X_d$ by the variable $x_1,\dots,x_d$, respectively.

\

\makeatletter
\renewcommand{\@chapapp}{Lecture}
\makeatother
%\chapter{}
%\newpage
%\chapter{}

\chapter{Polynomial optimization and semidefinite programming}

Abstract:\\
\textit{Semidefinite programming is introduced. Positivstellens\"atze and the moment problem provide methods for determining the minimum of a  polynomial over a semi-algebraic set via Lasserre relaxation.
From the Archimedean Positivstellensatz   convergence results are obtained.}

\bigskip

\section{Semidefinite programming}\label{semidefiniteprogr}

Semidefinite programming is a generalization of linear programming, where the linear constraints are replaced by linear \emph{matrix} constraints. 

 Let ${\Sym}_n$ denote the  real symmetric $n\times n$-matrices and $\langle A,B\rangle: ={\Tr} \, AB.$
 $A\succeq 0$ means that the matrix $A$ is positive semidefinite and $A \succ 0$  that $A$ is positive definite. 
 
 \smallskip
 
We define a semidefinite program and its dual program.
Suppose  $b\in \R^m$ and $m+1$ matrices $A_0,\dots,A_m\in {\Sym}_n$ are given.
\smallskip

 The primal \emph{semidefinite program} (SDP) is the following:
\begin{align}\label{defstandardsp}
 p_*=\inf_{y\in \R^m} \big\{ b^Ty:~ A(y):=A_0+y_1A_1+\dots+y_mA_m\succeq 0\}.
\end{align}
 That is, one minimizes the linear function\, $b^T y=\sum_{j=1}^m b_jy_j$\, in a vector variable\, $y=(y_1,\dots,y_m)^T\in \R^m$ subject to the \emph{linear matrix inequality}\index{Linear matrix inequality} (LMI) constraint 
\begin{align}\label{spectrahedron}
A(y):=A_0+y_1A_1+\dots+y_mA_m\succeq  0.
\end{align}
 The set of  points $y\in \R^m$ satisfying  (\ref{spectrahedron}) is called a \emph{spectrahedron}.
 Since $A(y)\succeq 0$  if and only if  its principal minors  are nonnegative and each such minor 
 is a polynomial,  a spectrahedron is a closed semi-algebraic set.
By (\ref{spectrahedron}), a spectrahedron is  also  convex. A  convex  closed semi-algebraic set is not necessarily a spectrahedron.

If all matrices $A_j$ are diagonal,  the constraint $A(y)\succeq 0$ consists  of linear inequalities, so  (\ref{defstandardsp}) is a linear  program. Conversely, each linear problem is a semidefinite program by writing the linear constraints as an LMI with a  diagonal matrix. 
\smallskip

 The \emph{dual program} associated with (\ref{defstandardsp}) is defined by
\begin{align}\label{dualsp}
 p^*=\sup_{Z\in {\Sym}_n} \{-\langle  A_0,Z\rangle: Z\succeq 0~~\text{ and }~~ \langle A_j,Z\rangle =b_j,~j=1,\dots,m\}.
\end{align}
Thus, one maximizes the linear function\, $-\langle A_0,Z\rangle=-{\Tr}\, A_0 Z$~ in a matrix variable $Z\in {\Sym}_n$ subject to the constraints\, $Z\succeq 0$\, and\, 
$\langle A_j,Z\rangle = {\Tr}\, A_jZ=b_j$.  It can be shown that the dual program (\ref{dualsp}) is also a semidefinite program.

 A vector $y\in \R^m$ resp. a matrix $Z\in {\Sym}_n$   is  called \emph{feasible}\index{Semidefinite program ! feasible vector}  for (\ref{defstandardsp}) resp. (\ref{dualsp}) if it satisfies the corresponding constraints.
 If there are no feasible points, we set $p_*=+\infty$ resp. $p^*=-\infty$.
 A program  is called \emph{feasible} if it has a feasible  point. 

\begin{proposition}\label{p_xp^*}
If $y$   is  feasible for (\ref{defstandardsp}) and $Z$   is feasible for (\ref{dualsp}), then
\begin{align}\label{dualitydiff}
b^Ty\geq p_*\geq p^*\geq-\langle A_0,Z\rangle.
\end{align}
\end{proposition}
\begin{proof}
Since $y$ is feasible for (\ref{defstandardsp}) and $Z$ is feasible for (\ref{dualsp}), $A(y)\succeq 0$ and $Z\succeq 0$. Therefore,  $\langle A(y),Z\rangle \geq 0$, so using  (\ref{dualsp}) we derive
\begin{align*}\label{anullx}
b^Ty =\sum_{j=1}^m y_j\langle A_j,Z\rangle=\langle A(y),Z\rangle -\langle A_0,Z\rangle\geq-\langle A_0,Z\rangle.
\end{align*}
Taking the infimum over $y$ and the supremum over $Z$ we obtain (\ref{dualitydiff}).
\qed \end{proof}

 In contrast to linear programming,   $p_*$ is not equal to $p^*$ in general.
 The number $p_*-p^*$ is called the \emph{duality gap}.
 A simple example of a semidefinite program with positive duality gap is the following:
Minimize $x_1$ subject to the constraint
 \begin{gather*}
\left(
\begin{array}{lll}
0 & x_1 & ~~0\\
x_1 & x_2 & ~~0\\
0 & 0 & x_1+1
\end{array}\right) \succeq 0.
\end{gather*}
In this case, we have $p_*=0$ and $p^*=-1$.

 The next  proposition shows that, under the stronger assumption of \emph{strict feasibility},  the duality gap is zero.
\begin{proposition}\label{stricltydeasible}
\begin{itemize}
\item[(i)]~If\, $p_*>-\infty$ and (\ref{defstandardsp}) is strictly feasible (that is, there is a $y\in \R^m$ such that $A(y)\succ 0$), then  $p_*=p^*$ and the supremum in (\ref{dualsp}) is a maximum.
\item[(ii)]~
If $p^*<+\infty$ and (\ref{dualsp}) is strictly feasible (that is, there is a  $Z\in {\Sym}_n$ such that $Z\succ0$ and $\langle A_j,Z\rangle =b_j, j=1,\dots,m$), then  $p_*=p^*$ and the infimum in (\ref{defstandardsp}) is a minimum.
\end{itemize}
\end{proposition}
\begin{proof} 
%\cite[Proposition 16.2]{mp}.
 The proof is based on separation of convex sets, see [MP, Proposition~16.2].
$\qed$ \end{proof}

We give two examples where semidefinite programs appears.

\begin{example}(\textit{Largest eigenvalue of a symmetric matrix})\\
Let $\lambda_{\max}(B)$ denote the largest eigenvalue of $B\in {\Sym}_n$. Then
\begin{align*}
\lambda_{\max}(B)=\min_{y\in \R}~ \{ y: (yI-B)\succeq 0\}
\end{align*}
gives\, $\lambda_{\max}(B)$  by a semidefinite program. Since this program and its dual are strictly feasible (take $y\in \R$ such that $(yI-B)\succ 0$ and $Z=I$), Proposition \ref{stricltydeasible} yields 
\begin{align*}
\lambda_{\max}(B)=\max_{Z\in {\Sym}_n}~ \{\langle B,Z\rangle: Z\succeq 0, \langle I,Z\rangle =1\} .
\end{align*}
There is also a semidefinite program for the sum of the $j$ largest eigenvalues. $\hfill \circ$
\end{example}

\begin{example}\label{feasibilitysos} (\textit{Sos representation of a polynomial})\\
 A polynomial $f(x)=\sum_{\alpha} f_\alpha x^\alpha\in \R_d[\ux]_{2n}$ is in 
 $\sum \R_d[\ux]_n^2$ if and only if there exists a positive semidefinite matrix\,  $G$ such that
\begin{align*}
 f(x)
 =(\gx_n)^T G \gx_n,~~\text{ where }~~ \gx_n =(1,x_1,\dots,x_d,x_1^2,x_1x_2,\dots,x_d^2,x_1^3,\dots,x_1^n,\dots,x_d^n)^T.
\end{align*}
If we write  $\gx_n (\gx_n)^T=\sum_{\alpha} A_\alpha x^\alpha$ with $A_\alpha \in {\Sym}_{d(n)}$ and  compare coefficients in the equation $f(x)=(\gx_n)^T G \gx_n$, we obtain $\langle G,A_\alpha \rangle ={\Tr}\, GA_\alpha =f_\alpha$ for\,   $\alpha \in \N_0^d, |\alpha|\leq {2n}$. Therefore, $f\in \sum \R_d[\ux]_n^2$\, if and only if there exists  a matrix $G\in {\Sym}_{d(n)}$ such that
\begin{align}\label{sosfeas}
 G\succeq 0 \quad\text{ and }~~~ \langle Q,A_\alpha\rangle =f_\alpha \quad\text{ for }~~\alpha\in \N_0^d, |\alpha|\leq 2n.
\end{align}
 Hence $f$ is a sum of squares if and only if the \emph{feasibility condition} (\ref{sosfeas}) of the corresponding semidefinite program is satisfied.  $\hfill \circ$
 \end{example}
 
\section{Lasserre  relaxations of polynomial optimization with constraints}\label{tworelaxations}

 We fix  $\mathsf{f}=\{f_0,\dots,f_k\}$, where  $f_j\in \R_d[\ux]$, $f_j\neq 0, f_0=1$, and  $p\in \R_d[\ux], p\neq 0$. 

Our aim is to minimize the polynomial\, $p$\, over the semi-algebraic set $ \cK(\mathsf{f})$:
\begin{align}\label{pmin}
p^{\min}:=\inf \{p(x):~ x\in \cK(\mathsf{f})\}.
\end{align}
Clearly,
\begin{align}\label{strong1}
 p^{\min}&=\sup \big\{ \lambda \in \R: \, p-\lambda \geq 0~~\text{ on }~~ \cK(\mathsf{f})\big\}.\\
p^{\min}&=\inf \big\{ L(p): L\in \,\R_d[\ux]^*, L(1)=1, L(f)\geq 0 ~~\text{ if }~~ f \geq 0~~\text{ on }~~ \cK(\mathsf{f})\big\}.\label{strong2}
\end{align}
 The idea is to weaken the constraints in (\ref{strong1})  and (\ref{strong2}):   We  require in (\ref{strong1})  that $p-\lambda $ is a certain weighted sum of squares which is nonnegative on 
 $ \cK(\mathsf{f})$ and  consider in (\ref{strong2})    functionals $L$ which are nonnegative only on certain weighted  sums of squares.
\smallskip

For $n \in \N_0$ let  $  n_j$ be the largest integer such that $n_j\leq \frac{1}{2}(n-\deg(f_j))$. Set
\begin{align*}
Q(\mathsf{f})_n &:
%=\bigg\{ \sum_{j=0}^k f_j\sigma_j:\,\sigma_j\in \sum \R_d[\ux]^2, \deg(f_j\sigma_j)\leq n\bigg\}\\&
=\bigg\{ \sum_{j=0}^k f_j\sigma_j:\, \sigma_j\in \sum \R_d[\ux]^2_{n_j}\bigg\},\\
Q(\mathsf{f})_n^*&:=\{ L\in\R_d[\ux]_n^*: L(1)=1 ~~\text{ and }~~ L(g)\geq 0,~   g\in  Q(\mathsf{f})_n\}.
\end{align*}
%Let $Q(\mathsf{f})_n^*$  denote the set of     linear functionals $L$ on $\R_d[\ux]_n$ such that $L(1)=1$ and $L(g)\geq 0$ for  $g\in  Q(\mathsf{f})_n$.

 Now we define two \textbf{Lasserre relaxations} by
\begin{align}
p^{\mom}_n &:=\inf \big\{ L(p): \, L \in Q(\mathsf{f})_n^*\big\},\label{defpmomn}\\
p^{\sos}_n&:=\sup \big\{ \lambda \in \R: \, p-\lambda \in Q(\mathsf{f})_n\big\},\label{defpsosn}
\end{align}
where we set $p^{\sos}_n=-\infty$ if there is no $ \lambda \in \R$ such that  $p-\lambda \in Q(\mathsf{f})_n$.

Both relaxations (\ref{defpmomn}) and (\ref{defpsosn}) can be reformulated in terms of a semidefinite program and its dual.
Let us begin with (\ref{defpmomn}).

A linear functional $L$ on $\R_d[\ux]_n$ is completely  described by the $\binom{d+n}{n}$ variables $y_\alpha:=L(x^\alpha)$,  $\alpha\in \N_0^d$, $|\alpha|\leq n$.
By definition, $L\in
Q(\mathsf{f})_n^*$ if and only if $L(1)=1$ and 
\begin{align}\label{lfjcondition}
L(f_jq^2)\geq 0~~~\text{ for }~~~q\in \R_d[\ux]_{n_j},~ j=0,\dots,k.
\end{align}
We reformulate these conditions in terms of the variables $y_\alpha$. Clearly, $L(1)=1$ means  $y_0=1$.
Let  $f_j=\sum_\alpha f_{j,\alpha}x^\alpha$. We denote by $H_{n_j}(f_jy)$   the $\binom{d+n_j}{n_j} \times\binom{d+n_j}{n_j}$-matrix  with entries 
\begin{align}\label{matentries}
H_{n_j}(f_jy)_{\alpha,\beta}:=\sum\nolimits_{\gamma} f_{j,\gamma}y_{\alpha+\beta+\gamma},\quad\text{ where }\quad |\alpha|\leq n_j, |\beta|\leq n_j. 
\end{align}
Thus $H_{n_j}(f_y)$ is a  localized  Hankel matrix. For $q=\sum_\alpha a_\alpha x^\alpha\in \R_d[\ux]_{n_j}$ we compute
\begin{align*}
L(f_jq^2)=\sum_{\alpha,\beta,\gamma} f_{j,\gamma} a_\alpha a_\beta L(x^{\alpha+\beta+\gamma})=\sum_{\alpha,\beta} H_{n_j}(f_jy)_{\alpha,\beta} a_\alpha a_\beta,
\end{align*}
so (\ref{lfjcondition}) holds if and only if the matrices $H_{n_j}(f_jy)$, $j=1,\dots,k$, are  positive semidefinite. 
Let $H(\mathsf{f})(y)$ denote the block diagonal  matrix with diagonal blocks $H_{n_0}(f_0y),H_{n_1}(f_1y),$ $\dots, H_{n_k}(f_ky)$. Then (\ref{lfjcondition}) is satisfied  if and only if  $H(\mathsf{f})(y)\succeq 0.$ 
The matrix $H(\mathsf{f})(y)$ has type  $N\times N$, where $N:=\sum_{j=0}^k\binom{d+n_j}{n_j}$.

 By (\ref{matentries}), all entries of $H(\mathsf{f})(y)$ are linear in the variables $y_\alpha$. Hence, inserting  $y_0=1$, there are (constant!) real symmetric $N\times N$-matrices $A_\alpha$ such that 
\begin{align}\label{defmsf}
 H(\mathsf{f})(y)= A_0+\sum_{\alpha\in \N_0^d,0<|\alpha|\leq n} y_\alpha A_\alpha.
\end{align}

Let
  $p(x)=\sum_\alpha p_\alpha x^\alpha$. Then $L(p)=p_0+\sum_{\alpha\neq 0}~ p_\alpha y_\alpha$.  By (\ref{defpmomn}), $p^{\mom}_n$ is the infimum of  
   $L(p)$ in the $M:=\binom{d+n}{n}- 1$ variables $y_\alpha$ subject to  $H(\mathsf{f})(y)\succeq 0$.

 Summarizing,  the relaxation  (\ref{defpmomn}) leads to the \textbf{primal semidefinite program}
\begin{align}\label{emprogmom}
p^{\mom}_n-p_0=\inf_{(y_\alpha)\in \R^M}\, \bigg\{~ \sum_{0<|\alpha|\leq n}~ p_\alpha y_\alpha:~  A_0+\sum_{0<|\alpha|\leq n} y_\alpha A_\alpha\succeq 0 \bigg\}.
\end{align}
\smallskip

Next we show that
 (\ref{defpsosn}) leads to the corresponding dual  program.
 Suppose  that $\lambda \in \R$ and $ p-\lambda \in Q(\mathsf{f})_n$, that is,
 \begin{align}\label{p-lambdaeq}
 p-\lambda =\sum_{j=0}^k f_j\sigma_j,\quad  \sigma_j\in \sum \R_d[\ux]^2_{n_j}.
 \end{align} 
Then   $\sigma_j\in \sum \R_d[\ux]_{n_j}^2$ if and only if there is a matrix $Z(j)\succeq 0$ such that
\begin{align*}
\sigma_j(x)=\sum_{|\alpha|,|\beta|\leq n_j} Z(j)_{\alpha,\beta} x^{\alpha+\beta}.
\end{align*}
Let $Z$ be the block diagonal $N\times N$-matrix with blocks $Z(0),\dots, , Z(k)$.
 Clearly, $Z\succeq 0$\, if and only if $Z(j)\succeq 0$\, for all $j$.
 
  Recall that   $f_j=\sum_\alpha f_{j,\alpha}x^\alpha.$  
  %From (\ref{defmsf}) it follows that the $j$-th diagonal block of the matrix $A_\alpha$  has the $(\beta,\gamma)$-matrix entry $\sum_{\delta} %f_{j,\delta}$, where the  summation is over all $\delta$ satisfying $\alpha=\beta+\gamma+\delta$. Hence, e
  Equating coefficients in (\ref{p-lambdaeq}) yields
\begin{align*}
 p_0-\lambda &=\sum_{j=0}^k ~f_{j,0}Z(j)_{00} =\TR A_0Z =\langle A_0,Z\rangle,\\
 p_\alpha &=\sum_{j=0}^k~ \sum_{\beta+\gamma+\delta=\alpha}\ f_{j,\delta} Z(j)_{\beta,\gamma} =\TR A_\alpha Z =\langle A_\alpha,Z\rangle , ~~~\alpha\neq 0.
 \end{align*}
Clearly, taking the supremum 
of\, $\lambda$ in (\ref{defpsosn}) is equivalent to  taking the supremum
of\, $\lambda{-}p_0=-\langle A_0,Z\rangle$ subject to the conditions $p_\alpha=\langle A_\alpha,Z\rangle$,   $0<|\alpha|\leq n$.

 Thus, the second relaxation (\ref{defpsosn}) leads to the corresponding \textbf{dual  program} 
 \begin{align}\label{semidprsos}
 p^{\sos}_n{-} p_0= \sup_{Z\in {\Sym}_N} \big\{- \langle A_0,Z\rangle:~ Z\succeq 0 ,~~ p_\alpha  =\langle A_\alpha,Z\rangle ~~\text{ for }~~ 0<| \alpha|\leq n \big\}.
 \end{align}

\section{Polynomial optimization with constraints}\label{polyoptwithconstraint}
%Let us retain the notation from the preceding section.

First we note that for $n\in \N$,
\begin{align}\label{psosinq1}
p^{\sos}_n\leq p^{\sos}_{n+1}\,~~~~ &\text{ and }~~~~\, p^{\mom}_n\leq p^{\mom}_{n+1}\\
p^{\sos}_n\leq & p^{\mom}_n\leq p^{\min}.\label{psosinq2}
\end{align}

We give the simple proofs of these inequalities.  Since $Q(\mathsf{f})_n$ is a subspace of  $Q(\mathsf{f})_{n+1}$,    $ p-\lambda \in Q(\mathsf{f})_n$ implies that $ p-\lambda \in Q(\mathsf{f})_{n+1}$, so  $p^{\sos}_n\leq p^{\sos}_{n+1}$. Since  restrictions of functionals from  $Q(\mathsf{f})_{n+1}^*$ belong to $Q(\mathsf{f})_n^*$, it follows that $p^{\mom}_n\leq p^{\mom}_{n+1}$.

Polynomials of $Q(\mathsf{f})$ are nonnegative on $ \cK(\mathsf{f})$. Hence each point evaluation at $x\in  \cK(\mathsf{f})$ is in $Q(\mathsf{f})_n^*$, so $p^{\mom}_n\leq p(x)$  for  $x\in  \cK(\mathsf{f})$. This implies $p^{\mom}_n\leq p^{\min}$. 

Let $L \in Q(\mathsf{f})_n^*$. If $p-\lambda \in Q(\mathsf{f})_n$, then $L(p-\lambda)=L(p)-\lambda \geq 0$,  that is, $\lambda \leq L(p).$ Taking the supremum over $\lambda$ and the infimum over $L$ we get $p^{\sos}_n\leq p^{\mom}_n$.
\smallskip

As an application of the Archimedean Positivstellensatz 
we show that for an Archimedean module $Q(\mathsf{f})$ both relaxations converge to the minimum of $p$.
\begin{theorem}\label{archconvergenthm}
Suppose  the quadratic module $Q(\mathsf{f})$ is Archimedean. Then the set $ \cK(\mathsf{f})$ is compact, so $p$ attains its minimum over  $ \cK(\mathsf{f})$, and we have
\begin{align}\label{limitpmonn}
\lim_{n\to \infty}~ p^{\sos}_n=\lim_{n\to \infty}~ p^{\mom}_n = p^{\min}.
\end{align}
\end{theorem} 
\begin{proof}
We have noted  that $Q(\mathsf{f})$ is Archimedean  implies that $ \cK(\mathsf{f})$ is compact. 

Let $\lambda\in \R$ be such that $\lambda <p^{\min}$. Then $p(x)-\lambda >0$ on $\cK(\mathsf{f})$. Therefore,  $p-\lambda \in Q(\mathsf{f})$ by the Archimedean Positivstellensatz. This means that $p-\lambda =\sum_j f_j\sigma_j$ for some elements\, $\sigma_j\in \sum \R_d[\ux]^2$. We choose $n\in \N$ such that $n\geq \deg(f_j\sigma_j)$ for all $j$. Then   $p-\lambda=\sum_j f_j\sigma_j\in Q(\mathsf{f})_n$ and hence  $p^{\sos}_n \geq \lambda$ by the definition  of $p^{\sos}_n$.  Since $\lambda <p^{\min}$ was arbitrary and  $p^{\sos}_n\leq p^{\mom}_n\leq p^{\min}$ and $p^{\sos}_n\leq p^{\sos}_{n+1}$ by (\ref{psosinq1}) and (\ref{psosinq2}), this implies (\ref{limitpmonn}).
\qed \end{proof}

The following  propositions deal with two special  situations. The first shows that the supremum in (\ref{semidprsos}) is attained if $ \cK(\mathsf{f})$ has interior points. 
\begin{proposition}\label{psos=pmom}
Suppose that $ \cK(\mathsf{f})$ has a nonempty interior. Then $ p^{\sos}_n= p^{\mom}_n$ for   $n\in \N$. If\, $p^{\mom}_n>-\infty$, then the supremum in (\ref{semidprsos}) is a maximum.
\end{proposition}
\begin{proof} [MP, Proposition 16.7]
\qed \end{proof}

 The quadratic module $Q(\mathsf{f})$ is called \emph{stable} if  for  $n\in \N_0$ there  is a  $k(n)\in \N_0$ such that each $q\in  Q(\mathsf{f})_n$ can be written  as $q=\sum_j f_j\sigma_j$ with $\sigma_j\in \sum \R_d[\ux]^2$,  $\deg(f_j\sigma_j)\leq k(n)$.
 (This is a restriction on degree cancellations for elements of $Q(\mathsf{f})$.)

\begin{proposition}\label{stablefiniteconv}
Suppose that the quadratic module $Q(\mathsf{f})$ is stable. Then there exists an  $n_0\in \N_0$, depending only on $\deg(p)$, such that
$p^{\sos}_n=p^{\sos}_{n_0}$ for all $n\geq n_0$.
\end{proposition}
\begin{proof} 
%\cite[Proposition 16.8]{mp}. 
 {[MP, Proposition 16.8]}.
\qed
\end{proof}

It can be shown that if\, $\dim \cK(\mathsf{f})\geq 2$, then the assumptions of  Theorem \ref{archconvergenthm} and Proposition \ref{stablefiniteconv} exclude each other. 

 It is natural to look for conditions which imply finite  convergence for the limit 
$\lim_{n\to \infty}~ p^{\mom}_n = p^{\min}.$ 
The flat extension 
theorem of  Lecture 5
gives such a result: 
\smallskip

\textit{Let $m:=\max \{1, \deg (f_j): j=0,\dots,k\}$ and $n>m$.
If the infimum in (\ref{defpmomn}) is attained at $L$ and\,  $\rank H_{n-m}(L)=\rank H_n(L) $, then\,  $ p^{\mom}_k = p^{\min}$\, for all\,  $ k\geq n$.}
\smallskip

%\end{document}

\makeatletter
\renewcommand{\@chapapp}{Lecture}
\makeatother
%\chapter{}
%\newpage
%\chapter{}

\setcounter{tocdepth}{3}

\chapter{Truncated multidimensional moment problem: existence via positivity}

%\end{document}
Abstract:\\
\textit{The theorem of Richter-Tchakaloff is proved. Existence criteria by positivity conditions are formulated and discussed.
Stochel's theorem is stated.}
\bigskip

The next three Lectures are concerned with the multidimensional truncated moment problem. We consider  the following general setup:
\smallskip

\noindent
$\bullet$~ $\cX$ is a \textbf{locally compact} Hausdorff space,\\
$\bullet$~ $E$ is a \textbf{finite-dimensional vector space of continuous real-valued functions} on $\cX$,\\
$\bullet$~ $K$ is a \textbf{closed} subset of $\cX$.

\smallskip
That $E$ has finite dimension is crucial! Further, we assume the following:
\begin{align}\label{g-condition}
\textbf{ There exists a function $e\in E$ such that $e(x)\geq 1$ for $x\in \cX$.} 
\end{align}

Let us recall the following fundamental notion.

\begin{dfn}
A linear  functional $L:E\to \R$   is called a \textbf{$K$-moment functional} if  there exists a   Radon measure $\mu$  on $\cX$ \textbf{supported on $K$} such that
\begin{align*}
%\label{intreplf}
L(f)=\int_\cX f(x) \, d\mu(x)\quad \text{\textit{for}}~~ f\in E?
\end{align*}
In the case $K=\cX$ we say that  $L$ is   a \textbf{moment functional}. 
\end{dfn}

Our guiding example and most important case is the following:\smallskip

 $\cX=\R^d$ and $E=\R_d[\ux]_m:=\{p\in \R_d[\ux]:\deg (p)\leq m\}$.
 \smallskip
 
\noindent 
 In this case we have the \emph{classical multi-dimensional truncated moment problem.} Assumption (\ref{g-condition})is satisfied with $e(x):=1$.
\smallskip

Further, $\delta_x$ denotes the Delta measure and $l_x$ is the point evaluation functional on $E$ at $x\in \cX$. For $k\in \N$,  a \emph{$k$-atomic measure} is a measure $\mu=\sum_{i=1}^k c_i\delta_{x_i}$, where all $c_i>0$ and the $x_i$ are pairwise different points of $\cX$. We  consider the zero measure as a $0$-atomic measure.

\smallskip
\section{The Richter-Tchakaloff Theorem}

The following theorem is the most important general result on truncated moment problems. It implies that  $K$-moment functionals on $E$ have always finitely atomic representing measures and are finite linear combinations of point evaluations. Thus, convex analysis comes up as a useful technical tool  to study  moment problems. 

The theorem was proved in full generality by Hans Richter (1957). It is often called  Tchakaloff theorem in the literature and \textbf{Richter--Tchakaloff theorem} in  [MP]. From its publication data and its generality it seems to be fully justified to call it \emph{Richter theorem}. We reproduce the proof given in [MP].

\begin{theorem}\label{richter}
Suppose   $(\cY,\mu)$ is a measurable space and $V$ is a finite-dimensional linear   space of $\mu$-integrable measurable real-valued functions on $(\cY,\mu)$. Let  $L^\mu$ denote the  functional on $V$ defined by $L^\mu(f)=\int f\, d\mu$, $f\in V$. Then there is a $k$-atomic measure $\nu=\sum_{j=1}^k m_j\delta_{x_j}$ on $\cY$,  $k\leq \dim V $, such that $L^\mu=L^\nu$, that is,\begin{align*} \dint f d\mu=\dint f d\nu\equiv \sum_{j=1}^k m_jf(x_j), \quad f\in V.\end{align*}
\end{theorem}
\begin{proof}
Let $C$ be the convex cone in the dual space $V^*$ of all nonnegative linear combinations of point evaluations $l_x$, where $ x\in\cY,$ and let $\ov{C}$ be the closure of $C$ in  $V^*$. We prove by induction on $m:=\dim \,V$\, that\, $L^\mu \in C.$

 First let $m=1$ and $V=\R{\cdot}f$. Set $c:=\int f d\mu$. If $c=0$, then $\int (\lambda f) d\mu =0\cdot l_{x_1}(\lambda f),$  $\lambda \in \R$, for any $x_1\in \cY$. Suppose now that $c>0$. Then $f(x_1)>0$ for some $x_1\in \cY$. Hence $m_1:=cf(x_1)^{-1}>0$ and   $\int( \lambda f) d\mu =m_1 l_{x_1}(\lambda f)$
  for $\lambda \in \R$. The case $c<0$ is treated similarly.

Assume  that the assertion holds for vector spaces of dimension $m{-}1.$ Let $V$ be a vector space of dimension $m$. By standard approximation of  $\int f d\mu$ by integrals of simple functions it follows that $L^\mu \in \ov{C}$. We now distinguish between two cases.
\smallskip

Case 1: $L^\mu$ is an interior point of $\ov{C}$.\\
By a basic result of convex analysis, the  convex set  $C$ and its closure $\ov{C}$ have the same  interior points. Hence we have $L^\mu\in C$ in Case 1.
%(by Proposition \ref{convexprope1}(ii)),  we have $L^\mu\in C$.
%\smallskip

Case 2: $L^\mu$ is a boundary point of $\ov{C}$.\\ Then there exists a supporting hyperplane $F_0$ for the cone\, $\ov{C}$ at $L^\mu$
%(by Proposition \ref{convexprope2}(ii))
, that is, $F_0$ is a linear functional on $V^*$ such that $F_0\neq 0$, $F_0(L^\mu)=0$ and $F_0(L)\geq 0$ for all $L\in \ov{C}$. Because $V$ is  finite-dimensional,  there is a function $f_0\in V$ such that $F_0(L)=L(f_0),$   $L\in V^*$. For $x\in \cY$, we have $l_x\in C$ and hence  $F_0(l_x)=l_x(f_0)=f_0(x)\geq 0.$ Clearly, $F_0\neq 0$ implies that $f_0\neq 0$.
We  choose an $(m{-}1)$-dimensional linear subspace $V_0$ of $V$ such that $V=V_0\oplus \R{\cdot} f_0$. Let us set $\cZ:=\{x\in \cY:f_0(x)=0\}$. Since\, $0=F_0(L^\mu)=L^\mu(f_0)=\int f_0 d\mu$ and $f_0(x)\geq 0$ on $\cY$, it follows that $f_0(x)=0$ $\mu$-a.e. on $\cY$, that is, $\mu(\cY\backslash \cZ)=0$. Now we define a measure  $\tilde{\mu}$  on $\cZ$  by $\tilde{\mu}(M)=\mu(M\cap \cZ)$. Then 
\begin{align*}
 L^\mu(g)
 = \dint_{\cY} g\, d\mu
 =\dint_{\cZ} g\, d\mu
 =\dint_\cZ g\, d\tilde{\mu}
 =L^{\tilde{\mu}}(g)\quad\text{ for }~~ g\in V_0.
\end{align*} 
We apply the induction hypothesis to the functional $L^{\tilde{\mu}}$ on  $V_0\subseteq \cL^1(\cZ,\tilde{\mu})$. Since $ L^\mu= L^{\tilde{\mu}}$ on $V_0$, there exist  $\lambda_j\geq 0$ and $x_j\in \cZ, j=1,\dots,n$, such that for $f\in V_0$,
\begin{equation}\label{reprLaslincombin0}
L^\mu(f)=\sum_{j=1}^{n} \lambda_j l_{x_j}(f).
\end{equation}
%for $f\in V_0$.
Since $f_0=0$ on $\cZ$, hence $f_0(x_j)=0$, and $L^\mu(f_0)=0$, (\ref{reprLaslincombin0}) holds for $f=f_0$ as well and so for all $f\in V$. Thus, $L^\mu\in C$. This completes the induction proof.

The set $C$ is a cone in the $m$-dimensional real vector space $V^*$. Since $L^\mu\in C$,  it follows from
Carath\'eodory's theorem 
%(Theorem \ref{caratheothm}(ii)) implies 
that there is a representation (\ref{reprLaslincombin0})  with $n\leq m $. This means that $L^\mu$ is the integral of the measure $\nu=\sum_{j=1}^n \lambda_j \delta_{x_j}$. Clearly, $\nu$ is $k$-atomic, where $k\leq n\leq m$. (We only have  $k\leq n$, since some numbers $\lambda_j$ in (\ref{reprLaslincombin0}) could be  zero and the points $x_j$ are not  necessarily different.)
$\hfill \qed$ \end{proof}
 
 Recall that Carath\'eodory's theorem (see, e.g., [MP, Proposition~A.35]) says the following:
 If $X$ is a subset of a $d$-dimensional real vector space $V$, then  each element of the cone generated by $X$ is a nonnegative  combination 
of $d$ points of $X$.
\smallskip

 For the truncated moment problem this theorem has the following corollary.
 \begin{corollary}\label{carrichter} Each moment functional $L$ on  $E$  has a $k$-atomic representing measure $\nu$, where $k\leq \dim\, E$. If\, $\mu$ is a representing measure of $L$ and $\cY$ is a Borel subset of $\cX$ such that $\mu(\cX\backslash \cY)=0$, then all atoms of $\nu$ can be chosen from $\cY$.
\end{corollary}
\begin{proof}
Apply Theorem \ref{richter} to the measure space $(\cY,\mu \lceil \cY)$ and $V=E$.
$\hfill\qed$\end{proof}

\section{Positive semidefinite $2n$-sequences}

If the space   $\cX$ is \emph{compact},  each $E_+$-positive linear functional on $E$ is a moment functional by Proposition 1.7 of Lecture 1. %\ref{integralrepcompactcase}.
If $\cX$ is  not compact, this is no longer true. In this section, we  discuss this  for the truncated Hamburger moment problem of $E=\R[x]_{2n}$, $\cX=\R$.

%The following theorem characterizes the positive semidefiniteness of a $2n$-sequence in terms of an integral representation (\ref{mpcondtion}).
\begin{theorem}\label{hmppositivesemi}
For a real sequence $s=(s_j)_{j=0}^{2n}$ the following  are equivalent:
\begin{itemize}
\item[\em (i)]~  $s$ is positive semidefinite, that is,
\begin{align*}
\sum_{k,l=0}^n \, s_{k+l}c_kc_l \geq 0\quad\text{for~ all}~~  (c_0,\dots,c_n)^T\in \R^{n+1}, n\in \N.
\end{align*}
\item[\em (ii)]~  The Hankel matrix $H_n(s)=(s_{i+j})_{i,j=1}^n$ is positive semidefinite.
\item[\em (iii)]~  $L_s(p^2)\geq 0$ for all $p\in \R[x]_n$.
\item[\em (iv)]~ $L_s(q)\geq 0$ for all $q\in \R[x]_{2n}$ such that $q(x)\geq 0$ on $\R$.
\item[\em (v)]~  There exists a   Radon measure $\mu$ on $\R$ and a  number\, $a \geq 0$ such that  
\begin{align}\label{mpcondtion}
s_j=\int_\R x^j\, d\mu(x) ~~ \text{for}~~ j=0,\dots,2n-1,~~\text{and}~~
s_{2n} =a+\int_\R x^{2n}\, d\mu(x).
\end{align}
\end{itemize}
\end{theorem}
\begin{proof} % The crucial part of the  equivalence of (i), (ii), and (iii) is easily verified.
(iii)$\to$(v): The proof is  based on Proposition 1.7 from Lecture 1.
%\ref{integralrepcompactcase}.  
Let $\cX:=\R\cup \{\infty\}$ denote the one point compactification of $\R$. The functions 
\begin{align*}u_j(x):=
x^j(1+x^{2n})^{-1},\quad  j=0,\dots,2n,\end{align*}
are continuous functions on $\cX$ by  $u_j(\infty):=0,  j=0,\dots,2n-1$, and $u_{2n}(\infty):=1$. Now we define a linear functional $\tilde{L}$ on 
\begin{align*}
 F:=\Lin\{u_j:j=0,\dots,2n\}\quad\text{ by }~~~ \tilde{L}(u_j)=L_s(x^j), j=0,\dots,2n.
\end{align*}
Using that positive polynomials of $\R[x]$ are sums of squares one easily verifies that $\tilde{L}$ is  $F_+$-positive.  Hence, by Proposition 1.7, %\ref{integralrepcompactcase}, 
$\tilde{L}$  has a representing Radon measure $\tilde{\mu}$ on $\cX$.

 Set $a=\tilde{\mu}(\{\infty\})$ and define   measures $\hat{\mu}$ and $\mu$ on $\R$ by $\hat{\mu}(M):=\tilde{\mu}(M)$ for $M\subseteq \R$ and $d\mu:=(1+x^{2n})^{-1}d\hat{\mu}$. Then, for $j=0,\dots,2n$, 
 \begin{align*}s_j=L_s(x^j)=\tilde{L}(u_j)=\int_\cX u_j(x)\, d\tilde{\mu}=a u_j(\infty)+\int_\R \frac{x^j}{1+x^{2n}}\, d\hat{\mu}= a\delta_{j,2n}+ \int_\R x^j\, d\mu, \end{align*}

(v)$\to$(ii): Let $(c_0,\dots,c_n)^T\in \R^{n+1}$. Using (\ref{mpcondtion}) we derive
\begin{align}\label{iiiimpliesi}
\sum_{k,l=0}^n s_{k+l} c_kc_l =ac_n^2+ \sum_{k,l=0}^n \int_\R c_kc_lx^{k+l} d\mu =ac_n^2+\int_\R \bigg(\sum_{k=0}^n c_kx^k\bigg)^2 d\mu\geq 0,
\end{align}
since $a\geq 0$. This proves (ii).

The other implications are trivial or easily checked.
$\hfill\qed$ \end{proof}   

\begin{example}\label{semidefnotmp}
$s=(0,0,1)$ is positive semidefinite. Since $s_0=0$, $s$ cannot be given by a positive measure on $\R$, but $s$ has a representation (\ref{mpcondtion}) with $\mu=0$,  $a=1$. $\hfill \circ$
\end{example}

This simple example shows that the $E_+$-positivity of  a linear functional or the positive  semidefiniteness of a sequence $s$ are 
 \emph{not sufficient} for representing it by a positive measure. Roughly speaking, "there may be atoms at infinity". 
 
 But if there is a polynomial $p\in \R[x]$ of degree $n$ such that $L_s(p^2)=0$,  it follows from (\ref{iiiimpliesi}) that $a=0$, hence $L_s$ is indeed a moment functional. The notion of the \emph{Hankel rank} can be used  to characterize truncated Hamburger moment functionals, see [MP, Section~9.5].

\section{The truncated moment problem on   projective space}\label{tmpprojectivespace}

We  study  truncated moment problems on the $d$-dimensional real projective space   $\dP^d(\R)$  and apply this to the truncated $K$-moment problem   
for \emph{closed} sets  in $\R^d$.

The points of  $\dP^d(\R)$ are equivalence classes of $(d+1)$-tuples $(t_0,\dots,t_d)\neq 0$ of real numbers under the equivalence relation 
\begin{align}\label{equivalencerel}
(t_0,\dots,t_d)\sim(t^\prime_0,\dots,t^\prime_d)\quad\text{ if }\quad  (t_0,\dots,t_d)=\lambda(t^\prime_0,\dots,^\prime_d)\quad\text{ for }~~\lambda\neq 0.
\end{align} 
 The equivalence class is  $[t_0:\dots:t_d]$.
Thus, $\dP^d(\R)=(\R^{d+1}\backslash \{0\})/ \sim.$  The map 
\[\varphi:\R^d\ni(t_1,\dots,t_d)\mapsto [1:t_1:\dots: t_d]\in\dP^d(\R)\]
 is  injective. 
 We  identify  $t\in\R^d$  with  $\varphi(t)\in\dP^d(\R)$. Then  $\R^d\subseteq\dP^d(\R)$. The complement of $\R^d$ in $\dP^d(\R)$ is the hyperplane $\dH_\infty^d =\{[0:t_1:\dots: t_d]\in \dP^d(\R)\}$.

We denote by  $\cH_{d+1,2n}$  the  \emph{homogeneous}  polynomials of $\R[x_0,x_1,\dots,x_d]$\, of degree $2n$. 
%Recall that $\R_d[\ux]_{2n}$ are the polynomials from $\R[x_1,\dots,x_d]$ of degree at most $2n$.  
The map
 \begin{align*}
\phi:p(x_0,\dots,x_d)\mapsto \hat{p}(x_1,\dots,x_d):=p(1,x_1,\dots,x_d)
\end{align*}
is a  bijection of the vector spaces\, $\cH_{d+1,2n}$\, and\, $\R_{d}[\ux]_{2n}$.

The projective space $\dP^d(\R)$ is a compact space and  
%$\R^d$ is dense in $\dP^d(\R)$. 
 each $q\in \cH_{d+1,2n}$ can be  considered as a continuous function $\tilde{q}$ on this space by
\begin{align}\label{actionprojective}
\widetilde{q}(t):=\frac{q(t_0,\dots,t_d)}{(t_0^{2}+\dots+t_d^{2})^n}, ~~t=[t_0:\dots:t_d]\in \dP^d(\R).
\end{align}
We set $E:=\{ \widetilde{q}:q\in \cH_{d+1,2n}\}$. For $e(x):=(x_0^{2}+\dots+x_d^{2})^n$,\, we have $\tilde{e}\in E$ and   $\tilde{e}(t)=1$ for  $t\in \dP^d(\R)$, so assumption (\ref{g-condition}) is fulfilled. Since  $\dP^d(\R)$ is compact,  each $E_+$positive linear functional on $E$ is a moment functional. Then the representing measure is a Radon measure on $\dP^d(\R)$. 
\smallskip

Let us return to the  $K$-moment problem for  $\R_d[\ux]_{2n}$ for a \emph{closed} set $K\subseteq\R^d$.

The closure  $K^-$  of $K$ in $\dP^d(\R)$ is the disjoint union of $K$ and   $K_\infty:=K^- \cap\dH_\infty^d$. We want to characterize linear functionals on $\R_d[\ux]_{2n}$ which are nonnegative on
\begin{align}\label{PoscK}
 {\Pos}(K)_{2n}
 :=\{p \in \R_d[\ux]_{2n}: p(x)\geq 0\quad\text{ for }\quad x\in K\}.
\end{align}

For $p\in \R_d[\ux]_{2n}$ let $p_{2n}$ denote its homogeneous part of degree $2n$ and put
\begin{align*}
 \widetilde{p_{2n}}(t)
 := \frac{p_{2n}(t_1,\cdots,t_d)}{(t_1^{2}+\dots+t_d^{2})^n}\quad\text{ for }\quad t=(0:t_1:\dots:t_d)\in \dH_\infty^d\cong \dP^{d-1}(\R).
\end{align*}

%Note that the fraction   is a well-defined continuous function on  $\dH_\infty^d.$
The  new ingredient is that measures  on  $K_\infty$ give ${\Pos}(K)_{2n}$-positive  functionals.
\begin{lemma}\label{linftypos}
If $\mu_\infty$ is a  Radon measure on $\dH_\infty^d$ supported on $K_\infty$, then 
\begin{align}\label{Linfty}
L_\infty(p):=\int_{K_\infty} \widetilde{p_{2n}}(t)\, d\mu_\infty(t), \quad p\in \R_d[\ux]_{2n},
\end{align}
defines a linear functional on $\R_d[\ux]_{2n}$ such that $L_\infty(p)\geq 0$ for $p\in {\Pos}(K)_{2n}.$
\end{lemma}
\smallskip

 The next theorem characterizes   ${\Pos}(K)_{2n}$-positive linear functionals on $\R_d[\ux]_{2n}$.
\begin{theorem}\label{posKchar}
Suppose $K\subseteq \R^d$ is  closed  and  $L$ is a linear functional on $\R_d[\ux]_{2n}$. Then  we have
\begin{align}\label{kpositivetiy}
 L(p)\geq 0 \quad\text{ for }~~~~ p\in {\Pos}(K)_{2n}
\end{align}
 if and only if there are  Radon measures $\mu$ on $\R^d$  supported on $K$ and $\mu_\infty$ on $\dH_\infty^d$  supported on $K_\infty$ such that
\begin{align}\label{L+Linfty}
L(p)=\int_K p(t)\, d\mu(t) +\int_{K_\infty} \widetilde{p_{2n}}(t)\, d\mu_\infty(t) \quad \text{for}\quad p\in \R_d[\ux]_{2n}.
\end{align}
\end{theorem}
\begin{proof}
% \cite[Theorem 17.3]{mp} 
 {[MP, Theorem~17.3]}.
 $\hfill \qed$
\end{proof}

 If the set $K$ in Theorem \ref{posKchar} is  compact, then $K_\infty$ is empty, so  the second summand in (\ref{L+Linfty}) does not occur and  we recover the known result that for compact sets $K$ the positivity  condition  characterizes $K$-moment functionals.

 The   moment problem on  closed subsets $K$ of $\dP^d(\R)$ has several advantages.
 First, $K$ is compact, so  $K$-moment functionals are characterized by  positivity conditions.
 Secondly, homogeneous polynomials are more convenient to deal with and additional technical tools such as the apolar scalar product (see [MP, Section~9.1]) are available.

%\section{Moment functionals via extensions}
\smallskip

As an easy application of Theorem \ref{posKchar} we show that $K$-moment functionals for \emph{closed} subsets $K$ of $\R^d$  can be characterized be the following extension property.
\begin{theorem}\label{extrsnion2n}
Let $K$ be a closed subset of $\R^d$ and  $L_0$ a linear functional on $\R_d[\ux]_{2n-2}$. Then $L_0$ is a  $K$-moment functional on $\R_d[\ux]_{2n-2}$ if and only if $L_0$ admits an extension  to a ${\Pos}(K)_{2n}$-positive linear functional $L$ on $\R_d[\ux]_{2n}$:
% {kpositivetiy}
\begin{align}\label{extendedposit}
L(p)\geq 0\quad \text{for}\quad p\in {\Pos}(K)_{2n}.
\end{align}
\end{theorem}
\begin{proof}
Let $L_0$ be
 a  $K$-moment functional on $\R_d[\ux]_{2n-2}$. By Richter's theorem,  $L_0$ has a finitely atomic representing measure $\mu$. Then it suffices to define $L$ by $L(p)=\int f\, d\mu$, $f\in\R_d[\ux]_{2n}$. 

Conversely, assume  (\ref{extendedposit}) holds. By Theorem \ref{posKchar}, $L$ is of the form  (\ref{L+Linfty}). If $p\in\R_d[\ux]_{2n-2}$, then  $p_{2n}=0$. Hence the second summand in (\ref{L+Linfty}) vanishes and we obtain\, $L_0(p)=\int_K \, p(t)\, d\mu(t)$\, for $p\in\R_d[\ux]_{2n-2}$.
$\hfill\qed$ \end{proof}
\section{Stochel's theorem}
The following interesting result is \textbf{Stochel's theorem} (2001). It says that 
if for a functional $L$ on $\R_d[\ux]$ the truncated $K$-moment problem  on $\R_d[\ux]_{2n}$ is solvable for all $n\in \N$, then $L$ is a $K$-moment functional on $\R_d[\ux]$. 
\begin{theorem}\label{stochel}
Let $K$ be a closed subset of $\R^d$ and $L$  a linear functional on $\R_d[\ux]$. Suppose  for each $n\in \N$  the restriction $L_n:=L\lceil\R_d[\ux]_{2n}$ is a  $K$-moment functional on $\R_d[\ux]_{2n}$, that is, there exists a Radon measure $\mu_n$ on $\R^d$ supported on $K$ such that
\[
L_n(p)\equiv L(p)=\int p(x)\, d\mu_n(x) \quad \text{for}~~ p \in\R_d[\ux]_{2n}.
\]
Then $L$ is a $K$-moment functional on $\R_d[\ux]$. 
%Further, $L$ has a representing measure $\mu$ which is the limit of  a subsequence $(\mu_{n_k})_{k\in \N}$ in the vague convergence of measures on $K$. 
\end{theorem}
\begin{proof}  
%\cite[Theorem 17.3]{mp} 
 {[MP, Theorem~17.3]}.
%Apply Theorem \ref{mlcountablecase} to $\cX=K$, $E={\sA}\lceil K$, $E_n=\R_d[\ux]_{2n}$.
$\hfill\qed$
\end{proof}

In terms of sequences  Theorem \ref{stochel} can be rephrased as follows: 
 Let $s=(s_\alpha)_{\alpha\in \N_0^d}$ be a real multisequence. If for each\, $n\in \N$ the truncation\,  $s^{(n)}:=(s_\alpha)_{\alpha\in \N_0^d, |\alpha|\leq 2n}$\, has a representing measure supported on $K$, then the  sequence $s$ does as well.

%\end{document}

\makeatletter
\renewcommand{\@chapapp}{Lecture}
\makeatother
%\chapter{}
%\newpage
%\chapter{}
%\newpage
%\chapter{}
%\newpage
%\setcounter{page}{1}

\setcounter{tocdepth}{3}

\chapter{Truncated multidimensional moment problem: existence via flat extensions}

Abstract:\\
\textit{Hankel matrices are introduced and their basic properties are investigated. 
The flat extension theorem of Curto and Fialkow is  treated.}
\bigskip

Hankel matrices appeared already in Lectures 1 and 3. They are also an important technical tool for the study of multidimensional  truncated moment problems. The flat extension developed in Section \ref{flatexnetiontheorem} is essentially based on rank conditions of Hankel matrices.

\section{Hankel matrices}\label{hankelmatrices}

 Throughout this Lecture, we suppose that $\sN$ be a nonempty finite subset of $\N_0^d$ and ${\cA}=\Lin\{ x^\alpha:\alpha \in {\sN}\}$.
 We shall  consider the moment problem on the (finite-dimensional) linear subspace $\cA^2$ (not of $\cA$!) of the polynomials $\R[x_1,\dots,x_d]$:
%Suppose $L$ is a  linear functional on the  linear subspace 
\begin{align*}
{\cA}^2:={\Lin}\, \{ p\, q:\, p,q\in {\cA}\}={\Lin}\, \{x^\beta: \beta\in {\sN}+{\sN}\}.
\end{align*} 
Suppose $L$ is a linear functional on the real vector space $\cA^2$.
%of\, $\R_d[\ux]$. Note that in contrast to Sections \ref{formulation}--\ref{tmpprojectivespace} we now consider  functionals  on ${\cA}^2$ rather %than ${\cA}.$

%The following definition contains  the basic notions studied in this section.

\begin{definition}\label{hankeldef} $p\, q\in \R_d[\ux]_{n-1}$
The \emph{Hankel matrix} of  $L$ is the symmetric $|\sN|\times|\sN|$-matrix\, 
\begin{align*}
 H(L)
 =(h_{\alpha,\beta})_{\alpha,\beta\in {\sN}} ,\quad\text{ where }~~ h_{\alpha,\beta}:=L(x^{\alpha+\beta}), ~~ \alpha,\beta\in {\sN}.
\end{align*}
 The  number\, $\rank L:=\rank H(L)$\, is called the \emph{rank} of  $L$.
\end{definition} 

We fix an ordering of the index set $\sN$.
For $f=\sum_{\alpha\in {\sN}} f_\alpha x^\alpha\in {\cA}$ we denote by $\vec{f}=(f_\alpha)^T\in \R^{|{\sN}|}$  the coefficient vector of $f$ according  to the ordering of  $\sN$.
 Define 
\begin{align}\label{defkernle}
 \cN_L&= \{f\in {\cA}: L(fg)=0~~\text{ for all }~~g\in {\cA}\},\\
 \cV_L&=\{t\in \R^d: f(t)=0~~\text{ for all }~ f\in \cN_L\}.\label{v(L)variety}
\end{align}
Note that $\cV_L$ is a real algebraic subset of $\R^d$.

 The next proposition shows that properties of the functional $L$ can be  nicely translated into those for the matrix $H(L)$.

\begin{proposition}\label{propHankelL}\begin{itemize}
\item[\em (i)] For $f\in \cA$ and $g\in \cA$ we  have
\begin{align}\label{lhankelfg}
L(fg)=\vec{f}^{\, T} H(L) \vec{g}.
\end{align}
\item[\em (ii)]~\,$L(f^2)\geq 0$\,  for  $f\in\cA$\,   if and only if the matrix $H(L)$ is positive semi-definite.
\item[\em (iii)]~ ${\rank} \, L=\dim\, ({\cA} /  \cN_L)$.
\item[\em (iv)]~ $f\in \cN_L$ if and only if  $\vec{f}\in\ker H(L).$ 
 \item[\em (v)]~ If $L$ is a moment functional on $\cA^2$, then\, ${\supp}\, \mu \subseteq \cV_L$  for each representing measure $\mu$. 
\end{itemize}
\end{proposition}
\begin{proof}
(i): Let $f=\sum_{\alpha\in {\sN}} f_\alpha x^\alpha\in {\cA}$ and $g=\sum_{\alpha\in {\sN}} g_\alpha x^\alpha\in {\cA}$. Then
\begin{align*}
L(fg)=\sum_{\alpha,\beta\in {\sN}} f_\alpha g_\beta L(x^{\alpha+\beta})=\sum_{\alpha,\beta\in {\sN}} h_{\alpha,\beta} f_\alpha g_\beta= 
(f_\alpha)^{\, T} H(L) (g_\beta)=\vec{f}^{\, T} H(L) \vec{g}.
\end{align*}

(ii) and (iv) follow at once from (i).
%iii): By the rank-nullity characterization,
\begin{align*}
~~~~\text{(iii):} ~~\rank H(L)
 = \dim \R^{|{\sN}|} - \dim\, \ker H(L)
 =\dim {\cA}-\dim  \cN_L
 =\dim ({\cA}/\cN_L).
\end{align*}

(v): If $p\in \cN_L$, then $L(p^2)=\int p(t)^2\, d\mu =0$. Since $p(x)$ is continuous on $\R^d$, this implies  $\supp\, \mu$ is contained in the zero set of $p$. Thus, $\supp \,\mu \subseteq \cV_L$.
$\hfill \qed$ \end{proof}

 Now  suppose that the functional $L$ is \emph{positive}, that is, $L(f^2)\geq 0$ for  $f\in {\cA}$.
 Then the symmetric bilinear form $\langle \, \cdot,\cdot\rangle$ on the quotient  space $\cD_L:={\cA}/\cN_L$ defined by
\begin{align}\label{defscalar}
\langle f+\cN_L,g+\cN_L \rangle:=L(fg),\quad f,g\in {\cA}.
\end{align}
 is nondegenerate and positive definite, so it is  an inner product.
 By Proposition \ref{propHankelL}(iii),  $\dim\, \cD_L =\rank\, L$.
 Thus, $(\cD_L,\langle \cdot, \cdot\rangle)$ is a \emph{real finite-dimensional Hilbert space}.
 This Hilbert space is another useful tool in treating the moment problem. 
Since  element of $ {\cA}^2$ are sums of products $fg$  with 
$f,g\in{\cA}$, the  functional $L$ on $\cA^2$ can be  recovered from the   space $(\cD_L,\langle \cdot, \cdot\rangle)$ by using equation
 (\ref{defscalar}).

The next proposition shows that $\cN_L$ obeys some ideal-like properties.
\begin{proposition}\label{idealLN}
Let $p\in \cN_L$ and $q\in {\cA}$.
\begin{itemize}
\item[\em (i)]\, If $L$ is a positive functional and $pq^2\in {\cA}$, then $pq\in \cN_L$.
\item[\em (ii)]\, If $L$ is a  $K$-moment functional and $pq\in \cA$, then $pq\in \cN_L$.
\end{itemize}
\end{proposition}
\begin{proof}
(i): Since the functional $L$ is positive, the Cauchy--Schwarz inequality  holds. Using that $pq^2\in {\cA}$ and $L(p^2)=0$ we obtain
\begin{align*}
L((pq)^2)^2=L(p\, pq^2)^2\leq L(p^2)L((pq^2)^2)=0,
\end{align*}
so that $pq\in \cN_L$ by Proposition \ref{propHankelL}(iv).

(ii) Since $L$ is a $K$-moment functional, it has a representing measure $\mu$, which  is supported on $K$. 
 Then\, $\supp\mu \subseteq \cV_L$ by Proposition \ref{propHankelL}(v). Therefore, since $p\in \cN_L$ and hence $p(x)=0$ on $\cV_L$, we get
\begin{align*}
L((pq)^2)=\int_K (pq)^2(x)\, d\mu(x) =\int_{\cV_L\cap K} p(x)^2q(x)^2\, d\mu(x)= 0.
\end{align*} 
Thus,   $p\, q\in \cN_L$  again by Proposition \ref{propHankelL}(iv).
$\hfill\qed$ \end{proof}

We restate the preceding results in the case of our standard example.
\begin{corollary}\label{pqidealpro}
%Let  $\sN$ be the set\, $\N_{0,2n}^d$\, (see (\ref{defN_om})) 
Let $L$ be a linear functional on ${\cA}=\R_{d}[\ux]_{2n}$. Suppose that $p\in \cN_L$ and $q\in \R_d[\ux]_n$.
\begin{itemize}
\item[\em (i)]\,  If $L$ is a positive functional and  $p\, q\in \R_d[\ux]_{n-1}$, then $p\, q\in \cN_L$.
\item[\em (ii)]~ If $L$ is a $K$-moment functional and  $p\, q\in \R_d[\ux]_{n}$, then $p\, q\in \cN_L$.
\end{itemize}
\end{corollary}
\begin{proof} 
(i): Proposition \ref{idealLN}(i) yields the assertion  in the  case $q=x_j,$  $j=1,\dots,d$. Since $\cN_L$ is a vector space,  repeated applications give the general case.

(ii) follows  from Proposition \ref{idealLN}(ii).
$\hfill\qed$ \end{proof}

%Assertions (i) and (ii) of 
Corollary \ref{pqidealpro} show an important difference between positive functionals and moment functionals: The assertion $pq\in \cN_L$ holds for a moment functional if $p\, q\in \R_d[\ux]_n$, while for a positive functional it is assumed that   $p\, q\in \R_d[\ux]_{n-1}$.

\section{The full moment problem with finite rank Hankel matrix}\label{fullmpfinitehankel}

 Definition \ref{hankeldef} of the Hankel matrix $H(L)$ extends verbatim to linear functionals $L$ on the polynomial algebra $\R_d[\ux]$.
 In this case we have the \emph{full} moment problem and in the.
 If this Hankel matrix $H(L)$ has finite rank, there is the following result.

\begin{theorem}\label{fullmpfiniterank}
Suppose that $L$ is a positive linear functional on $\R_d[\ux]$ such that\, ${\rank}\, H(L)=r$, where $r\in \N$.  Then $L$ is a moment functional with unique representing measure $\mu$. This measure $\mu$ has $r$ atoms and\,  
${\supp}\, \mu=\cV_L$.
\end{theorem}
\begin{proof} 
%\cite[Theorem 17.29]{mp}. 
 {[MP, Theorem~17.29]}.
$\hfill \qed$
\end{proof}

 Idea of proof: The GNS representation $\pi_L$ acts on a finite-dimensional Hilbert space. Then $\pi_L(x_1),\dots,\pi_L(x_d)$ are commuting self-adjoint operators and their spectral measure leads to the measure $\mu$.

\section{Flat extensions and the flat extension theorem}\label{flatexnetiontheorem}

Recall that $\sN$ is a finite subset of $\N_0^d$ and  ${\cA}=\Lin \{x^\alpha: \alpha \in {\sN}\}$. Let $\sN_0$ be a proper subset of $\sN$ and $\cB=\Lin \{x^\alpha: \alpha \in {\sN}_0\}$. %For  a linear functional $L$  on $\cA^2$ we denote  by $L_0$ its restriction to $\sB^2$. 

\begin{definition}\label{defflat}
A   linear functional $L$ on ${\cA}^2$ is called  \emph{flat} with respect to ${\cB}^2$  if
\begin{align*}
{\rank}\, H(L)={\rank}\, H(L_0),
\end{align*} where
$L_0$ denotes the restriction to ${\cB}^2$ of $L$.
\end{definition}
To motivate this definition  we  write the Hankel matrix $H(L)$ as a block matrix
\[ H(L) = \left( \begin{array}{ll}
     H(L_0)& ~X_{12}\\
     ~ X_{21} & ~X_{22}
   \end{array} \right).  \]
Then $H(L)$ is  a flat extension  of the matrix $H(L_0)$ 
if\, ${\rank}\, H(L) ={\rank}\, H(L_0)$. This is the useful notion of flatness for block matrices.

The next proposition shows that for flat functionals the positivity on the smaller space $\cB^2$ implies the positivity on the larger space $\cA^2$.

\begin{proposition}\label{flatpositive}
Let $L$ be a linear functional on ${\cA}^2$ which is flat with respect to ${\cB}^2$. If $L(p^2)\geq 0$ for all $p\in {\cB}$, then $L(q^2)\geq 0$ for all $q\in {\cA}$.
\end{proposition}
\begin{proof}
%\cite[Proposition 17.34]{mp}.
 {[MP, Proposition~17.34]}.
$\hfill \qed$
\end{proof}

 Though there are  general versions of the flat extension theorem (see \cite{mp}),  we   restrict ourselves to the  important case where $\cB=\R_d[\ux]_{n-1}$ and $ {\cA}=\R_d[\ux]_{n}$.
%\begin{align*}
%{\sB}^2=\R_d[\ux]_{2n-2},\quad  {\cA}^2=\R_d[\ux]_{2n}.
%\end{align*}

 Let $L$ be a linear functional  on $\R_d[\ux]_{2n}$.
 The Hankel matrix $H_n(L)$  has the entries
\[
 h_{\alpha,\beta}
 :=L(x^{\alpha+\beta}),\quad\text{ where }~~ \alpha,\beta\in \N_0^d, |\alpha|, |\beta|\leq n.
\]
The Hankel matrix $H_{n-1}(L)$ of its restriction to $\R_d[\ux]_{2n-2}$ has the entries $h_{\alpha,\beta}$ with $ |\alpha|,|\beta|\leq n-1.$ 

 The main result of this Lecture is  the following \textbf{flat extension theorem} of R. Curto and L. Fialkow (1996).
 It says that if a functional on  $\R_d[\ux]_{2n} $ is flat with respect  to $\R_d[\ux]_{2n-2}$ and positive on 
$\R_d[\ux]_{2n-2}$, then it is a moment functional.
\begin{theorem}\label{flatmpr2nd}
Suppose  $L$ is a linear functional on $\R_d[\ux]_{2n}, n\in \N$, such that 
\[
 L(p^2)\geq 0~~\text{ for }~~p\in \R_d[\ux]_{n-1} ~~~\text{ and }~~~r:=\rank H_n(L)=\rank H_{n-1}(L).
\]
 Then $L$ is a moment functional with a unique representing measure. This measure is $r$-atomic and we have $r\leq\binom{n-1+d}{d}=\dim \R_d[\ux]_{n-1}$.
\end{theorem}
\begin{proof} 
%\cite[Theorem 17.37]{mp}. 
 {[MP, Theorem~17.37]}.
$\hfill \qed$
\end{proof}

 Idea of proof: The flatness  is used to extend $L$ to a  functional $\tilde{L}$ on $\R_d[\ux]$ which is flat with respect to $\R_d[\ux]_{n-1}$. This is the crucial part of the proof. By Proposition \ref{flatpositive}, $\tilde{L}$ is positive. Since $\rank\, \tilde{L}$ is finite, $\tilde{L}$ is a moment functional by Theorem \ref{fullmpfiniterank}. 

\smallskip

 The next  theorem  deals with the \emph{$\cK$-moment problem}, where 
\begin{align}\label{K_f}
 K
 :=K(\mathsf{f})=\{x\in \R^d: f_1(x)\geq 0,\dots,f_k(x)\geq 0\}.
\end{align}
%Set $m:=\max \{1, \deg (f_j): j=1,\dots,k\}.$

\begin{theorem}\label{flatsemialgebracset}
Let $L$ be linear functional on $\R_d[\ux]_{2n}$ and $r:={\rank}\, H_n(L)$. Then the following are equivalent:
\begin{itemize}
\item[\em (i)]~
 $L$ is a   $\cK$-moment functional which has an $r$-atomic representing measure  with all atoms in $\cK$.
\item[\em (ii)]~  $L$ extends to a linear functional $\tilde{L}$ on $\R_d[\ux]_{2(n+m)}$ such that 
\[
 \rank H_{n+m}(\tilde{L})
 =\rank H_n(L) ~~~\text{ and }~~~\tilde{L}(f_jp^2)\geq 0
\]
 for  $p\in \R_d[\ux]_n$ and $j=0,\dots,k, f_0=1$,  $m:=\max \{1, \deg (f_j): j=1,\dots,k\}.$
 \end{itemize}
 \end{theorem}
\begin{proof}
%\cite[Theorem 17.38]{mp}.
 {[MP, Theorem~17.38]}.
$\hfill \qed$
\end{proof}

 This result says that \emph{all} $\cK$-moment functionals on $\R_d[\ux]_{2n}$  can be obtained by flat extensions to some appropriate larger space $\R_d[\ux]_{2(n+m)}$.

\section{Hankel matrices of functionals with finitely atomic measures}

By the Richter-Tchakaloff theorem \ref{richter}, each moment functional on $\cA^2$ has a finitely atomic representing measure. Therefore, 
 the corresponding Hankel matrices are of particular interest.

In this Section, we suppose $\mu$ is a finitely atomic measure 
\begin{align}\label{defimusignes}
\mu=\sum_{j=1}^k c_j\delta_{x_j},\quad\text{ where }~~~ c_j>0,~ x_j\in \R^d~~~\text{ for }~~j=1,\dots,k,
\end{align}
and $L$  is the corresponding moment functional on ${\cA}^2$:
\begin{align}\label{defLsigenesmeas}
L(f)=\int f(x)\, d\mu= \sum_{j=1}^k c_j f(x_j), ~~f\in {\cA}^2.
\end{align}
Clearly, $\mu$ is $k$-atomic if and only if the points $x_j$ are pairwise distinct. 

For $x\in \R^d$, let $\gs_{{\sN}}(x)=\gs(x)$ denote the column vector $(x^{\alpha})_{\alpha\in {\sN}}$. Note that $\gs_{\sN}(x)$  is the moment vector of the delta measure $\delta_x$ for $\cA$, not for ${\cA}^2$!

\begin{proposition}\label{fiitehankelatomic} 
For the Hankel matrix $H(L)$ we have
\begin{align}\label{Hankelkatomicrep}
&H(L)=\sum_{j=1}^k \,c_j \gs_{\sN}(x_j)\, \gs_{\sN}(x_j)^T,\\ &{\rank}\, H(L)\leq k=|{\supp}\, \mu|\leq |\cV_L|,
\label{rankleqk}
\end{align}
${\rank}\, H(L)=k$ if and only if  the vectors\, $\gs_{\sN}({x_1}),\dots, \gs_{\sN}({x_k})$\,  are linearly independent.
\end{proposition}
\begin{proof} For  $x\in \R^d$ the $(\alpha,\beta)$-entry of 
$\gs(x) \gs(x)^T$ is\, $x^{\alpha+\beta}.$ Thus, $\gs(x) \gs(x)^T$  is  the Hankel matrix of the point evaluation $l_x$ on $\cA^2$. Since  $L=\sum_j c_j l_{x_j}$, this gives (\ref{Hankelkatomicrep}).

By (\ref{Hankelkatomicrep}),  $H(L)$ is a sum of  $k$  rank one matrices\, $c_j \gs(x_j)\gs(x_j)^T$, so ${\rank}\, H(L)\leq k$. Since ${\supp}\, \mu \subseteq \cV_L$, it is obvious that $|{\supp}\, \mu|\leq |\cV_L|.$
 
Now we prove the last assertion.

If $f=\sum_{\alpha \in {\sN}} f_\alpha x^\alpha \in {\cA}$ and    $e_\alpha=(\delta_{\alpha,\beta})_{\beta \in {\sN}}$ is the $\alpha$-th basis vector, then  
\begin{align}\label{rangeH(L)}
H(L)\vec{f}=\sum_{\alpha\in {\sN}} f_\alpha H(L)e_\alpha=
\sum_{\alpha\in {\sN}} f_\alpha \sum_{j=1}^k c_j x_j^\alpha \gs(x_j)=\sum_{j=1}^k c_jl_{x_j}(f) \gs(x_j).
\end{align}
 Since $\im H(L)$ is contained in the span of $\gs(x_1),\dots,\gs(x_k)$, it follows  from (\ref{rangeH(L)}) that  $\rank H(L)=\dim\im H(L)$ is\, $k$\, if and only if the point evaluations $l_{x_1},\dots,l_{x_k}$ on ${\cA}$ are linearly independent.
 Clearly,   $\sum_j a_j\gs(x_j)=0$ if and only if $\sum_j a_j l_{x_j}=0$ on ${\cA}$. Hence ${\rank}\, H(L)=k$ if and only if\,  $\gs({x_1}),\dots, \gs({x_k})$\,  are linearly independent. 
$\hfill \qed$ \end{proof}

 \begin{corollary}
For each moment functional $L$ on ${\cA}^2$ we have
\begin{align}\label{rankleqvlcard}
{\rank}\, H(L)\leq |\cV_L|.
\end{align}
 \end{corollary} 
 \begin{proof}
By Theorem  \ref{richter} and  Proposition \ref{propHankelL}(v),\, $L$ has   a finitely atomic representing measure $\mu$ such that ${\supp}\, \mu\subseteq \cV_L.$ Then ${\rank}\, H(L)\leq |\supp \, \mu|$
by (\ref{rankleqk}), which yields  (\ref{rankleqvlcard}). 
$\hfill \qed$ \end{proof}
\begin{example}
Let $d=1, {\sN}=\{0,1\}$, so that ${\cA}=\{a+bx:a,b\in \R\}$. The functionals $l_{-1},l_0,l_1$ are linearly independent on ${\cA}^2$, but they are linearly dependent on $\cA$, since $2l_0=l_{-1}+l_1$ on $\cA$. For $L=l_{-1}+l_0+l_1$ we have ${\rank}\, H(L)=2$.
$\hfill \circ$ 
\end{example} 
 
 The last theorem in this lecture is only a slight variation of Theorem \ref{flatsemialgebracset} in the case $\cK=\R^d$. It says that \emph{all}  moment functionals on $\R_d[\ux]_{2n}$ have extensions to  some 
$\R_d[\ux]_{2n+2k}$ which are flat with respect to $\R_d[\ux]_{2n+2k-2}$.
\begin{theorem} A linear functional $L$ on $\R_d[\ux]_{2n}$ is a  moment functional if and only if  there exist  a number $k\in \N$ and an extension of $L$ to a positive linear functional  $\tilde{L}$ on $\R_d[\ux]_{2n+2k}$ such that 
$ {\rank}~ H_{n+k}(\tilde{L})={\rank}~ H_{n+k-1}(\tilde{L}).$
\end{theorem}
\begin{proof} The  if part follows from the flat extension Theorem \ref{flatmpr2nd}. Conversely, suppose $L$ is moment functional. By Theorem \ref{richter}, is has a finitely atomic representing measure $\mu$. Define $\tilde{L}(f)=\int f d\mu$ on $\R_d[\ux]$. Then ${\rank}\, H_m(\tilde{L})\leq |{\supp}~\mu|$  by (\ref{rankleqk}). Hence   ${\rank}\, H_m(\tilde{L})= {\rank}\,H_{m-1}(\tilde{L})$ for some $m>n$.
\end{proof}

 We summarize some of the preceding results.
 The following  are \emph{necessary} conditions for  a linear functional $L$ on ${\cA}^2$ to be a moment functional:

\noindent  $\bullet$~ the \emph{positivity condition:}
\begin{align}\label{poscondi}
L(f^2)\geq 0\quad\text{ for }~~~ f\in {\cA},
\end{align}
$\bullet$~ the \emph{rank-variety condition:} 
\begin{align}\label{rankconKR^d}
{\rank}\, H(L)\leq |\cV_L|,
\end{align}
$\bullet$~ the \emph{consistency condition}:
\begin{align}\label{consistency1} 
p\in \cN_L ,q\in {\cA}~~\text{ and }~~ pq\in {\cA}\quad\text{ imply }~~~  pq\in \cN_L.
\end{align}
 Conditions (\ref{rankconKR^d}) and (\ref{consistency1}) follow from  (\ref{rankleqvlcard}) and Proposition \ref{idealLN}. 
\smallskip

Problem:
\textit{When are these (and may be, additional) conditions  sufficient for being a moment functional?}
\smallskip

 One result in this direction is the following: Curto and Fialkow (2005) have shown for the polynomials $\R_2[\ux]_{2n} $ in 2 variables that condition (\ref{poscondi}) and (\ref{consistency1}) are sufficient if there is a polynomial  of degree at most $2$ in the kernel $\cN_L$. 

\begin{definition} A moment functional on $\cA^2$ is called \emph{minimal} if ${\rank}\, H( L)=|\cV_L|$.
\end{definition}
 
It is not difficult to show that each minimal moment functional has a unique representing measure and this measure is ${\rank}\, H( L)$-atomic.

\begin{example}\label{vlhllage}  (\textit{An example for which~ $|\cV_L|>\rank H(L)$}\,)\\
 Suppose that $d=2,n\geq 3,$ and ${\cA}=\R[x_1,x_2]_{n}.$ Set
\begin{align*}
p(x_1,x_2)=(x_1-\alpha_1)\cdots(x_1-\alpha_n), ~~~  q(x_1,x_2)=(x_2-\beta_1)\cdots(x_2-\beta_n),
\end{align*}
where  $\alpha_1<\dots<\alpha_n$,\,   $\beta_1<\dots< \beta_n.$
 Then\, $\cZ(p)\cap \cZ(q)=\{ (\alpha_i,\beta_j): i,j=1,\dots,n\}$.

Define an $n^2$-atomic measure $\mu$ such that the atoms are the points of  $\cZ(p)\cap \cZ(q)$   and a  moment functional on ${\cA}^2$\, by $L(f)=\int f\, d\mu$.
Since $L(p^2)=L(q^2)=0$, we have $p, q\in \cN_L$, so that ${\rank}\, H(L)={\dim}\, ({\cA}/\cN_L)\leq\binom{n+2}{2}- 2 $. Then~ $|\cV_L|=n^2$, so
\begin{align}\label{dffrVLHL}
 |\cV_L|- {\rank}\, H(L) \geq n^2-\binom{n+2}{2}+2
 =\binom{n-1}{2}~.
\end{align}
This  shows that the difference\,  $|\cV_L|- {\rank}\, H(L)$\, can be arbitrarily large. $\hfill \circ$
\end{example}

%\end{document}

\makeatletter
\renewcommand{\@chapapp}{Lecture}
\makeatother
%\chapter{}
%\newpage
%\chapter{}

\setcounter{tocdepth}{3}

\chapter{Truncated multidimensional moment problem: core variety and moment cone}

Abstract:\\
\noindent
\textit{The  core variety is defined and basic results about the core variety are obtained.
Results on the structure of  the moment cone are discussed.}
\bigskip

 In this Lecture, we retain the setup stated at the beginning of Lecture 8.
 Recall that $E$ is a \emph{finite-dimensional} vector space of real-valued continuous functions on a locally compact space $\cX$.

\section{Strictly positive linear functionals}
 As noted in  Lecture 8 (see Example \ref{semidefnotmp} therein), $E_+$-positive  functionals are not necessarily moment functionals.
 But \emph{strictly} positive functionals are, as Proposition \ref{existencemfstrict} below shows.

\begin{definition}\label{strictposfunc}
A linear functional $L$ on $E$ is called {\em strictly $E_+$-positive}\index{Functional! strictly $E_+$-positive} if
\begin{align}\label{strictposL}
 L(f)>0\quad\text{ for all }~~~ f\in E_+,~ f\neq 0.
\end{align}
\end{definition}

It is not difficult to verify that $L$ is strictly positive if and only if $L$ is an inner point of the dual wedge $(E_+)^{\smallwedge}$ in $E^*$.

\begin{proposition}\label{existencemfstrict}
Suppose $L$ a strictly $E_+$-positive linear functional on $E$.
Then $L$ is a moment functional. Further,  for each $x\in \cX$, there is a finitely atomic representing measure $\nu$ of $L$ such that\, $\nu(\{x\})>0.$  
\end{proposition}
\begin{proof}
%\cite[Theorem 1.30]{mp}.
 {[MP, Theorem~1.30]}.
$\hfill \qed$
\end{proof}

\section{The core variety}

Suppose $L$ is a moment functional on $E$. A natural question is to describe the set of points of $\cX$ which are atoms of some representing measure of $L$.
The idea to tackle this problem is the following simple fact.
\begin{lemma}\label{pospol}
 Let  $\mu$ be a representing measure for a moment functional $L$. If $f\in E$ satisfies  $f(x)\geq 0$ on $\supp\, \mu$ and  $L(f)=0$, then $\supp\, \mu$ is contained in the zero set $\cZ(f):=\{x\in \cX:f(x)=0\}$ of $f$.
\end{lemma}

\begin{proof} Suppose $x_0\in \cX$ and   $x_0\notin\cZ(f)$. Then $f(x_0)>0$. Since $f$ is continuous, there are an open neighborhood $U$ of $x_0$ and an $\varepsilon >0$ such that $f(x)\geq \varepsilon $ on $U$. Then 
\begin{align*} 0=\int_\cX f(x)\, d\mu\geq \int_{U}\, f(x) \,d\mu \geq \varepsilon \mu(U)\geq 0,\end{align*}
so that  $\mu(U)=0$. Therefore,  $x_0\notin \supp\, \mu$.
\qed \end{proof} 
 This core variety is defined by a repeated application of this 
idea. It was invented by L. Fialkow (2015).

Suppose $L$ is an arbitrary linear functional   on $E$  such that $L\neq 0$. We define inductively cones $\cN_k(L)$, $k\in \N,$ of $E$ and  subsets $\cV_j(L)$, $j\in \N_0,$ of $\cX$ by $\cV_0(L)=\cX$, 
\begin{align}\label{N_KLE}
 \cN_k(L)&:=\{ p\in E: L(p)=0,~~ p(x)\geq 0 ~~\text{ for }~~x\in \cV_{k-1}(L)\, \},\\
\cV_j(L)&:=\{ x\in \cX: p(x)=0~~\text{ for }~p\in \cN_j(L)\}.\label{V_JL}
\end{align}

\begin{definition}\label{corevariety} 
The {\em core variety}  $\cV(L)$ of the  functional $L$, $L\neq 0$, on $E$ is 
\begin{align}\label{corevarVLE}
\cV(L):=\bigcap_{j=0}^\infty \cV_j(L).
\end{align}
\end{definition}

 If $\R=\cX$ and $E$ is a subset $\R_d[\ux]$, then $\cV(L)$ is the zero set of real polynomials, that is, $\cV(L)$ is a \emph{real algebraic set}.

For instance, if $L$ is strictly positive, then $\cN_1(L)=\{0\}$ and therefore $\cV(L)=\cX$.

Some basic facts on these sets are collected in the next proposition. 
\begin{proposition}\label{propcorevariety}
\begin{itemize}
\item[\em (i)]\ $\cN_{j-1}(L)\subseteq \cN_{j} (L)$ and $\cV_j(L)\subseteq \cV_{j-1} (L)$ for $j\in \N.$ 
\item[\em (ii)]~ If $\mu$ is a representing measure of $L$, then\,  ${\supp}\, \mu \subseteq \cV (L).$
\item[\em (iii)]~ There exists a $k\in \N_0$ such that 
\begin{align}\label{termiantes}
\cX =\cV_0(L) \supsetneqq \cV_1(L)\supsetneqq ... \supsetneqq \cV_k(L)= \cV_{k+j}(L)=\cV(L),\quad j\in \N.
\end{align}
\end{itemize}
\end{proposition}

For a moment functional $L$ let $k(L)$ denote the number $k$ in equation (\ref{termiantes}), that is, $k(L)$ is  the smallest $k$ such that $\cV_k(L)=\cV(L)$. It is an interesting problem to characterize the set of moment sequences $L$ for which $k(L)$ is a fixed number $k$. Further, it is likely to expect that for each number $n\in \N$ there exists a moment functional $L$ such that $n=k(L)$.  

The importance of the core variety  stems from the following three theorems. The first theorem is due to Blekherman and Fialkow, the two other are from my recent joint paper with Ph. di Dio.

\begin{theorem}\label{coreexistence}
A linear functional\, $L\neq 0$ on $E$ is a  moment functional if and only if $L(e)\geq 0$ and 
the core variety $\cV(L)$ is not empty. 
\end{theorem}
\begin{proof}
%\cite[Theorem 18.22]{mp}
 {[MP, Theorem~18.22]}.
$\hfill \qed$
\end{proof}

For a moment functional $L$ on $E$  we define the set of possible atoms:
\begin{align}\label{definitionW(L)}
\cW(L) :=\{ x\in \cX:\text{$\mu(\{x\})>0$ for some representing measure $\mu$ of $L$}\}.
\end{align}

The second theorem  says that the core variety is just  
the set of possible atoms. In particular, it implies that $\cW(L)$ is a \emph{closed} subset  of $\cX$.

\begin{theorem}\label{corevarsuppmu}
Let $L$ be a truncated moment functional on $E$. Then \begin{align}\label{wequalv}
\cW(L)=\cV(L).
\end{align}
Each representing measure $\mu$ of $L$ is supported on $\cV(L)$. For each point $x\in \cV(L)$ there is a finitely atomic representing measure $\mu$ of $L$ which has $x$ as an atom. 
 \end{theorem}
\begin{proof}
%\cite[Theorem 18.21]{mp}. 
 {[MP, Theorem~18.21]}. 
$\hfill\qed$
\end{proof}

A moment functional is  called \emph{determinate} if it has a unique representing measure.

\begin{theorem}\label{detminacysizeofW_+} For any  
  moment functional $L$ on $E$ the 
  following are equivalent:
\begin{itemize}
\item[\em (i)]~   $L$ is  determinate.
\item[\em (ii)]~\,  $|\cV(L)|\leq \dim (E\lceil \cV(L)).$
%\item[\em (iii)]~\,  There is  a representing measure $\mu$ of $L$ such that $|\supp \, \mu|> \dim ({\sA}\lceil \cV(L)).$
\end{itemize}
\end{theorem}
\begin{proof}
%\cite[Theorem 18.23]{mp}.
 {[MP, Theorem~18.23]}.
 $\hfill\qed$
\end{proof}

 Thus, in particular,\, $L$ is not determinate if\,  $|\cV(L)|>\dim E.$ 
\smallskip

 We close this section with three  illustrating  examples.
 \begin{example} (\textit{$\cV(L)\neq \emptyset$\, and $L$ is not a truncated moment functional})\\
Let $d=1$, ${\sN}=\{0,2\}$. Then $E=\{a+bx^2: a,b\in \R\}$.  Define a linear functional $L$ on $E$ by\, $L(a+bx^2)=-a$. Clearly, $L$ is not a  moment functional, because $L(1)=-1$. 

 Then $E_+=\{a+bx^2: a\geq 0,b\geq 0 \}$, so  $\cN_1(L)= \R_+ \cdot x^2$ and $\cV_1(L)=\{ 0\}$. Hence 
 \[\{f\in E: f(x)\geq 0~~\text{ for }~~ x\in \cV_1(L)\}=\{a+bx^2: a\geq 0,b\in \R \}.\] Therefore, $\cN_2(L)=\R\cdot x^2$ and $\cV_2(L)=\{0\}$, so that $\cV(L)=\{0\}\neq \emptyset.$ 
$\hfill \circ$ \end{example}
\begin{example}\label{wlneqvl} (\textit{A truncated moment functional with $\cV(L)=\cV_2(L)\neq \cV_1(L)$} ) \\
Let $d=1$ and ${\sN}=\{0,2,4,5,6,7,8\}$. Then $E={\Lin} \, \{ 1,x^2,x^4,x^5,x^6,x^7,x^8\}$. We fix a real number $\alpha >1$ and define $\mu=\delta_{-1}+\delta_1+\delta_\alpha.$ For the corresponding moment functional\, $L\equiv L^\mu=l_{-1}+l_1+l_\alpha$\,  on $E$ we have
\begin{align}\label{counterexamplexv+l}
\cW(L)=\cV(L)=\cV_2(L)=\{1, -1,\alpha\} \subset \{1,-1,\alpha,-\alpha\}=\cV_1(L).
\end{align}

Let us prove (\ref{counterexamplexv+l}).
Put $p(x):=(x^2-1)^2(x^2-\alpha^2)^2$. Clearly, $p\in E$, $L(p)=0$ and $p\in {\Pos}(\R)$, so  $p\in \cN_1(L)$. Conversely, let $f\in \cN_1(L)$. Then $L(f)=0$ implies that $f(\pm 1)=f(\alpha)=0$. Since $f\in {\Pos}(\R)$, the zeros\, $ 1, -1, \alpha$ have even multiplicities. Hence  $f(x)=(x-1)^2(x+1)^2(x-\alpha)^2(ax^2+bx+c)$ with $a,b,c\in \R$.  Since $x$ and $x^3$ are not in ${\sA}$, the  coefficients of $x, x^3$ vanish. This yields\,  $ax^2+bx+c=a(x+\alpha)^2$ with  $a\geq 0$. Thus $\cN_1(L)= \R_+ {\cdot} p$. Hence\, $\cV_1(L)=\cZ(p)=\{1,-1,\alpha,-\alpha\}.$

Now we set $q(x)=x^4(x^2-1)(\alpha-x)$. Then $q\in E$. Since $q(\pm 1)=q(\alpha)=0$ and $q(-\alpha)=\alpha^4(\alpha^2-1)2\alpha>0,$ we have $q(x)\geq 0$ on $\cV_1(L)$ and $L(q)=0$. Thus,  $q\in \cN_2(L)$ and hence\, $\cV_2(L)\subseteq \cV_1(L)\cap \cZ(q)=\{ 1,-1,\alpha\}.$

Since $1,-1,\alpha$ are atoms of $\mu$,   $\{1,-1, \alpha\}\subseteq \cW(L)$. Now (\ref{counterexamplexv+l}) follows.  The moment functional $L$   is determinate. $\hfill \circ$
\end{example}

\begin{example} (\textit{An example based on the Robinson polynomial})

Let $\cH_{3,6}$ denote the homogeneous polynomials in $3$ variables of degree $6$ and  $\dP^2(\R)$ the $2$-dimensional projective space. 
The   
polynomial $R\in \cH_{3,6}$, defined by  
\begin{align*}
R(x,y,z):=&x^6+y^6+z^6+3x^2y^2z^2 -x^4y^2 -x^4z^2 - x^2y^4- y^4z^2-x^2z^4-y^2z^4,
\end{align*}
is called the \emph{Robinson polynomial}.
It has a number of interesting properties:
\begin{proposition}\label{robins}
\begin{itemize}
\item[\em (i)]
$R(x,y,z)\geq 0$ for  $(x,y,z)\in\R^3$.
\item[\em (ii)]~ $R$   is not a sum of squares in $\R[x,y,z]$. 
\item[\em (iii)]~ $R$ has exactly $10$ zeros $t_1,\dots,t_{10}$, given by (\ref{robzero1})--(\ref{robzero2}), in  $\dP^2(\R)$. 
\item[\em (iv)]~ If $p\in \cH_{3,6}$ vanishes on $t_1,\dots,t_{10}$, then $p=\lambda R$ for $\lambda\in \R$.
\end{itemize}
\end{proposition}
\begin{proof}
%\cite[Proposition 19.19]{mp}.
 {[MP, Proposition~19.19]}.
 $\hfill\qed$
\end{proof}

Assertion (i) follows from the identity
\begin{align*}
(x^2+y^2)R= x^2z^2(x^2-z^2)^2+y^2z^2(y^2-z^2)^2+ (x^2-y^2)^2(x^2+y^2-z^2)^2.
\end{align*}
The zeros of $R$ in the projective space $\dP^2(\R)$ are:
\begin{align}\label{robzero1}
&t_1=(1,1,1), t_2=(1,1,-1), t_3=(1,-1,1), t_4=(1,-1,-1), t_5=(1,0,1),\\& t_6=(1,0,-1), t_7=(1,1,0), t_8=(1,-1,0), t_9=(0,1,1), t_{10}=(0,1,-1).\label{robzero2}
\end{align}
Fix $t_{0}\in\dP^2(\R)$ such that $t_{0}\neq  t_j$, $j=1,\dots,10$. Then $R(t_{0})\neq 0$. Put 
\begin{align*} 
 \nu=\sum\nolimits_{j=1}^{10} m_j\delta_{t_j},\quad \mu=\nu+m_0\delta_{t_0},\quad\text{ where }~~ ~m_j> 0
\end{align*}
 Let $L^\nu$ and $L^\mu$ be the  moment functionals given by the atomic measures $\mu$ and $\nu$, respectively, on the vector space $\cH_{3,6}$ of continuous functions   on the compact topological space $ \dP^2(\R)$. 
From Proposition \ref{robins}(iv) we obtain:

$\bullet$~ $L^\mu$ is strictly positive.
Hence $\cV(L^\mu)=  \dP^2(\R)$ and each  $x\in\dP^2(\R)$ is atom of a representing measure.

$\bullet$~ For $L^\nu$  only the points $t_1,\dots,t_{10}$ are atoms of some representing measure. 
It can be shown that $\nu$ is determinate and $\cV(L)=\{t_1,\dots,t_{10}\}$.

The measure $\mu$ differs from $\nu$ by a single atom, but the corresponding core varieties are opposite extreme cases, one is the whole space and the other is discrete.$\hfill \circ$

\end{example}

\section{The moment cone}\label{momentconedef}%%%
%%%%
From now on, ${\sA} := \{a_1,\dots,a_m\}$ denotes a fixed  basis of the  vector space $E$.

There is a one-to-one correspondence between linear functionals $L$ on $E$ and vectors $s=(s_1,\dots,a_m)\in \R^m$ given by $L(a_j)=s_j, j=1,\dots,m$.
For $s\in \R^m$  the corresponding functional is the Riesz functional $L_s$ of $s$.

Let $\cM_+(E)$ denote the set of Radon  measures on $\cX$ for which all functions of $E$ are $\mu$-integrable. For $\mu\in \cM_+(E)$,
\begin{align}\label{defmomentf}
L_s^\mu (f)=\int f(x) d\mu(x),\quad f\in E, 
\end{align}
is equivalent to
\begin{align}\label{moments}
s_j(\mu)=\int a_j(x) d\mu(x),\quad j=1,\dots,m.
\end{align}
Recall that the functional $L_s$ in (\ref{defmomentf}) is called the \emph{moment functional} of $\mu$ and the vector $s(\mu)=(s_1(\mu=,\dots,s_m(\mu))$ given by (\ref{moments}) is  the \emph{moment vector} of $\mu$. Thus, $s\in \R^m$ is  moment vector if and only if $L_s$ is a moment functional.

\begin{definition}
 The \textbf{moment cone} $\cS$ is the cone of all \emph{moment sequences}, that is,
\[
 \cS
 := \{ (s_1(\mu),\dots,s_m(\mu)):~ \mu\in \cM(E)~ \},
\]
and $\cL$ denotes the cone of all \emph{moment functionals} in $E^*$, that is,
\[
\cL
:=\{ L^\mu\in E^*: L^\mu=\int f d\mu,f\in E, ~~\textit{where}~~ \mu\in \cM_+(E)\}.
\]
\end{definition}

For $x\in \cX$, the moment vector of the delta measure $\delta_x$ is
\[
 s_{\sA}(x)
 :=(a_1(x),\dots,a_m(x)).
\]
Since each moment functional has a finitely atomic representing measure by Theorem \ref{richter}, each vector of $\cS$ is a nonnegative linear combination of vectors $s_{\sA}(x)$ for $x\in \cX$, that is, $\cS$ is the convex cone in $\R^m$ generated by the vectors $s_{\sA}(x), x\in \cX$.

The map $s\mapsto L_s$ is a bijection of $\cS$ and $\cL$. Further, we have
\begin{align*}
\R^m=\cS-\cS~~~~\text{ and }~~~E^*=\cL-\cL.
\end{align*}

In general, both cones $\cS$ and $\cL$ are not closed, as the following example shows.

\begin{example}
Let $\sA=\{1,x,x^2\}$ on $\cX=\R$. For $\mu_n=n^{-2}\delta_n$ we have 
$s(\mu_n)=(n^{-2}, n^{-1},1)$. But  $s:=\lim_n s(\mu_n)=(0,0,1)$ is not in $\cS$.
$\hfill \circ$
\end{example}

For a cone $C$ in a real vector space $V$, its \emph{dual cone} is defined by
\[
C^\wedge=\{ \varphi\in V^*: \varphi(c)\geq 0~~~\text{ for }~~c\in C\}.
\]
Let ${\ov{\cL}}$\, denote the closure of the cone ${\cL}$ in  $E^*$.

\begin{proposition}\label{cLcPwedge} ~
${\cL}\subseteq (E_+)^{\wedge} ={\ov{\cL}}$~  and~ ${\cL}^{\wedge}=E_+=(E_+)^{{\smallwedge}{\smallwedge}}.$\\
If the space $\cX$ is compact, then the cone ${\cL}$  is closed in the norm topology of the dual space $E^*$ and we have ${\cL}=(E_+)^{\smallwedge}$.
\end{proposition}
\begin{proof}
%\cite[Propositions 1.26 and 1.27]{mp}.
 {[MP, Propositions~1.26,~1.27]}.
 $\hfill\qed$
\end{proof}

Finally, we turn to the supporting hyperplanes  of the cone $\cL$ of moment functionals.
Suppose  $L\in \cL$. Recall from (\ref{N_KLE}) that   $\cN_1(L)=\{ f\in E_+: L(f)=0~\}.$

The next proposition shows that the nonzero elements of $\cN_1(L)$  correspond to the supporting hyperplanes of the cone $\cL$.

\begin{proposition}\label{bundardayinnermc} 
\begin{itemize}
\item[\em (i)]\ Let  $p\in \cN_1(L), p\neq 0$. Then\, $\varphi_p(L')=L'(p), L'\in E^*$,\, defines a supporting functional\, $\varphi_p$\, of the cone ${\cL}$ at $L$. Each supporting functional of $\cL$ at $L$ is of this form.
\item[\em (ii)]~ $L$ is a boundary point of the cone\, $\cL$\, if and only if\, $\cN_1(L)\neq\{0\}.$ 
\item[\em (iii)]~ $L$ is an inner point of the cone\, $\cL$\, if and only if\,   $\cN_1(L)=\{0\}.$
\end{itemize}
\end{proposition}
\begin{proof}
%\cite[Proposition 1.42]{mp}.
 {[MP,Proposition~1.42]}.
 $\hfill\qed$
\end{proof}

An \emph{exposed face} of a cone $C$ in a finite-dimensional real vector space is a subcone of the form $F=\{f\in C:\varphi(f)=0\}$ for
some functional $\varphi\in C^{\smallwedge}$. 

Since  $\cL^{\smallwedge}=E_+$,  each  $\varphi\in \cL^{\smallwedge}$ is of the form 
$\varphi_p(L')=L'(p), L'\in E^*,$ for some $p\in E_+$. Hence the exposed faces of the cone $\cL$ in $E^*$ are precisely the sets 
the sets
\begin{align}\label{f-pdefi}
 F_p
 :=\{L'\in \cL: \varphi_p(L')\equiv L'(p) =0\},\quad\text{ where }~~ p\in E_+.
\end{align}

Let $L\in\cL$. Since $\cL\subseteq (E_+)^{\smallwedge}$,  $N_1(L)$ is an  exposed face of the cone $E_+$. 
If $\cX$ is compact, then $(E_+)^{\smallwedge}={\cL}$ by  Proposition \ref{cLcPwedge}, so  each exposed face of $E_+$
is of this form. 
Thus,  in this case the subcones $\cN_1(L)$ are precisely the  exposed faces of $E_+$.

Recall that for inner points of $\cL$ we have $\cW(L)=\cX$ (by Proposition \ref{existencemfstrict}) and hence $\cW(L)=\cV_1(L)$. In general, $\cW(L)\neq \cV_1(L)$, as shown by  Example \ref{wlneqvl}.

The next result characterizes those boundary points for which $\cW(L)=\cV_1(L).$

\begin{proposition}\label{thm:WV+cases}
Let $L$ be a boundary point of $ {\cL}$. Then  $\cW(L)=\cV_(L) $ if and only if  $L$ lies in the relative interior of an exposed face of the cone $\cL$.
\end{proposition}
\begin{proof}
%\cite[Theorem 1.45]{mp}.
 {[MP, Theorem~1.45]}.
 $\hfill\qed $
\end{proof}

\secret{

\section{Carath\'eodory  numbers}
%\end{document}
By Richter's theorem, each  sequence $s\in \cS$ has a $k$-atomic representing measure, with $k\leq \dim\, E$. 
It is natural to ask for smallest possible number $k$.

\begin{definition}\label{defcaratheorynumnber} 
Let  $s\in \cS$. The  \emph{Carath\'eodory number}\,  ${\sC}(s)={\sC}(L_s)$ is the smallest number $k\in \N_0$ such that $s$, equivalently $L_s$, has a $k$-atomic representing measure.

The \emph{Carath\'eodory number}\, ${\sC}(\cS)$ is the largest\, ${\sC}(s)$ for $s\in \cS$.
\end{definition}

Thus, $\sC(\cS)$  the smallest  number\, $n\in \N$ such that each  sequence  $s\in \cS$ has a $k$-atomic representing measure, where $k\leq n$. We set ${\sC}(s)=0$  for $s=0$.
Then
\begin{align*}
  {\sC}(\cS) ={\sup} \, \{ {\sC}(s):\, s \in {\cS}_{\sA}\, \}  \leq m= {\dim}\, E.
\end{align*}

\begin{example} ({\it Carath\'eodory number  in dimension one on $\R[x]_m$})\\
Let $d=1, m\in \N,$  ${\sA}=\{1,x,\dots,x^m\}$, and  $\cX:=[a,b],$ where
 $a,b\in \R, a<b.$ Then $E={\Lin} \, {\sA}=\R[x]_m$.  For the Carath\'eodory number we have
 \begin{align*}
 {\sC}(\cS)=n+1\quad {\rm if}~~ m=2n~~{\rm or}~~  m=2n+1,~~n\in \N.
 \end{align*} That is, if  $\lfloor \frac{m}{2}\rfloor$ denotes the largest integer $k$ satisfying $k\leq \frac{m}{2}$, we can write
\begin{align}\label{caronecase}
{\sC}(\cS)=\left\lfloor \frac{m}{2}\right\rfloor +1.
\end{align}

\end{example}
}

%

%\backmatter          

%\addcontentsline{toc}{chapter}{Symbol index}
%\printindex[sym]
%\addcontentsline{toc}{chapter}{Index}
%%%%%%%%%%%%%%%%%%%%%%%%%%%%%%%%%%%%%%%%%%%%%%%%%%%%%%%%%%%%%%%%%%%%%%

%\printindex

\end{document}